\documentclass[eqthmnum,nocolour]{jt-calcs}
\usepackage{mathrsfs}
\usepackage{booktabs}
\usepackage{lscape}
\usepackage{tikz-cd}

\usepackage[backend=bibtex,style=alphabetic,sorting=nyt]{biblatex}
\bibliography{bibliography}

\usepackage{pgf}
\usepackage{tikz}
\usetikzlibrary{matrix,arrows}

\newcommand{\bl}[1]{\ensuremath{\mathscr{#1}}}

\newcommand{\sbpair}[2]{\ensuremath{\bigl(\begin{smallmatrix} #1 \\ #2 \end{smallmatrix}\bigr)}}
\newcommand{\ssymb}[2]{\ensuremath{\bigl[\begin{smallmatrix} #1 \\ #2 \end{smallmatrix}\bigr]}}

\newcommand{\symb}[2]{\ensuremath{\begin{bmatrix} #1 \\ #2 \end{bmatrix}}}

\DeclareMathOperator{\sss}{ss}

\DeclareMathOperator{\ind}{ind}
\DeclareMathOperator{\res}{res}
\DeclareMathOperator{\IC}{IC}

\DeclareMathOperator{\std}{std}

\crefname{enumi}{}{}
\crefformat{enumi}{#2#1#3}

\newcommand{\der}{\ensuremath{\mathrm{der}}}
\newcommand{\ex}{\ensuremath{\mathrm{ex}}}

\newcommand{\Ql}{\ensuremath{\overline{\mathbb{Q}}_{\ell}}}

\newcommand{\bPhi}{\ensuremath{\mathbf{\Phi}}}
\newcommand{\bOmega}{\ensuremath{\mathbf{\Omega}}}

\newcommand{\Centu}[1]{\Cent_{\mathrm{U}}(#1)}
\renewcommand{\epsilon}{\varepsilon}

\theoremstyle{nonumberplain}
\theorembodyfont{\itshape}

\newenvironment{assumption}{
\begin{center}

\begin{tabular}{!{\vrule width 0.75pt}p{15cm}!{\vrule width 0.75pt}}
\noalign{\hrule height 0.75pt}
\cellcolor[gray]{0.85}
\ignorespaces
}
{
\\[.5ex]
\noalign{\hrule height 0.75pt}
\end{tabular}
\end{center}
\ignorespacesafterend
}

\title{Multiplicities in GGGRs for Classical Type Groups with Connected Centre I}
\author{Jay Taylor}

\begin{document}
\start

\begin{abstract}
Assume $\bG$ is a connected reductive algebraic group defined over $\overline{\mathbb{F}_p}$ such that $p$ is good prime for $\bG$. Furthermore we assume that $Z(\bG)$ is connected and $\bG/Z(\bG)$ is simple of classical type. Let $F$ be a Frobenius endomorphism of $\bG$ admitting an $\mathbb{F}_q$-rational structure $G = \bG^F$. This paper is one of a series whose overall goal is to compute explicitly the multiplicity $\langle D_G(\Gamma_u),\chi\rangle$ where: $\chi$ is an irreducible character of $G$, $D_G(\Gamma_u)$ is the Alvis--Curtis dual of a generalised Gelfand--Graev representation of $G$ and $u \in G$ is contained in the unipotent support of $\chi$. In this paper we complete the first step towards this goal. Namely we explicitly compute, under some restrictions on $q$, the scalars relating the characteristic functions of character sheaves of $\bG$ to the almost characters of $G$ whenever the support of the character sheaf contains a unipotent element. We achieve this by adapting a method of Lusztig who answered this question when $\bG$ is a special orthogonal group $\SO_{2n+1}(\mathbb{K})$. Consequently the main result of this paper is due to Lusztig when $\bG = \SO_{2n+1}(\mathbb{K})$.
\end{abstract}

\tableofcontents
\newpage

%
\section{Introduction}
\begin{pa}
Throughout this article $\bG$ will be a connected reductive algebraic group defined over an algebraic closure $\mathbb{K} = \overline{\mathbb{F}}_p$ of the finite field $\mathbb{F}_p$ where $p>0$ is a good prime for $\bG$. We will assume that $F : \bG \to \bG$ is a Frobenius endomorphism admitting an $\mathbb{F}_q$-rational structure $G := \bG^F$. In \cite{kawanaka:1986:GGGRs-exceptional} Kawanaka associated to every rational unipotent element $u \in G$ a representation $\Gamma_u$ of $G$ called a generalised Gelfand--Graev representation (GGGR). These representations are such that $\Gamma_u$ is the regular representation when $u$ is the identity and a Gelfand--Graev representation when $u$ is a regular unipotent element. In both of these cases the decomposition of $\Gamma_u$ into its irreducible constituents is well understood. However for an arbitrary unipotent element the explicit decomposition of $\Gamma_u$ is unknown.
\end{pa}

\begin{pa}
We will say that $(\star)$ holds if the following two conditions are satisfied.
\begin{enumerate}
	\item[($\star_p$)] $p$ is large enough so that $\exp : \mathfrak{g} \to \bG$ and $\log : \bG \to \lie{g}$ define inverse bijections between the varieties of nilpotent and unipotent elements, where $\mathfrak{g}$ is the Lie algebra of $\bG$.
	\item[($\star_q$)] $q\geqslant q_0(\bG)$ where $q_0(\bG)$ is a constant depending only on the root system of $\bG$, (see \cite[Theorem 1.14]{lusztig:1990:green-functions-and-character-sheaves}).
\end{enumerate}
Let $D_G$ denote the Alvis--Curtis duality map and recall that for each irreducible character $\chi$ there exists a sign $\epsilon_{\chi} \in \{\pm1\}$ such that $\epsilon_{\chi}D_G(\chi)$ is again an irreducible character of $G$. In \cite{lusztig:1992:a-unipotent-support}, assuming $(\star)$ holds, Lusztig has associated to every irreducible character $\chi$ a unique $F$-stable unipotent conjugacy class $\mathcal{O}_{\chi}$ of $\bG$ called the \emph{unipotent support} of $\chi$. He has also shown that if $u \in \mathcal{O}_{\chi}^F$ then the multiplicity of $\chi$ in $\epsilon_{\chi}D_G(\Gamma_u)$ is a ``small'' integer in the sense that it is bounded independently of $q$. Specficially we have
\begin{equation*}
\langle \epsilon_{\chi}D_{\bG}(\Gamma_u),\chi\rangle \leqslant |A_{\bG}(u)|/n_{\chi}
\end{equation*}
where $A_{\bG}(u)$ is the component group of the centraliser $C_{\bG}(u)/C_{\bG}(u)^{\circ}$ and $n_{\chi}$ is the generic denominator of $\chi$, (see \cite[2.8(b)]{geck:1999:character-sheaves-and-GGGRs}).

In the extreme case that $|A_{\bG}(u)| = n_{\chi}$ the multiplicity is either $0$ or $1$ and one can show that $\langle \epsilon_{\chi}D_G(\Gamma_u),\chi\rangle = 1$ for a unique GGGR associated to the unipotent support of $\chi$. This fact has already had interesting applications to the representation theory of $G$, (see \cite{geck:1999:character-sheaves-and-GGGRs}, \cite{geck-hezard:2008:unipotent-support}, \cite{taylor:2011:on-unipotent-supports}). However in \cite{taylor:2011:on-unipotent-supports} it became clear that more explicit information concerning the multiplicities is required even in this extreme case, (see \cite[Remark 5.6]{taylor:2011:on-unipotent-supports}). Specficially one would like to know which irreducible characters occur with non-zero multiplicity in the same GGGR.
\end{pa}

\begin{pa}
It is the purpose of this and the following articles to answer this question. Our line of attack is as follows. In \cite{lusztig:1992:a-unipotent-support} Lusztig has explicitly decomposed the GGGRs as a sum of characteristic functions of $F$-stable unipotently supported character sheaves. Therefore the problem of understanding multiplicities in GGGRs can be translated into a problem concerning the multiplicities in characteristic functions of character sheaves. Lusztig has conjectured that the characteristic functions of character sheaves coincide with the almost characters of $G$ up to multiplication by roots of unity. Furthermore in \cite{lusztig:1984:characters-of-reductive-groups} Lusztig has explicitly decomposed the almost characters as a sum of irreducible characters.

From this we see that there are two problems to overcome. The first problem, which is the focus of this article, is to describe explicitly the scalars relating the characteristic functions of unipotently supported character sheaves and the corresponding almost characters of $G$. The second problem is then to carry out the translations as described above; this will be the focus of the following articles. With this in mind the main result of this article is as follows.
\end{pa}

\begin{thm}\label{thm:main-theorem}
Assume that $Z(\bG)$ is connected and $\bG/Z(\bG)$ is simple of type $\B_n$, $\C_n$ or $\D_n$, (recall also that $p > 2$). Let us further assume that $F$ is a split Frobenius endomorphism and that $q \equiv 1 \pmod{4}$ when $\bG/Z(\bG)$ is of type $\C_n$ or $\D_n$. If $A \in \widehat{\bG}^F$ is an $F$-stable unipotently supported character sheaf then there exists a canonical isomorphism $\phi : F^*A \to A$ such that the characteristic function $(-1)^n\chi_{A,\phi}$ determined by $\phi$ is the corresponding almost character of $G$.
\end{thm}

\begin{pa}
The above is essentially part of the remark made in \cite[7.11]{lusztig:1986:on-the-character-values}. However, in comparison to the type $\B_n$ case, there are more technical difficulties to overcome in the type $\C_n$ and $\D_n$ cases. In general this comes from having to deal with induction of characters from maximal rank reflection subgroups of type $\B_n$ and $\D_n$ Weyl groups, (see \cref{cor:multiplicity-1}). Had we simply wanted to prove the above theorem this article would be much shorter. However our end goal is to prove a version of the above theorem without the restriction on $q$ or the assumption that $F$ is split, although here we fall short of this. The problem with proving the above theorem without these assumptions is the following. One must deal with almost characters contained in the subspace spanned by a Lusztig series $\mathcal{E}(G,s)$ where $s$ is a semisimple element of the dual group whose centraliser has a twisted type $\D$ factor. To apply the line of argument followed in this article one needs to explicitly decompose the Deligne-Lusztig induction of almost characters corresponding to cuspidal character sheaves. However in the case of twisted type $\D$ this does not seem to be treated in the literature in the level of detail which we need here. The author hopes to address this issue in the sequel to this article.
\end{pa}

\begin{acknowledgments}
The author gratefully acknowledges financial support from ERC Advanced Grant 291512 and would like to thank Geordie Williamson and Sebastian Herpel for useful conversations.
\end{acknowledgments}
%
\section{Conventions}\label{sec:conventions}
\begin{pa}\label{pa:conventions}
Given a variety $\bX$ over $\mathbb{K}$ we denote by $\mathscr{D}\bX = \mathscr{D}_c^b(\bX,\overline{\mathbb{Q}}_{\ell})$ the bounded derived category of $\overline{\mathbb{Q}}_{\ell}$-constructible sheaves on $\bX$. Furthermore we denote by $\mathscr{M}\bX$ the full subcategory of $\mathscr{D}\bX$ whose objects are the perverse sheaves on $\bX$. Assume $\bH$ is a connected algebraic group acting on $\bX$ then we take the statement $A \in \mathscr{D}\bX$ is $\bH$-equivariant to be as defined in \cite[\S1.9]{lusztig:1985:character-sheaves}. If $\bX$ is itself a connected algebraic group then, unless otherwise explicitly stated, we take $\bX$-equivariance to be with respect to the natural conjugation action of $\bX$ on itself.

As any element $A \in \mathscr{D}\bX$ is a complex of sheaves we may construct for any $i \in \mathbb{Z}$ the $i$th cohomology sheaf $\mathscr{H}^iA$ with respect to the maps in the complex. Given any element $x \in \bX$ we then denote by $\mathscr{H}^i_xA$ the corresponding stalk of $\mathscr{H}^iA$. For any $A \in \mathscr{D}\bX$ we call $\supp(A) := \{x \in \bX \mid \mathscr{H}_x^iA \neq 0$ for some $i \in \mathbb{Z}\}$ the support of $A$. If $\varphi : \bX \to \bY$ is a morphism then we denote by $\varphi^* : \mathscr{D}\bY \to \mathscr{D}\bX$ the inverse image functor, $\varphi_* : \mathscr{D}\bX \to \mathscr{D}\bY$ the right derived direct image functor and $\varphi_! : \mathscr{D}\bX \to \mathscr{D}\bY$ the right derived direct image functor with compact support. If $\varphi$ is smooth with connected fibres of dimension $d$ then we denote by $\tilde{\varphi}$ the shifted inverse image $\varphi^*[d]$, (c.f.\ \cite[\S1.7]{lusztig:1985:character-sheaves}).
\end{pa}

\begin{pa}
Assume $\bX \subseteq \bY$ is a subvariety then for any $A \in \mathscr{D}\bY$ we denote by $A|_{\bX}$ the complex $i^*A \in \mathscr{D}\bX$ where $i : \bX \hookrightarrow \bY$ is the inclusion map, (we call this the restriction of $A$ to $\bX$). Assume now that $\bX$ is a smooth open dense subset of its closure $\overline{\bX}$ and that $\mathscr{L}$ is a local system on $\bX$, (by which we mean a locally constant $\overline{\mathbb{Q}}_{\ell}$-constructible sheaf with finite dimensional stalks), then $\mathscr{L}[\dim\bX] \in \mathscr{M}\bX$ is a perverse sheaf on $\bX$. We denote by $\IC(\overline{\bX},\mathscr{L})[\dim\bX] \in \mathscr{M}\overline{\bX}$ the intersection cohomology complex determined by $\mathscr{L}$ which is an element of $\mathscr{M}\overline{\bX}$ extending $\mathscr{L}[\dim\bX]$, i.e.\ we have $\IC(\overline{\bX},\mathscr{L})|_{\bX} \cong \mathscr{L}$. We may freely consider this as an element of $\mathscr{M}\bY$ by extending $\IC(\overline{\bX},\mathscr{L})[\dim\bX]$ to $\bY$ by 0 on $\bY - \overline{\bX}$ and we will do so without explicit mention.
\end{pa}

\begin{pa}\label{pa:direct-image-commutative}
For convenience we recall here the following base change isomorphism, (see \cite[(1.7.5)]{lusztig:1985:character-sheaves-I}). Assume $\bX$, $\bY$, $\bZ$ and $\bW$ are varieties over $\mathbb{K}$ and that we have a commutative diagram
\begin{equation*}
\begin{tikzcd}
\bX \arrow{d}{f}\arrow{r}{\phi} & \bY \arrow{d}{g}\\
\bZ \arrow{r}{\psi} & \bW
\end{tikzcd}
\end{equation*}
such that $\phi$ and $\psi$ are smooth maps with connected fibres of common dimension. Then we have an isomorphism $f_!\circ\tilde{\phi} = \tilde{\psi}\circ g_!$ of functors $\mathscr{D}\bY \to \mathscr{D}\bZ$. In particular if every fibre of $\phi$ and $\psi$ is simply a point then we have $f_!\circ\phi^* = \psi^*\circ g_!$.
\end{pa}

\begin{pa}\label{pa:conventions-finite-groups}
If $G$ is a finite group we will denote by $\Ql G$ the $\Ql$-group algebra of $G$. We will take the statement ``$M$ is a $G$-module'' to mean $M$ is a finite dimensional module for the group algebra $\Ql G$, we will consider this to be either a left or a right module as appropriate. If $E$ and $E'$ are two $G$-modules then we will denote by $\Hom_G(E,E')$ the space of all $\Ql G$-module homomorphisms $f : E \to E'$. We will denote by $\Irr(G)$ a set of representatives from the isomorphism classes of simple $G$-modules; however we will also use $\Irr(G)$ to denote the corresponding set of irreducible characters.
\end{pa}

\begin{pa}\label{pa:semidirect-products}
Assume now that $\phi : G \to G$ is an automorphism and let us denote by $\widetilde{G}$ the semidirect product $G \rtimes \langle \phi \rangle$ where $\langle \phi \rangle \leqslant \Aut(G)$ is the cyclic subgroup generated by the automorphism $\phi$. Let $H \leqslant G$ be a subgroup and $g \in G$ be such that $\phi(gHg^{-1}) = H$ then $H$ is a normal subgroup of the group $\widetilde{H} = H\langle(\phi(g),\phi)\rangle$ and we denote by $H.\phi g$ the set $\{(h\phi(g),\phi) \in \widetilde{G} \mid h \in H\}$, (note that we simply write $H$ for its image $H\times \{1\}$ in $\widetilde{G}$). Furthermore we denote by $\Cent(H.\phi g)$ the $\Ql$-vector space of functions $f : H.\phi g \to \Ql$ which are invariant under conjugation by $H$. We can define on $\Cent(H.\phi g)$ an inner product $\langle -,- \rangle_{H.\phi g} : \Cent(H.\phi g) \times \Cent(H.\phi g) \to \Ql$ by setting
\begin{equation*}
\langle f, f' \rangle_{H.\phi g} = \frac{1}{|H|}\sum_{h \in H}f(h\phi(g),\phi)\overline{f'(h\phi(g),\phi)},
\end{equation*}
where $\overline{\phantom{x}} : \Ql \to \Ql$ is a fixed automorphism such that $\overline{\omega} = \omega^{-1}$ for every root of unity $\omega \in \Ql^{\times}$. We may also define a $\Ql$-linear map $\Ind_{H.\phi g}^{G.\phi} : \Cent(H.\phi g) \to \Cent(G.\phi)$ as in \cite[(1.2)]{bonnafe:2006:sln}.
\end{pa}

\begin{pa}\label{pa:conventions-coset-identification}
As the composition $\phi\ad_g$ is an automorphism of $G$ we may form, as before, the semidirect product $G \rtimes\langle\phi\ad_g\rangle$. The restriction of $\phi\ad_g$ to $H$ is also an automorphism of $H$ so we have $H \rtimes \langle \phi\ad_g \rangle$ is naturally a subgroup of $G \rtimes \langle \phi\ad_g \rangle$. We may now define a surjective homomorphism of groups $\psi_g : G \rtimes \langle \phi\ad_g\rangle \to \widetilde{G}$ given by
\begin{equation*}
\psi_g(h,\phi^i\ad_{\phi^{1-i}(g)\cdots\phi^{-1}(g)g}) = (h\phi(g)\phi^2(g)\cdots\phi^i(g),\phi^i)
\end{equation*}
for any $i \in \mathbb{N}$. The restriction of $\psi_g$ defines a bijection $G.\phi\ad_g \to G.\phi$, (resp.\ $H.\phi\ad_g \to H.\phi g$), which respects the action of $G$, (resp.\ $H$), by conjugation. In particular $\psi_g$ induces vector space isomorphisms $\Cent(G.\phi) \to \Cent(G.\phi\ad_g)$ and $\Cent(H.\phi g) \to \Cent(H.\phi\ad_g)$ and carries the induction map $\Ind_{H.\phi g}^{G.\phi}$ to $\Ind_{H.\phi\ad_g}^{G.\phi\ad_g}$. Hence it will be sufficient to compute the induction of characters from $\widetilde{H}$ to $\widetilde{G}$.

Finally we will denote by $\Irr(H)^{\phi g} = \{\chi \in \Irr(H) \mid \chi\circ\phi\circ\ad_g = \chi\}$ the set of $\phi\ad_g$-invariant characters of $H$. If $\chi \in \Irr(H)^{\phi g}$ then we fix an arbitrary extension $\widetilde{\chi}$ of $\chi$ to $\widetilde{H}$, which exists because $\widetilde{H}/H$ is cyclic. We may then give an orthonormal basis of $\Cent(H.\phi g)$, with respect to $\langle -,- \rangle_{H.\phi g}$, by setting $\Irr(H.\phi g) = \{\widetilde{\chi}|_{H.\phi g} \mid \chi \in \Irr(H)^{\phi g}\}$, (this basis depends upon the choice of extension of $\chi$ to $\widetilde{H}$).
\end{pa}
%
\section{Character Sheaves}\label{sec:char-sheaves}
\begin{pa}\label{pa:grothendieck-group}
Let us now fix a Borel subgroup $\bB_0 \subset \bG$ and a maximal torus $\bT_0 \subset \bB_0$, (we assume for later that both $\bT_0$ and $\bB_0$ are $F$-stable). We will denote by $(W_{\bG},\mathbb{S})$ the Coxeter system of $\bG$ where $W_{\bG} = W_{\bG}(\bT_0) = N_{\bG}(\bT_0)/\bT_0$ is the Weyl group with respect to $\bT_0$ and $\mathbb{S}$ is the set of Coxeter generators determined by $\bB_0$. If $\bT$ is a torus then we denote by $\mathcal{S}(\bT)$ the set of (isomorphism classes) of local systems $\mathscr{L}$ of rank 1 on $\bT$ such that $\mathscr{L}^{\otimes m}$ is isomorphic to $\overline{\mathbb{Q}}_{\ell}^{\times}$ for some $m$ coprime to $p$, (we call such a local system \emph{tame}). The Weyl group $W_{\bG}$ acts naturally on $\bT_0$ and this in turn gives us an action on $\mathcal{S}(\bT_0)$ by $\mathscr{L} \mapsto (w^{-1})^*\mathscr{L}$; we denote the corresponding set of orbits by $\mathcal{S}(\bT_0)/W_{\bG}$. Associated to each such local system Lusztig has defined a set of character sheaves $\widehat{\bG}_{\mathscr{L}}$, (see \cite[Definition 2.10]{lusztig:1985:character-sheaves}), which depends only on the $W_{\bG}$-orbit of $\mathscr{L}$. The set of character sheaves on $\bG$ is then defined to be
\begin{equation*}
\widehat{\bG} = \bigsqcup_{\mathcal{S}(\bT_0)/W_{\bG}} \widehat{\bG}_{\mathscr{L}},
\end{equation*}
whose elements are irreducible $\bG$-equivariant objects in $\mathscr{M}\bG$. We will denote by $\mathscr{K}_0(\bG)$ the subgroup of the Grothendieck group of $\mathscr{M}\bG$ spanned by the character sheaves of $\bG$. We then define a bilinear form $(-:-) : (\mathscr{K}_0(\bG) \otimes \Ql) \times (\mathscr{K}_0(\bG) \otimes \Ql) \to \Ql$ by setting
\begin{equation*}
(A:A') = \begin{cases}
1 &\text{if }A \cong A',\\
0 &\text{otherwise}.
\end{cases}
\end{equation*}
for all $A$, $A' \in \widehat{\bG}$.
\end{pa}

\begin{pa}\label{pa:duality}
We now consider the dual group and its relationship to tori following \cite[\S1.6]{lusztig:1990:green-functions-and-character-sheaves}. Let $\bG^{\star}$ be a connected reductive algebraic group and $F^{\star}$ a Frobenius endomorphism of $\bG^{\star}$. We assume fixed an $F^{\star}$-stable Borel subgroup $\bB_0^{\star}\subset \bG^{\star}$ and an $F^{\star}$-stable maximal torus $\bT_0^{\star} \subset \bB_0^{\star}$ such that the quadruples $(\bG,\bT_0,\bB_0,F)$ and $(\bG^{\star},\bT_0^{\star},\bB_0^{\star},F^{\star})$ are in duality, (see \cite[\S8.4]{lusztig:1984:characters-of-reductive-groups}). We will denote by $(W_{\bG^{\star}},\mathbb{T})$ the Coxeter system of $\bG^{\star}$ where $W_{\bG^{\star}} = W_{\bG^{\star}}(\bT_0^{\star}) = N_{\bG^{\star}}(\bT_0^{\star})/\bT_0^{\star}$ is the Weyl group with respect to $\bT_0^{\star}$ and $\mathbb{T}$ is the set of Coxeter generators determined by $\bB_0^{\star}$. After fixing an isomorphism $\mathcal{S}(\mathbb{K}^{\times}) \to \mathbb{K}^{\times}$ we obtain for any subtorus $\bS \leqslant \bT_0$ an isomorphism $\lambda_{\bS} : \mathcal{S}(\bS) \to \bS^{\star}$ where $\bS^{\star} \leqslant \bT_0^{\star}$ is the corresponding dual torus. The isomorphism $\lambda_{\bT_0}$ is compatible with the natural actions of $W_{\bG}$ and $W_{\bG^{\star}}$ and the natural actions of $F$ and $F^{\star}$. Assume $\mathscr{L} \in \mathcal{S}(\bT_0)$ is a local system and $s \in \bT_0^{\star}$ is the image of $\mathscr{L}$ under $\lambda_{\bT_0}$ then we may denote the set $\widehat{\bG}_{\mathscr{L}}$ by $\widehat{\bG}_s$ without ambiguity and $\widehat{\bG}_s$ will depend only on the $W_{\bG^{\star}}$-orbit of $s$. 
\end{pa}

\begin{pa}\label{pa:definition-of-induction}
We will denote by $\mathcal{Z}$ the set of all pairs $(\bL,\bP)$ such that $\bP$ is a parabolic subgroup of $\bG$ and $\bL \leqslant \bP$ is a Levi complement of $\bP$. We then define $\mathcal{Z}_{\std}$ to be the subset consisting of all standard pairs $(\bL,\bP)$, i.e.\ $\bP$ contains $\bB_0$ and $\bL$ is the unique Levi complement of $\bP$ containing $\bT_0$, (note that every pair in $\mathcal{Z}$ is conjugate to a pair in $\mathcal{Z}_{\std}$). Furthermore we define $\mathcal{L}_{\std} \subset \mathcal{L}$ to be the sets of Levi subgroups obtained from $\mathcal{Z}_{\std} \subset \mathcal{Z}$ by projecting onto the first factor. We define similar sets $\mathcal{Z}_{\std}^{\star}\subset\mathcal{Z}^{\star}$ and $\mathcal{L}_{\std}^{\star} \subset \mathcal{L}^{\star}$ for $\bG^{\star}$ defined with respect to $\bT_0^{\star} \leqslant \bB_0^{\star}$ and we recall that there are natural bijections $\mathcal{Z}_{\std} \leftrightarrow \mathcal{Z}_{\std}^{\star}$ and $\mathcal{L}_{\std} \leftrightarrow \mathcal{L}_{\std}^{\star}$.

Assume now that $(\bL,\bP) \in \mathcal{Z}$ and let $A_0 \in \mathscr{M}\bL$ be an $\bL$-equivariant perverse sheaf on $\bL$. In \cite[\S4.1]{lusztig:1985:character-sheaves} Lusztig has associated to $A_0$ a complex $\ind_{\bL\subset\bP}^{\bG}(A_0) \in\mathscr{D}\bG$, which we call the induced complex. We recall the construction of this complex following \cite[\S4.1]{lusztig:1985:character-sheaves}. Consider the following diagram
\begin{equation}
\begin{tikzcd}
\bL & \hat{X} \arrow{l}[swap]{\pi}\arrow{r}{\sigma} & \tilde{X} \arrow{r}{\tau} & \bG
\end{tikzcd}
\end{equation}
where we have
\begin{gather*}
\begin{aligned}
\hat{X} &= \{(g,h) \in \bG \times \bG \mid h^{-1}gh \in \bP\} \quad&\quad \tilde{X} &= \{(g,h\bP) \in \bG \times (\bG/\bP) \mid h^{-1}gh \in \bP\}
\end{aligned}\\
\begin{aligned}
\pi(g,h) &= \hat{\pi}_{\bP}(h^{-1}gh) \quad&\quad \sigma(g,h) &= (g,h\bP) \quad&\quad \tau(g,h\bP) &=g
\end{aligned}
\end{gather*}
where $\hat{\pi}_{\bP} : \bP \to \bL$ is the canonical projection map. Since $A_0$ is $\bL$-equivariant there exists a canonical perverse sheaf $D$ on $\tilde{X}$ such that $\tilde{\pi}A_0 = \tilde{\sigma}D$, (c.f.\ the notation in \cref{pa:conventions}). We then define $\ind_{\bL\subset\bP}^{\bG}(A_0) = \tau_!D$.
\end{pa}

\begin{pa}\label{pa:induction-char-sheaves}
We say a character sheaf $A \in \widehat{\bG}$ is non-cuspidal if there exists a pair $(\bL,\bP) \in \mathcal{Z}$ and a character sheaf $A_0 \in \widehat{\bL}$ such that $A$ is a direct summand of the induced complex $\ind_{\bL\subset\bP}^{\bG}(A_0)$; otherwise we say $A$ is cuspidal. Lusztig has shown that if $A_0 \in \widehat{\bL}$ then $\ind_{\bL\subset\bP}^{\bG}(A_0)$ is semisimple and contained in $\mathscr{M}\bG$, (see \cite[Proposition 4.8(b)]{lusztig:1985:character-sheaves}). Furthermore for any $A \in \widehat{\bG}$ there exists a pair $(\bL,\bP) \in \mathcal{Z}$ and a cuspidal character sheaf $A_0 \in \widehat{\bL}$ such that $A$ occurs as a direct summand of $\ind_{\bL\subset\bP}^{\bG}(A_0)$, (see \cite[Theorem 4.4]{lusztig:1985:character-sheaves}).

Let us now fix a pair $(\bL,\bP) \in \mathcal{Z}$ such that there exists a cuspidal character sheaf $A_0 \in \widehat{\bL}$. By \cite[Proposition 3.12]{lusztig:1985:character-sheaves} we have $A_0$ is isomorphic to an intersection cohomology complex $\IC(\overline{\Sigma},\mathscr{E})[\dim\Sigma]$ where $\Sigma \subset \bL$ is the inverse image under $\bL \mapsto \bL/Z^{\circ}(\bL)$ of an isolated conjugacy class and $\mathscr{E}$ is a local system on $\Sigma$. From the proof of this result we know that $(\Sigma,\mathscr{E})$ is a cuspidal pair in the sense of \cite[Definition 2.4]{lusztig:1984:intersection-cohomology-complexes}. Of particular interest to us will be the special case where $\supp(A_0) \cap \bL_{\uni} \neq \emptyset$, where for any connected reductive algebraic group $\bH$ we denote by $\bH_{\uni}$ the variety of unipotent elements. Assume this is so then there exists a triple $(\mathcal{O}_0,\mathscr{E}_0,\mathscr{L})$ consisiting of: a unipotent conjugacy class $\mathcal{O}_0 \subset \bL$, a cuspidal local system $\mathscr{E}_0$ on $\mathcal{O}_0$ and a local system $\mathscr{L} \in \mathcal{S}(Z^{\circ}(\bL))$ such that $A_0$ is isomorphic to 
\begin{equation}\label{eq:cusp-uni-sup}
A_{\mathscr{L}} := \IC(\overline{\Sigma},\mathscr{E}_0\boxtimes\mathscr{L})[\dim\Sigma]
\end{equation}
where $\Sigma = \mathcal{O}_0Z^{\circ}(\bL)$. Here we consider $\mathscr{E}_0\boxtimes\mathscr{L}$ as a local system on $\mathcal{O}_0 \times Z^{\circ}(\bL)$ which we identify with the open subset $\mathcal{O}_0Z^{\circ}(\bL) \subseteq \overline{\mathcal{O}_0}Z^{\circ}(\bL)$ under the multiplication morphism in $\bL$. With all of this we may now define a map
\begin{equation}\label{eq:ind-P-construction}
(\bL,\bP,\mathcal{O}_0,\mathscr{E}_0,\mathscr{L}) \longrightarrow \ind_{\bL\subset\bP}^{\bG}(A_{\mathscr{L}})
\end{equation}

\begin{assumption}
From this point forward $\Sigma$ will always denote the variety $\mathcal{O}_0Z^{\circ}(\bL)$ where $\mathcal{O}_0 \subset \bL$ is a unipotent conjugacy class supporting a cuspidal local system.
\end{assumption}

Let $A \in \widehat{\bG}$ be a character sheaf which occurs as a direct summand of the induced complex $\ind_{\bL\subset\bP}^{\bG}(A_0)$, (where here $A_0 \in \widehat{\bL}$ is still assumed to be cuspidal), then by \cite[\S2.9]{lusztig:1986:on-the-character-values} we have
\begin{equation}\label{eq:supp}
\supp A = \bigcup_{x\in\bG} x(\supp A_0)\bU_{\bP}x^{-1}
\end{equation}
where $\bU_{\bP}\leqslant \bP$ is the unipotent radical of $\bP$. In particular we have $\supp A \cap \bG_{\uni} \neq \emptyset$ if and only if $\supp A_0 \cap \bL_{\uni} \neq \emptyset$. Hence if we are only interested in character sheaves supported by a unipotent element then we need only concern ourselves with those character sheaves occurring in an induced complex of the form given in \cref{eq:cusp-uni-sup}. We end our discussion of induced complexes with the following lemma which gives some facts concerning the inverse image of an induced complex.
\end{pa}

\begin{lem}\label{lem:F-action-ind}
We denote by $\bH$ a connected reductive algebraic group and by $\bP$ a parabolic subgroup of $\bH$ with Levi complement $\bL$. Furthermore we assume that $\bG$ is a connected reductive algebraic group and $i : \bG \to \bH$ is a bijective morphism of varieties.
\begin{enumerate}[label=(\alph*)]
\item For any $A_0 \in \widehat{\bL}$ we have
\begin{equation*}
i^*\ind_{\bL\subset\bP}^{\bH}(A_0) = \ind_{i^{-1}(\bL)\subset i^{-1}(\bP)}^{\bG}(i^*A_0).
\end{equation*}
Furthermore any character sheaf $A \in \widehat{\bH}$ is cuspidal if and only if $i^*A \in \widehat{\bG}$ is cuspidal.

\item Let us also assume that $\bG = \bH$ and $\bL = i(\bL)$ then any isomorphism $\phi : i^*A_0 \to A_0$ induces an isomorphism
\begin{equation*}
\tilde{\phi} : i^*\ind_{\bL\subset i(\bP)}^{\bH}(A_0) \to \ind_{\bL\subset \bP}^{\bH}(A_0).
\end{equation*}
\end{enumerate}
\end{lem}

\begin{proof}
Firstly by \cite[\S24.1 - Proposition B]{humphreys:1975:linear-algebraic-groups} we have $i$ induces a bijection between the Borel subgroups of $\bG$ and $\bH$ so this implies that $\bQ := i^{-1}(\bP)$ is a parabolic subgroup of $\bG$ with Levi complement $\bM := i^{-1}(\bL)$. For any object $\square$ introduced in \cref{pa:definition-of-induction} we write $\square_{\bL\subset\bP}^{\bH}$, (resp.\ $\square_{\bM\subset\bQ}^{\bG}$), to indicate that it is defined with respect to $\ind_{\bL\subset\bP}^{\bH}$, (resp.\ $\ind_{\bM\subset\bQ}^{\bG}$). We now have a diagram
\begin{center}
\begin{tikzcd}
\bM \arrow{d}[swap]{i} & \hat{X}_{\bM\subset\bQ}^{\bG} \arrow{l}[swap]{\pi_{\bM\subset\bQ}^{\bG}}\arrow{d}[swap]{i}\arrow{r}{\sigma_{\bM\subset\bQ}^{\bG}} & \tilde{X}_{\bM\subset\bQ}^{\bG} \arrow{d}[swap]{i}\arrow{r}{\tau_{\bM\subset\bQ}^{\bG}} & \bG \arrow{d}[swap]{i}\\
\bL & \hat{X}_{\bL\subset\bP}^{\bH} \arrow{l}[swap]{\pi_{\bL\subset\bP}^{\bH}} \arrow{r}{\sigma_{\bL\subset\bP}^{\bH}} & \tilde{X}_{\bL\subset\bP}^{\bH} \arrow{r}{\tau_{\bL\subset\bP}^{\bH}} & \bH
\end{tikzcd}
\end{center}
where the vertical maps are the obvious actions of $i$. The squares of the above diagram are clearly commutative. Let $D$ be the canonical perverse sheaf on $\tilde{X}_{\bL\subset\bP}^{\bH}$ satisfying $\tilde{\pi}_{\bL\subset\bP}^{\bH}A_0 = \tilde{\sigma}_{\bL\subset\bP}^{\bH}D$. The fibres of $\pi_{\bL\subset\bP}^{\bH}$ and $\pi_{\bM\subset\bQ}^{\bG}$ have the same dimension, as do the fibres of $\sigma_{\bL\subset\bP}^{\bH}$ and $\sigma_{\bM\subset\bQ}^{\bG}$, therefore we have $i^*D$ satisfies $\tilde{\pi}_{\bM\subset\bQ}^{\bG}i^*A_0 = \tilde{\sigma}_{\bM\subset\bQ}^{\bG} i^*D$ because the inverse image is contravariant. By definition we have $\ind_{\bL\subset\bP}^{\bH}(A_0) = (\tau_{\bL\subset\bP}^{\bH})_!D$ and $\ind_{\bM\subset\bQ}^{\bG}(i^*A_0) = (\tau_{\bM\subset\bQ}^{\bG})_!i^*D$ so the first part of the lemma follows if we can show the equality $i^*(\tau_{\bL\subset\bP}^{\bH})_!D = (\tau_{\bM\subset\bQ}^{\bG})_!i^*D$ but this is just \cref{pa:direct-image-commutative}. The conclusion concerning cuspidality is an immediate consequence of the first part, which proves (a).

We now prove (b). Let $\bQ = i(\bP)$ then collapsing the notation $\square_{\bL\subset\bP}^{\bH}$ simply to $\square_{\bP}$ we have a diagram
\begin{center}
\begin{tikzcd}
\bL \arrow{d}[swap]{i} & \hat{X}_{\bP} \arrow{l}[swap]{\pi_{\bP}}\arrow{d}[swap]{i}\arrow{r}{\sigma_{\bP}} & \tilde{X}_{\bP} \arrow{d}[swap]{i}\arrow{r}{\tau_{\bP}} & \bH \arrow{d}[swap]{i}\\
\bL & \hat{X}_{\bQ} \arrow{l}[swap]{\pi_{\bQ}} \arrow{r}{\sigma_{\bQ}} & \tilde{X}_{\bQ} \arrow{r}{\tau_{\bQ}} & \bH
\end{tikzcd}
\end{center}
with commutative squares. Let $K$, (resp.\ $K'$), be the canonical perverse sheaf on $\tilde{X}_{\bP}$, (resp.\ $\tilde{X}_{\bQ}$), satisfying $\tilde{\pi}_{\bP}A_0 = \tilde{\sigma}_{\bP}K$, (resp.\ $\tilde{\pi}_{\bQ}A_0 = \tilde{\sigma}_{\bQ}K'$). Now $\tilde{\pi}_{\bP}\phi$ defines an isomorphism $\tilde{\pi}_{\bP}i^*A_0 \to \tilde{\pi}_{\bP}A_0$ and using the commutativity of the above diagram we may view this as an isomorphism $\tilde{\sigma}_{\bP}i^*K' \to \tilde{\sigma}_{\bP}K$. As $\sigma_{\bP}$ is smooth with connected fibres we have $\tilde{\sigma}_{\bP}$ is a fully faithful functor, (see \cite[1.8.3]{lusztig:1985:character-sheaves}), hence there exists a unique isomorphism $\phi' : i^*K' \to K$ such that $\tilde{\pi}_{\bP}\phi = \tilde{\sigma}_{\bP}\phi'$. Using the arguments above we see that $\tilde{\phi} = (\tau_{\bP})_!\phi'$ gives the required isomorphism.
\end{proof}
%
\section{The Space of Unipotently Supported Class Functions}\label{sec:space-unip-supp-class-func}
\begin{pa}\label{pa:gen-spring-cor}
Let $\mathcal{N}_{\bG}$ denote the set of all pairs $\iota = (\mathcal{O},\mathscr{E})$ where $\mathcal{O}$ is a unipotent conjugacy class of $\bG$ and $\mathscr{E}$ is an irreducible $\bG$-equivariant local system. We denote by $\mathcal{N}_{\bG}^0 \subseteq \mathcal{N}_{\bG}$ the subset consisting of those pairs $(\mathcal{O},\mathscr{E})$ such that $\mathscr{E}$ is a cuspidal local system on $\mathcal{O}$, (see \cite[Definition 2.4]{lusztig:1984:intersection-cohomology-complexes}); we call the elements of $\mathcal{N}_{\bG}^0$ cuspidal pairs. Given a pair $\iota$ we will denote the class $\mathcal{O}$ by $\mathcal{O}_{\iota}$ and the local system $\mathscr{E}$ by $\mathscr{E}_{\iota}$.

Let us denote by $\widetilde{\mathcal{M}}_{\bG}$ the set of all pairs $(\bL,\nu)$ consisting of a Levi subgroup $\bL \in \mathcal{L}$ and a cuspidal pair $\nu \in \mathcal{N}_{\bL}^0$. We have $\bG$ acts naturally on $\widetilde{\mathcal{M}}_{\bG}$ by conjugation and we denote by $[\bL,\nu]$ the orbit containing $(\bL,\nu)$ and by $\mathcal{M}_{\bG}$ the set of all such orbits. Recall that in \cite[Theorem 6.5]{lusztig:1984:intersection-cohomology-complexes} Lusztig has associated to every pair $\iota \in \mathcal{N}_{\bG}$ a unique orbit $\mathcal{C}_{\iota} \in \mathcal{M}_{\bG}$. This is known as the generalised Springer correspondence.

\begin{assumption}
For each $\iota \in \mathcal{N}_{\bG}$ we will now choose a representative $(\bL_{\iota},\nu_{\iota}) \in \mathcal{C}_{\iota}$ such that $\bL_{\iota} \in \mathcal{L}_{\std}$ is a standard Levi subgroup.
\end{assumption}

\noindent Lusztig has shown that we have a disjoint union
\begin{equation*}
\mathcal{N}_{\bG} = \bigsqcup_{[\bL,\upsilon] \in \mathcal{M}_{\bG}} \bl{I}[\bL,\upsilon] \qquad\text{where}\qquad \bl{I}[\bL,\upsilon] = \{\iota \in \mathcal{N}_{\bG} \mid (\bL_{\iota},\upsilon_{\iota}) \in [\bL,\upsilon]\}.
\end{equation*}
We call $\bl{I}[\bL,\upsilon]$ a block of $\mathcal{N}_{\bG}$. If $\nu \in \mathcal{N}_{\bL}^0$ is cuspidal then $W_{\bG}(\bL) = N_{\bG}(\bL)/\bL$ is a Coxeter group, (see \cite[Theorem 9.2(a)]{lusztig:1984:intersection-cohomology-complexes}), and we have a bijection
\begin{equation}\label{eq:gen-spring-cor}
\bl{I}[\bL,\nu] \longleftrightarrow \Irr(W_{\bG}(\bL))
\end{equation}
for all $[\bL,\nu] \in \mathcal{M}_{\bG}$ which we denote $\iota \mapsto E_{\iota}$. If $\bL$ is a torus then $\nu$ is simply the pair consisting of the trivial class and trivial local system, in which case this bijection is the classical Springer correspondence and we call $\mathscr{I}[\bL,\nu]$ the principal block.
\end{pa}

\begin{pa}\label{pa:decomp-by-end-algebra}
To describe the correspondence in \cref{eq:gen-spring-cor} we will need to recall a description of semisimple objects in $\mathscr{M}\bG$ following \cite[3.7]{lusztig:1984:intersection-cohomology-complexes}. Let $K \in \mathscr{M}\bG$ be semisimple and let $\mathcal{A} = \End(K)$ be the endomorphism algebra of $K$, by which we mean the set $\Hom_{\mathscr{M}\bG}(K,K)$ in the category $\mathscr{M}\bG$. Assume $E$ is any finite dimensional $\mathcal{A}$-module then we define
\begin{equation}\label{eq:K-nu-E-L}
K_E = \Hom_{\mathcal{A}}(E,K) \in \mathscr{M}\bG.
\end{equation}
To see that this is an object of $\mathscr{M}\bG$ we can construct this in the following way. Let us pick a presentation $\mathcal{A}^m \overset{\varphi}{\to} \mathcal{A}^n \to E \to 0$ of the module $E$. Applying $\Hom_{\mathcal{A}}(-,K)$ and using the fact that we have an isomorphism $K \cong \Hom_{\mathcal{A}}(\mathcal{A},K)$ of $\mathcal{A}$-modules we get a diagram
\begin{equation}\label{eq:diag-Hom-MG}
\begin{tikzcd}
0 \arrow{r}{} & \Hom_{\mathcal{A}}(E,K) \arrow{d}{\cong}\arrow{r}{} & \Hom_{\mathcal{A}}(\mathcal{A}^n,K) \arrow{d}{\cong}\arrow{r}{\varphi^*} & \Hom_{\mathcal{A}}(\mathcal{A}^m,K) \arrow{d}{\cong}\\
0 \arrow{r}{} & \Ker_{\mathscr{M}\bG}(\varphi^*) \arrow{r}{} & K^n \arrow{r}{\varphi^*} & K^m
\end{tikzcd}
\end{equation}
with exact rows, ($\varphi^*$ is the map induced by $\varphi$ and the vertical arrows are $\mathcal{A}$-module isomorphisms). The $\mathcal{A}$-module homomorphism $\varphi^*$ is also a morphism in $\mathscr{M}\bG$ because $\mathscr{M}\bG$ is $\Ql$-linear and the kernel exists because $\mathscr{M}\bG$ is an abelian category. Hence $\Hom_{\mathcal{A}}(E,K) \cong \Ker(\varphi^*) \in \mathscr{M}\bG$ as desired. Now $K_E$ is a simple object of $\mathscr{M}\bG$ if and only if $E$ is a simple $\mathcal{A}$-module and we have a canonical isomorphism
\begin{equation}\label{eq:K-nu-L-decomp}
K \cong \bigoplus_E (E \otimes K_E)
\end{equation}
in $\mathscr{M}\bG$ given by $e\otimes f \mapsto f(e)$ in each summand. Here the sum runs over a set of representatives for the isomorphism classes of simple $\mathcal{A}$-modules.
\end{pa}

\begin{pa}\label{pa:unip-sup-char-sheaves}
Let $\bL \in \mathcal{L}$ be a Levi subgroup and $\nu = (\mathcal{O}_0,\mathscr{E}_0) \in \mathcal{N}_{\bL}^0$ a cuspidal pair then following \cite[2.3]{lusztig:1986:on-the-character-values} we define for any local system $\mathscr{L} \in \mathcal{S}(Z^{\circ}(\bL))$ a semisimple perverse sheaf $K_{\mathscr{L}} \in \mathscr{M}\bG$ in the following way. We define an open subset $\Sigma_{\reg}= \mathcal{O}_0\cdot Z^{\circ}(\bL)_{\reg} \subseteq \Sigma$ where $Z^{\circ}(\bL)_{\reg} = \{z \in Z^{\circ}(\bL) \mid C_{\bG}^{\circ}(z) = \bL\}$ and we denote by $Y$ the locally closed smooth irreducible subvariety of $\bG$ given by $\cup_{x \in \bG} x\Sigma_{\reg}x^{-1}$. We have the following diagram
\begin{equation}
\begin{tikzcd}
\Sigma & \hat{Y} \arrow{l}[swap]{\alpha}\arrow{r}{\beta} & \tilde{Y} \arrow{r}{\gamma} & Y
\end{tikzcd}
\end{equation}
where
\begin{gather*}
\begin{aligned}
\hat{Y} &= \{(g,x) \in \bG \times \bG \mid x^{-1}gx \in \Sigma\} \quad&\quad \tilde{Y} &= \{(g,x\bL) \in \bG \times (\bG/\bL) \mid x^{-1}gx \in \bL\}
\end{aligned}\\
\begin{aligned}
\alpha(g,x) &= x^{-1}gx \quad&\quad \beta(g,x) &= (g,x\bL) \quad&\quad \gamma(g,x\bL) &= g.
\end{aligned}
\end{gather*}
Since the local system $\mathscr{E}_0$ is $\bL$-equivariant there exists a canonical local system $\tilde{\mathscr{E}}_0$ on $\tilde{Y}$ such that $\beta^*\tilde{\mathscr{E}}_0 = \alpha^*(\mathscr{E}_0\boxtimes\Ql)$, (see \cite[1.9.3]{lusztig:1985:character-sheaves-I}). Let $\delta : \tilde{Y} \to Z^{\circ}(\bL)$ be the map given by $\delta(g,x\bL) = (xgx^{-1})_{\sss}$, (where $h_{\sss}$ is the semisimple part of $h$ for all $h \in \bG$), then $\tilde{\mathscr{L}} = \delta^*\mathscr{L}$ is a local system on $\tilde{Y}$ hence so is $\tilde{\mathscr{E}}_0\otimes\tilde{\mathscr{L}}$. By \cite[3.2]{lusztig:1984:intersection-cohomology-complexes} we have $\gamma$ is a Galois covering so $\gamma_* = \gamma_!$ because $\gamma$ is finite, (hence proper), which means that $\gamma_*(\tilde{\mathscr{E}}_0\otimes\tilde{\mathscr{L}}) = \gamma_!(\tilde{\mathscr{E}}_0\otimes\tilde{\mathscr{L}})$ is a semisimple local system on $Y$. We now define $K_{\mathscr{L}} \in \mathscr{M}\bG$ to be the intersection cohomology complex $\IC(\overline{Y},\gamma_*(\tilde{\mathscr{E}}_0\otimes\tilde{\mathscr{L}}))[\dim Y]$. With this we have defined a map
\begin{equation}\label{eq:ind-w-P-construction}
(\bL,\mathcal{O}_0,\mathscr{E}_0,\mathscr{L}) \longrightarrow K_{\mathscr{L}}
\end{equation}
\end{pa}

\begin{pa}\label{pa:endomorphism-algebra-A}
Let us keep the notation of \cref{pa:unip-sup-char-sheaves}. We denote by $N_{\bG}(\bL,\mathscr{L})$ the set of all elements $n \in N_{\bG}(\bL)$ such that $(\ad n)^*\mathscr{L}$ is isomorphic to $\mathscr{L}$, where $\ad n : \bG \to \bG$ is the conjugation morphism given by $(\ad n)(g) = ngn^{-1}$ for all $g \in \bG$. Clearly $\bL$ is contained in $N_{\bG}(\bL,\mathscr{L})$ and we denote the quotient group $N_{\bG}(\bL,\mathscr{L})/\bL$ by $W_{\bG}(\bL,\mathscr{L})$. Let us denote by $\mathcal{A}_{\mathscr{L}}$ the endomorphism algebra of the semisimple perverse sheaf $K_{\mathscr{L}}$. By \cite[\S2.4(a)]{lusztig:1986:on-the-character-values} there exists a set of basis elements $\{\Theta_v \mid v \in W_{\bG}(\bL,\mathscr{L})\} \subset \mathcal{A}_{\mathscr{L}}$ defining a canonical algebra isomorphism
\begin{equation}\label{eq:end-grp-alg-iso}
\mathcal{A}_{\mathscr{L}} \cong \Ql W_{\bG}(\bL,\mathscr{L}).
\end{equation}
Let us now assume that $\mathscr{L} = \Ql$ then we denote $\mathcal{A}_{\mathscr{L}}$ simply by $\mathcal{A}$ and $K_{\mathscr{L}}$ by $K$. By \cite[Theorem 9.2(b)]{lusztig:1984:intersection-cohomology-complexes} we have $W_{\bG}(\bL,\mathscr{L}) = W_{\bG}(\bL)$ when $\mathscr{L} = \Ql$ hence $\mathcal{A}$ is isomorphic to $\Ql W_{\bG}(\bL)$. Using the description given in \cref{pa:decomp-by-end-algebra} we have $E\mapsto K_E = \Hom_{W_{\bG}(\bL)}(E,K)$ gives a bijection between simple $W_{\bG}(\bL)$-modules and the simple summands of $K$ (up to isomorphism). Furthermore by \cite[Theorem 6.5]{lusztig:1984:intersection-cohomology-complexes} we have for each $E$ that there exists a unique pair $\iota \in \mathscr{I}[\bL,\nu]$ satisfying
\begin{equation*}
K_E|_{\bG_{\uni}} \cong \IC(\overline{\mathcal{O}}_{\iota},\mathscr{E}_{\iota})[\dim\mathcal{O}_{\iota} + \dim Z^{\circ}(\bL)].
\end{equation*}
The composition of the maps $E \mapsto K_E \mapsto \iota$ then gives the bijection in \cref{eq:gen-spring-cor} and we will freely denote the object $K_E$ by $K_{\iota}$.
\end{pa}

\begin{pa}\label{pa:K-nu-L-ind-A-nu-L}
We now end this section by explicitly describing how $K_{\mathscr{L}}$ is related to $\ind_{\bL\subset\bP}^{\bG}(A_{\mathscr{L}})$, (for this we follow \cite{lusztig:1984:intersection-cohomology-complexes}). Here we assume $(\bL,\bP) \in \mathcal{Z}$ and $\nu \in \mathcal{N}_{\bL}^0 \neq \emptyset$. Firstly let us recall the notation of \cref{pa:definition-of-induction} then we wish to describe the complex $K$. To do this we first observe that
\begin{equation*}
\tilde{\pi}A_{\mathscr{L}} = \IC(\hat{X}',\pi^*(\mathscr{E}_0\boxtimes\mathscr{L}))[\dim \overline{\Sigma}+\dim\bG + \dim\bU_{\bP}]
\end{equation*}
where $\bU_{\bP}$ is the unipotent radical of $\bP$ and $\hat{X}' = \pi^{-1}(\overline{\Sigma}) = \{(g,h) \in \bG \times \bG \mid h^{-1}gh \in \overline{\Sigma}\cdot\bU_{\bP}\}$. Note that for $l \in \bL$ we have $\pi^{-1}(l) = \{(g,h) \in \bG \times \bG \mid h^{-1}gh \in l\cdot \bU_{\bP}\}$ hence the fibres of $\pi$ have dimension $\dim\bG+\dim\bU_{\bP}$. Let $j : \tilde{X} \to \hat{X}$ be the map $j(g,h\bP) = (g,h)$ then by the construction in \cite[1.9.3]{lusztig:1985:character-sheaves-I} we see that
\begin{equation*}
K = j^*\IC(\hat{X}',\pi^*(\mathscr{E}_0\boxtimes\mathscr{L}))[\dim \overline{\Sigma}+\dim\bG + \dim\bU_{\bP} - \dim\bP] = \IC(\tilde{X}',\overline{\mathscr{E}_0\boxtimes\mathscr{L}})[\dim \tilde{X}']
\end{equation*}
where $\tilde{X}' = j^{-1}(\hat{X}') = \{(g,h\bP) \in \bG \times \bG/\bP \mid h^{-1}gh \in \overline{\Sigma}\cdot\bU_{\bP}\}$ and $\overline{\mathscr{E}_0\boxtimes\mathscr{L}} = j^*\pi^*(\mathscr{E}_0\boxtimes\mathscr{L})$. Note that the equality $\dim \tilde{X}' = \dim Y = \dim\bG - \dim\bP + \dim\overline{\Sigma}+ \dim\bU_{\bP}$ is given in the proof of \cite[Lemma 4.3(a)]{lusztig:1984:intersection-cohomology-complexes}.

Let $Y$ and $\tilde{Y}$ be as in \cref{pa:unip-sup-char-sheaves} then by \cite[Lemma 4.3(b)]{lusztig:1984:intersection-cohomology-complexes} we have $\tau(\tilde{X}') = \overline{Y}$ hence $\tau^{-1}(Y) \subset \tilde{X}' \subset \tilde{X}$. Furthermore by \cite[Lemma 4.3(c)]{lusztig:1984:intersection-cohomology-complexes} we have the map $(g,h\bL) \to (g,h\bP)$ defines an isomorphism $\kappa : \tilde{Y} \to \tau^{-1}(Y)$. As $\tau^{-1}(Y) \subset \tilde{X}'$ we have $K|_{\tau^{-1}(Y)} \cong \overline{\mathscr{E}_0\boxtimes\mathscr{L}}|_{\tau^{-1}(Y)}[\dim\tilde{X}']$ and according to \cite[4.4]{lusztig:1984:intersection-cohomology-complexes} we have $\kappa^*(\overline{\mathscr{E}_0\boxtimes\mathscr{L}}|_{\tau^{-1}(Y)}) \cong \tilde{\mathscr{E}}_0 \otimes \tilde{\mathscr{L}}$ hence
\begin{equation*}
(\tau_!K)|_Y = \tau_!(K_{\tau^{-1}(Y)}) \cong \gamma_*\kappa^*(\overline{\mathscr{E}_0\boxtimes\mathscr{L}}|_{\tau^{-1}(Y)}[\dim\tilde{X}']) \cong \gamma_*(\tilde{\mathscr{E}}_0\otimes\tilde{\mathscr{L}})[\dim Y].
\end{equation*}
Here we have applied \cref{pa:direct-image-commutative} to obtain $\tau_! = \ID^*\circ\tau_! = \gamma_!\circ\kappa^* = \gamma_*\circ\kappa^*$ where the last equality follows from the fact that $\gamma$ is proper. With this we now have the following.
\end{pa}

\begin{prop}[{}{Lusztig, \cite[Proposition 4.5]{lusztig:1984:intersection-cohomology-complexes}}]\label{prop:K-nu-L-iso-to-induction}
The complex $\tau_!K$ is a perverse sheaf and is canonically isomorphic to $\IC(\overline{Y},\gamma_*(\tilde{\mathscr{E}}_0\otimes\tilde{\mathscr{L}}))[\dim Y]$. In particular $\ind_{\bL \subset\bP}^{\bG}(A_{\mathscr{L}})$ and $K_{\mathscr{L}}$ are canonically isomorphic in $\mathscr{M}\bG$.
\end{prop}

\begin{rem}
Note that \cref{prop:K-nu-L-iso-to-induction} implies that $\ind_{\bL\subset\bP}^{\bG}(A_{\mathscr{L}})$ does not depend upon the parabolic $\bP$ so in this situation we will write simply $\ind_{\bL}^{\bG}(A_{\mathscr{L}})$ when convenient.
\end{rem}
%
\section{Split Elements and Lusztig's Preferred Extensions}\label{sec:split-elements}
\begin{assumption}
From now until the end of this article we assume that $\bG$ has a connected centre and $\bG/Z(\bG)$ is simple. Recall also our standing assumption that $p$ is a good prime for $\bG$.
\end{assumption}

\begin{pa}
In \cite[Remark 5.1]{shoji:1987:green-functions-of-reductive-groups}, assuming $F$ is a split Frobenius endomorphism, Shoji has defined the notion of a \emph{split} unipotent element. However there is also a notion of split unipotent element when $F$ is not split which we would like to recall here. All elements of this section are due to the combined efforts of Hotta--Springer, Beynon--Spaltenstein and Shoji. To give the definition we must first prepare some preliminary notions and notation.

Given a unipotent element $u \in \bG$ we denote by $\mathfrak{B}_u^{\bG}$ the variety of Borel subgroups containing the unipotent element $u$ and we denote its dimension $\dim\mathfrak{B}_u^{\bG}$ by $d_u$. The corresponding $\ell$-adic cohomology groups with compact support $H_c^i(\mathfrak{B}_u^{\bG}) := H_c^i(\mathfrak{B}_u^{\bG},\Ql)$ are modules for the direct product $A_{\bG}(u)\times W_{\bG}$ which are zero unless $i \in \{0,\dots,2d_u\}$, (in fact $i$ must be even). Note that here we take the $(A_{\bG}(u)\times W_{\bG})$-module structure to be the one described by Lusztig in \cite[\S3]{lusztig:1981:green-polynomials-and-singularities}, (this fits with the generalised Springer correspondence given in \cite{lusztig:1984:intersection-cohomology-complexes}). This differs from the original module structure given by Springer in \cite{springer:1978:construction-of-representations-of-weyl-groups} but one is translated to the other by composing the $W_{\bG}$-action with the sign character, (see \cite[Theorem 1]{hotta:1981:on-springers-representations}).

Assume now that $u \in G$ is fixed by $F$ then $F$ stabilises $\mathfrak{B}_u^{\bG}$ hence we have an induced action of $F$ in the compactly supported cohomology which we denote by $F^{\bullet} : H_c^{\bullet}(\mathfrak{B}_u^{\bG}) \to H_c^{\bullet}(\mathfrak{B}_u^{\bG})$. Let $\psi \in \Irr(A_{\bG}(u))$ be an irreducible character then we denote by $H_c^{\bullet}(\mathfrak{B}_u^{\bG})_{\psi}$ the $\psi$-isotypic component of the module $H_c^{\bullet}(\mathfrak{B}_u^{\bG})$. Springer's main result, (see \cite[Theorem 1.13]{springer:1978:construction-of-representations-of-weyl-groups}), shows that either $H_c^{2d_u}(\mathfrak{B}_u^{\bG})_{\psi}$ is 0 or a simple $(A_{\bG}(u) \times W_{\bG})$-module isomorphic to $\psi\otimes E_{u,\psi}$, (this is the classical formulation of the Springer correspondence).
\end{pa}

\begin{pa}
Let us decompose $F$ as $F_0\circ\sigma$ where $F_0$ is a field automorphism and $\sigma$ is a graph automorphism. We now define what it means for a unipotent element $u \in G$ to be \emph{split}. We will do this on a case by case basis as follows.
\begin{enumerate}[label=(\arabic*)]
	\item $\bG$ not of type $\E_8$ and $\sigma$ the identity. We say $u \in \bG^F$ is split if $F$ stabilises every irreducible component of $\mathfrak{B}_u^{\bG}$. \cite{beynon-spaltenstein:1984:green-functions,hotta-springer:1977:a-specialization-theorem,shoji:1982:on-the-green-polynomials-F4,shoji:1983:green-polynomials-of-classical-groups}\label{it:adjoint-simple}
	\item $\bG$ of type $\E_8$. If $q \equiv 1 \pmod{3}$ then we define split elements as in case (1). If $q \equiv -1 \pmod{3}$ then for each $F$-stable class $\mathcal{O}$ we define an element $u \in\mathcal{O}^F$ to be split if it satisfies the condition of (1) unless $\mathcal{O}$ is the class $\E_8(b_6)$ in the Bala--Carter labelling, (see \cite[pg.\ 177]{carter:1993:finite-groups-of-lie-type}). For this class we say $u \in \mathcal{O}^F$ is split if the restriction of $F^{\bullet}$ to $H^{2d_u}(\mathfrak{B}_u)_{\psi}$ is multiplication by $q^{d_u}$ for each irreducible character $\psi$ of $A_{\bG}(u) \cong \mathfrak{S}_3$ which is not the sign character. \cite[\S3 - Case (V)]{beynon-spaltenstein:1984:green-functions}\label{it:adjoint-simple-E8}
	\item $\bG$ of type $\A_n$ and $\sigma$ of order 2. We define any unipotent element $u \in \bG^F$ to be split.\label{it:adjoint-simple-An}
	\item $\bG$ of type $\D_n$ and $\sigma$ of order 2. Let $\bG_{\ad}$ be an adjoint group of type $\D_n$ and let us fix adjoint quotients $\alpha : \bG \to \bG_{\ad}$ and $\beta : \SO_{2n}(\mathbb{K}) \to \bG_{\ad}$ defined over $\mathbb{F}_q$. We say $u \in G$ is split if $(\beta^{-1}\circ\alpha)(u)$ is a split element as defined in \cite[2.10.1]{shoji:2007:generalized-green-functions-II}. Note that there can only be one element in the preimage of $\alpha(u)$ under $\beta$.\label{it:adjoint-simple-Dn}
	\item $\bG$ of type $\D_4$ and $\sigma$ of order 3. We say $u \in \bG^F$ is split if it satisfies the condition of (1) with respect to the split Frobenius endomorphism $F^3$. \cite[\S4.24]{shoji:1983:green-polynomials-of-classical-groups}\label{it:adjoint-simple-D4}
	\item $\bG$ of type $\E_6$ and $\sigma$ of order 2. In \cite[\S4]{beynon-spaltenstein:1984:green-functions} Beynon--Spaltenstein define a bijection between the set of $\bG^F$-conjugacy classes of unipotent elements and the set of $\bG^{F_0}$-conjugacy classes of unipotent elements. We say a unipotent element $u \in \bG^F$ is split if it is contained in a $\bG^F$-conjugacy class which is in bijective correspondence with a split $\bG^{F_0}$-conjugacy class under Beynon--Spaltenstein's bijection.\label{it:adjoint-simple-E6}
\end{enumerate}

By Poincar\'{e} duality we have an isomorphism $H_c^{2d_u}(\mathfrak{B}_u^{\bG}) \cong H^0(\mathfrak{B}_u^{\bG})$ where the latter group is simply the global sections of the constant sheaf $\Ql$. As the irreducible components of $\mathfrak{B}_u^{\bG}$ have the same dimension, (see \cite[I, Proposition 1.12]{spaltenstein:1982:classes-unipotentes}), they form a basis for $H^0(\mathfrak{B}_u^{\bG})$. In particular \cref{it:adjoint-simple} is equivalent to saying that $F^{\bullet}$ acts as the identity on $H_c^{2d_u}(\mathfrak{B}_u^{\bG})$. With this we see that the references given above show that every $F$-stable unipotent conjugacy class $\mathcal{O} \subset \bG$ contains a split unipotent element. Before continuing we recall the following properties of split elements.
\end{pa}

\begin{lem}
Assume $\mathcal{O}$ is an $F$-stable unipotent conjugacy class of $\bG$ and $u \in \mathcal{O}^F$ is a split unipotent element then the split elements contained in $\mathcal{O}^F$ form a single $G$-conjugacy class.
\end{lem}

\begin{proof}
Assume $\bG$ is of type $\A_n$ then this is trivial as $A_{\bG}(u)$ is trivial. If $\bG$ is of type $\B_n$, $\C_n$ or $\D_n$ then this is clear from the definition, (see \cite[\S2.7, \S2.10]{shoji:2007:generalized-green-functions-II}). If $\bG$ is of exceptional type then this is noted by Benyon-Spaltenstein in \cite[\S3]{beynon-spaltenstein:1984:green-functions}. The only case not explicitly dealt with is \cref{it:adjoint-simple-D4}, however this is easily checked in this case.
\end{proof}

\begin{lem}\label{lem:Frob-action}
If $u \in G$ is a split unipotent element then $F$ acts trivially on $A_{\bG}(u)$.
\end{lem}

\begin{proof}
If $\bG$ is of classical type then the action of $F$ on $A_{\bG}(u)$ is always trivial, (see for instance the proof of \cite[Proposition 2.4]{taylor:2011:on-unipotent-supports}). Assume now that $\bG$ is of exceptional type then from \cite{beynon-spaltenstein:1984:green-functions} one sees that Beynon--Spaltenstein start with an element satisfying this property then show that it is split, hence this certainly holds.
\end{proof}

\begin{pa}
The Frobenius induces an automorphism of $(W_{\bG},\mathbb{S})$, which we again denote by $F$, and we denote by $\widetilde{W}_{\bG}$ the semidirect product $W_{\bG} \rtimes \langle F \rangle$. Assume $u \in \bG^F$ is a split unipotent element and $\psi \in \Irr(A_{\bG}(u))$ then we consider the cohomology group $E = H_c^{2d_u}(\mathfrak{B}_u^{\bG})_{\psi}$ as a $W_{\bG}$-module. If $E$ is $F$-stable then, (by \cref{lem:Frob-action}), we may consider $H_c^{2d_u}(\mathfrak{B}_u^{\bG})_{\psi}$ as a $\widetilde{W}_{\bG}$-module which we denote by $\widetilde{E}$, (see \cite[\S4.1]{shoji:2007:generalized-green-functions-II}). In particular $\widetilde{E}$ is an extension of $E$ and it is our purpose to now describe this extension. Note that there may be several possibilities for such an extension but we recall that in \cite[\S17.2]{lusztig:1986:character-sheaves-IV} Lusztig has given a systematic non-canonical choice for such an extension called the \emph{preferred extension}.
\end{pa}

\begin{prop}\label{prop:frob-action}
Assume $u \in \bG_{\uni}^F$ is a split element and $q \equiv 1\pmod{3}$ if $\bG$ is of type $\E_8$. If $E$ is $F$-stable then the action of $F$ on $\widetilde{E}$ makes this the preferred extension of $E$. Assume $\bG$ is of type $\E_8$ and $q \equiv -1\pmod{3}$ then the same is true unless $u$ is contained in the class $\E_8(b_6)$ in which case the action of $-F$ on $\widetilde{E}$ makes this the preferred extension of $E$, (i.e.\ $-F$ is the identity).
\end{prop}

\begin{proof}
The cases of type $\A_n$, (resp.\ $\E_6$), follows from \cite[Lemma 3.2]{hotta-springer:1977:a-specialization-theorem}, (resp.\ \cite[\S4]{beynon-spaltenstein:1984:green-functions}), together with \cite[Proposition 2.23]{geck-malle:2000:existence-of-a-unipotent-support}. Note that the Springer correspondence we have described in \cref{eq:gen-spring-cor} differs from that in \cite{hotta-springer:1977:a-specialization-theorem} by tensoring with the sign character. The cases of type $\D_n$ with $F$ of order $2$ and $\D_4$ with $F$ of order $3$ are given by \cite[Theorem 4.3(ii)]{shoji:2007:generalized-green-functions-II} and \cite[\S4.24]{shoji:1983:green-polynomials-of-classical-groups} respectively. Finally the case of $\E_8$ follows from the discussion in \cite[\S3, Case (V)]{beynon-spaltenstein:1984:green-functions}.
\end{proof}
%
\section{\texorpdfstring{Bases of the Endomorphism Algebra $\mathcal{A}$}{Bases of the Endomorphism Algebra A}}\label{sec:bases-of-end-A}
\begin{pa}
Assume $\bL \in \mathcal{L}$ is a Levi subgroup supporting a cuspidal pair $(\mathcal{O}_0,\mathscr{E}_0) \in \mathcal{N}_{\bL}^0$ then we denote by $K$ the image of $(\bL,\mathcal{O}_0,\mathscr{E}_0,\Ql)$ under the map in \cref{eq:ind-w-P-construction} and by $\mathcal{A}$ the endomorphism algebra of $K$. As was mentioned in \cref{eq:end-grp-alg-iso} Lusztig has defined an isomorphism between $\mathcal{A}$ and the group algebra $\Ql W_{\bG}(\bL)$ by specifying a set of basis elements $\{\Theta_v \mid v \in W_{\bG}(\bL)\} \subset \mathcal{A}$. However in \cite[\S6.A]{bonnafe:2004:actions-of-rel-Weyl-grps-I} Bonnaf\'{e} has defined an alternative basis $\{\Theta_v' \mid v \in W_{\bG}(\bL)\} \subset \mathcal{A}$ which also defines such an isomorphism. In this article we will need to use both bases and we recall here results of Bonnaf\'{e} concerning the relationship between the two.
\end{pa}

\begin{assumption}
We now assume that $u_0 \in \mathcal{O}_0$ is a fixed unipotent element. If $\bL$ and $\mathcal{O}_0$ are $F$-stable then we assume that $u_0 \in \mathcal{O}_0^F$ is a split unipotent element.
\end{assumption}

\begin{pa}\label{pa:Lusztig-iso-A}
Assume $v \in W_{\bG}(\bL)$ then by \cite[5.4]{bonnafe:2004:actions-of-rel-Weyl-grps-I} we may find a representative $\dot{v} \in N_{\bG}(\bL) \cap C_{\bG}^{\circ}(u_0)$ which we assume fixed. Let us recall the notation of \cref{pa:unip-sup-char-sheaves}. For any $v \in W_{\bG}(\bL)$ we have by \cite[Theorem 9.2(b)]{lusztig:1984:intersection-cohomology-complexes} that there exists an isomorphism $\theta_v : \mathscr{E}_0\boxtimes\Ql \to (\ad\dot{v})^*(\mathscr{E}_0\boxtimes\Ql)$ which we assume fixed. This isomorphism induces an isomorphism $\hat{\theta}_v : \tilde{\mathscr{E}}_0 \to \gamma_v^*\tilde{\mathscr{E}}_0$ where $\gamma_v : \widetilde{Y} \to \widetilde{Y}$ is given by $\gamma_v(g,x\bL) = (g,x\dot{v}^{-1}\bL)$, (see the proof of \cite[Proposition 3.5]{lusztig:1984:intersection-cohomology-complexes}). By definition we have $\gamma_v\gamma = \gamma$ hence $\gamma_*\hat{\theta}_v$ defines an endomorphism of $\gamma_*\tilde{\mathscr{E}}_0$. As $K$ is the intersection cohomology complex $\IC(\overline{Y},\gamma_*\tilde{\mathscr{E}}_0)$ we may then define $\Theta_v$ to be the unique endomorphism of $K$ extending $\gamma_*\hat{\theta}_v$.

By \cite[Theorem 9.2(d)]{lusztig:1984:intersection-cohomology-complexes} there is a unique isomorphism $\theta_v$ for each $v \in W_{\bG}(\bL)$ such that $\Theta_v$ induces the identity on $\mathscr{H}_u^{-\dim Y}(K)$ where $u$ is any element of the induced unipotent class $\Ind_{\bL}^{\bG}(\mathcal{O}_0)$, (see \cite[Corollary 7.3(a)]{lusztig:1984:intersection-cohomology-complexes}). With this choice we have
\begin{equation}\label{eq:algebra-iso}
v \mapsto \Theta_v
\end{equation}
defines the required algebra isomorphism $\Ql W_{\bG}(\bL) \to \mathcal{A}$. Assume now that $\mathscr{L}$ is a local system on $Z^{\circ}(\bL)$. Let $K_{\mathscr{L}}$ be the image of $(\bL,\mathcal{O}_0,\mathscr{E}_0,\mathscr{L})$ under the map in \cref{eq:ind-w-P-construction} and let $\mathcal{A}_{\mathscr{L}}$ be the endomorphism algebra of $K_{\mathscr{L}}$ then from the discussion in \cite[2.3]{lusztig:1986:on-the-character-values} we have an embedding of algebras $\mathcal{A}_{\mathscr{L}} \hookrightarrow \mathcal{A}$ which corresponds under \cref{eq:algebra-iso} to the natural embedding $\Ql W_{\bG}(\bL,\mathscr{L}) \hookrightarrow \Ql W_{\bG}(\bL)$. In particular the restriction of \cref{eq:algebra-iso} to $\Ql W_{\bG}(\bL,\mathscr{L})$ defines the isomorphism mentioned in \cref{eq:end-grp-alg-iso}.
\end{pa}

\begin{pa}\label{pa:Bonnafe-iso-A}
In \cite[\S6.A]{bonnafe:2004:actions-of-rel-Weyl-grps-I} Bonnaf\'{e} shows that for every $v \in W_{\bG}(\bL)$ there exists an isomorphism $\theta_v' : \mathscr{E}_0 \to (\ad\dot{v})^*\mathscr{E}_0$ which induces the identity at the stalk of $u_0$. Clearly this also defines an isomorphism $\mathscr{E}_0\boxtimes\Ql \to (\ad\dot{v})^*(\mathscr{E}_0\boxtimes\Ql)$ which induces the identity at the stalk of any element $u_0z$ where $z\in Z^{\circ}(\bL)$; we will also denote this by $\theta_v'$. As is described in \cref{pa:Lusztig-iso-A} this isomorphism determines a unique endomorphism $\Theta_v'$ of $K$ and by \cite[Proposition 6.1]{bonnafe:2004:actions-of-rel-Weyl-grps-I} we have the map $v\mapsto \Theta_v'$ defines an algebra isomorphism $\Ql W_{\bG}(\bL) \to \mathcal{A}$. The following result describes the relationship between the two sets of basis elements for $\mathcal{A}$.
\end{pa}

\begin{prop}[{}{Bonnaf\'{e}, \cite[Corollary 6.2]{bonnafe:2004:actions-of-rel-Weyl-grps-I}}]\label{prop:bases-of-A}
There exists a linear character $\gamma_{\bL}^{\bG} : W_{\bG}(\bL) \to \Ql$ such that $\Theta_v' = \gamma_{\bL}^{\bG}(v)\Theta_v$ and so $\theta_v' = \gamma_{\bL}^{\bG}(v)\theta_v$ for all $v \in W_{\bG}(\bL)$.
\end{prop}

\begin{pa}
As is remarked in \cite[Remark 6.4]{bonnafe:2004:actions-of-rel-Weyl-grps-I} the linear character $\gamma_{\bL}^{\bG}$ is known explicitly in almost all cases. However we will see that, in the situation we consider in this article, the contribution from $\gamma_{\bL}^{\bG}$ will always be trivial.
\end{pa}
%
\section{Rational Structures}\label{sec:rational-structures}
\begin{pa}\label{pa:characteristic-function}
We say a character sheaf $A \in \widehat{\bG}$ is $F$-stable if we have an isomorphism $\phi_A : F^*A \to A$ in $\mathscr{D}\bG$ and we denote by $\widehat{\bG}^F \subseteq \widehat{\bG}$ the set of all $F$-stable character sheaves of $\bG$. Clearly this isomorphism induces a map $\mathscr{H}^iF^*A \to \mathscr{H}^iA$ for each $i \in \mathbb{Z}$ and hence a map $\mathscr{H}^i_xF^*A \to \mathscr{H}^i_xA$ for each $x \in \bG$; we denote both these maps again by $\phi_A$. If $x \in \bG^F$ then $\phi_A$ induces an automorphism of the stalk $\mathscr{H}^i_xA$ because $\mathscr{H}^i_xF^*A = \mathscr{H}^i_{F(x)}A = \mathscr{H}^i_xA$. For each such $A$ and $\phi_A$ we then define the \emph{characteristic function} $\chi_{A,\varphi_A} : \bG^F \to \overline{\mathbb{Q}}_{\ell}$ by setting
\begin{equation*}
\chi_{A,\phi_A}(g) = \sum_i (-1)^i\Tr(\phi_A, \mathscr{H}^i_gA)
\end{equation*}
for all $g \in \bG^F$. Note that this characteristic function depends upon the choice of isomorphism $\phi_A$. If $A$ is an element of $\widehat{\bG}^F$ then we will choose $\phi_A$ to satisfy the properties in \cite[\S25.1]{lusztig:1986:character-sheaves-V}. With such a choice the resulting characteristic function $\chi_{A,\phi_A}$ has norm 1; the choice of $\phi_A$ is unique up to scalar multiplication by a root of unity.
\end{pa}

\begin{pa}\label{pa:F-stable-pairs}
The Frobenius endomorphism acts on the set $\mathcal{N}_{\bG}$ by $\iota \mapsto F^{-1}(\iota) := (F^{-1}(\mathcal{O}_{\iota}),F^*\mathscr{E}_{\iota})$, where here $F^*\mathscr{E}_{\iota}$ is the inverse image of $\mathscr{E}_{\iota}$ under $F$. We say $\iota$ is $F$-stable if $F^{-1}(\mathcal{O}_{\iota}) = \mathcal{O}_{\iota}$ and $F^*\mathscr{E}_{\iota}$ is isomorphic to $\mathscr{E}_{\iota}$, (we also denote this by $\iota = F^{-1}(\iota)$), and denote the subset of all $F$-stable pairs by $\mathcal{N}_{\bG}^F$. The Frobenius also acts on the set $\mathcal{M}_{\bG}$ by $[\bL,\nu] \mapsto [F^{-1}(\bL), F^{-1}(\nu)]$. We say $[\bL,\nu]$ is $F$-stable if $[F^{-1}(\bL),F^{-1}(\nu)] = [\bL,\nu]$ and denote the subset of all $F$-stable pairs by $\mathcal{M}_{\bG}^F$. As $\bG/Z^{\circ}(\bG)$ is simple we see by the classification of cuspidal local systems, (see \cite[\S10-15]{lusztig:1984:intersection-cohomology-complexes}), that every standard pair $(\bL,\bP) \in \mathcal{Z}_{\std}$ such that $\mathcal{N}_{\bL}^0 \neq \emptyset$ satisfies $F(\bL) = \bL$ and $F(\bP) = \bP$. In particular, as $p$ is good we have $\mathcal{N}_{\bL}^0$ has cardinality 1 which implies $\mathcal{M}_{\bG}^F = \mathcal{M}_{\bG}$.
\end{pa}

\subsection{Twisted Levi Subgroups}
\begin{pa}\label{lem:G-split-levi}
We now wish to choose for each $F$-stable unipotently supported character sheaf $A \in \widehat{\bG}^F$ a distinguished isomorphism $\phi_A : F^*A \to A$. This was done by Lusztig in  \cite[\S3.3 - \S3.4]{lusztig:1986:on-the-character-values} under the assumption that $\bG$ is an adjoint simple group and $F$ is a split Frobenius endomorphism. Adapting an idea of Lusztig from \cite[\S10.3]{lusztig:1985:character-sheaves} we will extend this to the case where $F$ is not necessarily split and $\bG$ has a connected centre. Firstly let us assume that $(\bL,\bP) \in \mathcal{Z}_{\std}$ is a standard pair such that $\mathcal{N}_{\bL}^0 \neq \emptyset$ and $A_0 = A_{\nu}^{\mathscr{L}} \in \widehat{\bL}$ is a cuspidal character sheaf satisfying $(A:\ind_{\bL}^{\bG}A_0) \neq 0$, (see \cref{pa:induction-char-sheaves,lem:F-action-ind}). Using the argument in \cite[\S10.5]{lusztig:1985:character-sheaves} together with \cref{lem:F-action-ind} we see that there exists $x \in \bG$ such that ${}^x\bL$ is $F$-stable and $B_0 = (\ad x^{-1})^*A_0 \in \widehat{{}^x\bL}$ is an $F$-stable cuspidal character sheaf satisfying $(A:\ind_{{}^x\bL}^{\bG}B_0) \neq 0$, (we may then assume that $A$ is a summand of $\ind_{{}^x\bL}^{\bG}B_0$).
\end{pa}

\begin{pa}\label{pa:conj-levi-etc}
Let us assume that $(\bL,\bP) \in \mathcal{Z}_{\std}$ is a standard pair such that $(\mathcal{O}_0,\mathscr{E}_0) \in \mathcal{N}_{\bL}^0 \neq \emptyset$ and $\mathscr{L} \in \mathcal{S}(\bT_0)$ is a local system. For every $v \in W_{\bG}(\bL)$ we fix an element $g_v \in \bG$ such that $g_v^{-1}F(g_v) = F(\dot{v}^{-1})$ where $\dot{v} \in N_{\bG}(\bL)$ is as in \cref{pa:Lusztig-iso-A}. If $v$ is the identity in $W_{\bG}(\bL)$ then we will assume for convenience that $\dot{v} = g_v$ is the identity in $\bG$. We now define
\begin{equation*}
\bL_v = {}^{g_v}\bL \qquad \bP_v = {}^{g_v}\bP \qquad \mathcal{O}_v = {}^{g_v}\mathcal{O}_0 \qquad \Sigma_v = \mathcal{O}_vZ^{\circ}(\bL_v) \qquad \mathscr{E}_v = (\ad g_v^{-1})^*\mathscr{E}_0 \qquad \mathscr{L}_v = (\ad g_v^{-1})^*\mathscr{L}.
\end{equation*}
Through the isomorphism $\ad g_v^{-1} : \bL_v \to \bL$ we may identify the action of the Frobenius endomorphism $F$ on $\bL_v$ with the action of $F_v$ on $\bL$ defined by $F_v(l) = F(\dot{v}^{-1}l\dot{v})$ for all $l \in \bL$. Note that by \cite[Theorem 9.2(b)]{lusztig:1984:intersection-cohomology-complexes} and the remarks in \cref{pa:F-stable-pairs} we have $\mathcal{O}_v$ is $F$-stable and $F^*\mathscr{E}_v \cong \mathscr{E}_v$. Furthermore we have an isomorphism of abstract groups $\ad g_v^{-1} : \bL_v^F \to \bL^{F_v}$ and $(\ad g_v^{-1})^*$ induces a bijective correspondence between $\widehat{\bL}^{F_v}$ and $\widehat{\bL}_v^F$. In particular, assume $K \in \widehat{\bL}^{F_v}$ and $\psi : F_v^*K \to K$ is a fixed isomorphism then $\psi' : F^*K' \to K'$ is also an isomorphism where $\psi' = (\ad g_v^{-1})^*\psi$ and $K' = (\ad g_v^{-1})^*K$. Note that from the definitions we have the corresponding characteristic functions are related by $\chi_{K',\psi'}= \chi_{K,\psi} \circ \ad g_v^{-1}$.

Let us denote by $D_v$ the IC sheaf $\IC(\overline{\Sigma}_v,\mathscr{E}_v\boxtimes\mathscr{L}_v)[\dim \Sigma_v]$, (see \cref{eq:cusp-uni-sup}). From the discussion in \cref{lem:G-split-levi} we need only concern ourselves with summands of $\ind_{\bL_v}^{\bG}(D_v)$ where $D_v$ is $F$-stable. Clearly we have
\begin{equation*}
F^*D_v = \IC(\overline{\Sigma}_v,F^*\mathscr{E}_v\boxtimes F^*\mathscr{L}_v)[\dim \Sigma_v]
\end{equation*}
which implies that $F^*D_v \cong D_v$ if and only if $F^*\mathscr{L}_v \cong \mathscr{L}_v$. In particular we will assume that the local system $\mathscr{L}$ is such that the subset
\begin{equation*}
Z_{\mathscr{L}} = \{n \in N_{\bG}(\bL) \mid F(n\Sigma n^{-1})\text{ and }\ad(n)^*F^*\mathscr{L} \cong \mathscr{L}\}/\bL \subset W_{\bG}(\bL)
\end{equation*}
is non-empty. Note that $Z_{\mathscr{L}}$ is in fact a right coset of the group $W_{\bG}(\bL,\mathscr{L})$ in $W_{\bG}(\bL)$. If $\bS \leqslant \bG$ is any $F$-stable torus and $\mathscr{F} \in \mathcal{S}(\bS)^F$ is any $F$-stable local system then we may choose a distinguished isomorphism $\beta : F^*\mathscr{F} \to \mathscr{F}$ by the condition that the morphism induced by $\beta$ at the stalk of the identity element $1 \in \bS^F$ is the identity morphism. In particular this defines an isomorphism $\varphi_1^v : F^*\mathscr{L}_v \to \mathscr{L}_v$ for each $v \in Z_{\mathscr{L}}$.
\end{pa}

\subsection{Relative Weyl Groups}
\begin{pa}\label{pa:generators-rel-Weyl}
If $\bS$ is a torus then we denote by $X(\bS)$ the character group of $\bS$. For the following we refer to \cite[Proposition 1.12]{bonnafe:2004:actions-of-rel-Weyl-grps-I} and the references therein, (in particular \cite[Theorem 6]{howlett:1980:normalizers-of-parabolic-subgroups}). Let $\Phi \subset X(\bT_0)$ denote the roots of $\bG$ with respect to $\bT_0$ and $\Delta \subset \Phi^+$ the simple and positive systems of roots determined by $\bB_0$. Let $\Delta_{\bL} \subset \Delta$ be the set of simple roots determined by $\bL$ then we denote by $\widetilde{\Delta}_{\bL}$ the set $\Delta - \Delta_{\bL}$. For each $\alpha \in \widetilde{\Delta}_{\bL}$ let $\bM_{\alpha}$ be the Levi subgroup of $\bG$ determined by $\Delta_{\bL}\cup\{\alpha\}$ then the group $W_{\bM_{\alpha}}(\bL)$ contains a unique non-trivial element which we denote $s_{\bL,\alpha}$. For every $\alpha \in \widetilde{\Delta}_{\bL}$ we have $s_{\bL,\alpha}$ acts on $X(Z^{\circ}(\bL))$ as a reflection sending $\alpha$ to $-\alpha$. Taking $\widetilde{\mathbb{I}} = \{s_{\bL,\alpha} \mid \alpha \in \widetilde{\Delta}_{\bL}\}$ we have $(W_{\bG}(\bL),\widetilde{\mathbb{I}})$ is a Coxeter system with corresponding root system $\widetilde{\Phi}_{\bL} \subset X(Z^{\circ}(\bL))$ obtained as the image of $\widetilde{\Delta}_{\bL}$ under the action of $W_{\bG}(\bL)$. Note that if $\widetilde{\Phi}_{\bL}^+\subset\widetilde{\Phi}_{\bL}$ is the positive system of roots determined by $\widetilde{\Delta}_{\bL}$ then we have $\widetilde{\Phi}_{\bL}^+ = \widetilde{\Phi}_{\bL} \cap \Phi^+$.

Now there exists a root system $\Psi \subset \widetilde{\Phi}_{\bL}$ such that $W_{\bG}(\bL,\mathscr{L})$ is the reflection group of $\Psi$, (note that $\Psi$ is not necessarily closed in $\widetilde{\Phi}_{\bL}$). Taking $\Psi^+ = \Psi\cap\widetilde{\Phi}_{\bL}^+$ we have $\Psi^+$ is a system of positive roots in $\Psi$ which determines a unique set of Coxeter generators $\mathbb{J} \subseteq W_{\bG}(\bL,\mathscr{L})$. Assume $n \in N_{\bG}(\bL,\mathscr{L})$ then for any $v \in Z_{\mathscr{L}}$ we have an isomorphism
\begin{equation*}
(\ad F_v(n))^*\mathscr{L}\cong (\ad F_v(n))^*(F_v^{-1})^*\mathscr{L} = (F_v^{-1})^*(\ad n)^*\mathscr{L} \cong (F_v^{-1})^*\mathscr{L} \cong \mathscr{L},
\end{equation*}
in particular as $\bL$ is $F_v$-stable this shows that $F_v$ induces an automorphism $F_v : W_{\bG}(\bL,\mathscr{L}) \to W_{\bG}(\bL,\mathscr{L})$. By \cite[Lemma 1.9(i)]{lusztig:1984:characters-of-reductive-groups} the coset $Z_{\mathscr{L}}$ contains a unique element $w^{-1}$ of minimal length, (with respect to the length function of $(W_{\bG}(\bL),\widetilde{\mathbb{I}})$), satisfying the property $w^{-1}\Psi^+w \subset \Phi^+$. By \cite[I - (5.16.1), II - (4.2.1)]{shoji:1995:character-sheaves-and-almost-characters} we have $\dot{v}\bB_{\bL}\dot{v}^{-1} = \bB_{\bL}$ for all $v \in Z_{\mathscr{L}}$ where $\bB_{\bL} = \bB_0\cap\bL$. In particular, (as $\bB_0$ and $\bL$ are both $F$-stable), this shows that $F_w$ is an automorphism of the Coxeter system $(W_{\bG}(\bL,\mathscr{L}),\mathbb{J})$. Setting $\mathbb{J}_w = \{g_wxg_w^{-1} \mid x \in \mathbb{J}\}$ we have $(W_{\bG}(\bL_w,\mathscr{L}_w),\mathbb{J}_w)$ is a Coxeter system and $F$ induces an automorphism of this Coxeter system.
\end{pa}

\begin{assumption}
From now until the end of this article $w \in W_{\bG}(\bL)$ will always denote the unique element such that $w^{-1} \in Z_{\mathscr{L}}$ has minimal length.
\end{assumption}

\subsection{Isomorphisms for Local Systems}
\begin{pa}\label{pa:isomorphism-cuspidal-case}
We now consider how to choose an isomorphism $F^*\mathscr{E}_v \to \mathscr{E}_v$ for every $v \in W_{\bG}(\bL)$. First we will choose an isomorphism $\varphi_0 : F^*\mathscr{E}_0 \to \mathscr{E}_0$ in the following way. The unipotent conjugacy class $\mathcal{O}_0$ supporting the cuspidal local system $\mathscr{E}_0$ contains a split element $u \in \bL^F$, (with respect to the Frobenius endomorphism $F$), which is unique up to $\bL^F$-conjugacy, (note that $\mathcal{O}_0$ cannot be the class $\E_8(b_6)$). We now choose the isomorphism $\varphi_0$ by the requirement that the induced isomorphism $(\mathscr{E}_0)_u \to (\mathscr{E}_0)_u$ at the stalk of the fixed split element $u_0 \in \mathcal{O}_0^F$ is $q^{1/2(\dim(\bL/Z^{\circ}(\bL)) - \dim\mathcal{O}_0)}$ times the identity.

Following \cite[9.3]{lusztig:1990:green-functions-and-character-sheaves} we define an isomorphism $\varphi_0^v : F^*\mathscr{E}_v \to \mathscr{E}_v$, for every $v \in W_{\bG}(\bL)$, in the following way. Recall that the basis element $\theta_{F(v)}$ of $\mathcal{A}$, (c.f.\ \cref{pa:endomorphism-algebra-A}), defines an isomorphism $\theta_{F(v)} : \mathscr{E}_0 \to (\ad F(\dot{v}))^*\mathscr{E}_0$. Taking the inverse image of $\theta_{F(v)}$ along $\ad g_v^{-1}\circ F$ we see that we obtain a new isomorphism
\begin{equation*}
(\ad g_v^{-1}\circ F)^*\theta_v : F^*\mathscr{E}_v \to (\ad g_v^{-1})^*F^*\mathscr{E}_0.
\end{equation*}
Moreover the isomorphism $\varphi_0$ induces an isomorphism
\begin{equation*}
(\ad g_v^{-1})^*\varphi_0 : (\ad g_v^{-1})^*F^*\mathscr{E}_0 \to \mathscr{E}_v.
\end{equation*}
We will now take $\varphi_0^v$ to be the composition $(\ad g_v^{-1})^*\varphi_0\circ(\ad g_v^{-1}\circ F)^*\theta_v$, (note that by our assumption on $g_v$ we have $\varphi_0^v = \varphi_0$ if $v=1$). With this in mind we can define for each $v \in Z_{\mathscr{L}}$ a unique isomorphism $\phi_{D_v} : F^*D_v \to D_v$ with the property that the restriction of $\phi_{D_v}$ to $\Sigma_v$ coincides with $\varphi_0^v\boxtimes\varphi_1^v$.
\end{pa}

\begin{pa}\label{pa:iso-K-v}
Assume $v \in Z_{\mathscr{L}}$ then we will denote by $K_v$ the image of $(\bL_v,\mathcal{O}_v,\mathscr{E}_v,\mathscr{L}_v)$ under the map in \cref{eq:ind-w-P-construction}. We now consider how to choose an isomorphism $F^*K_v \to K_v$. Recalling the construction in \cref{pa:unip-sup-char-sheaves} we see that the isomorphisms $\varphi_0^v : F^*\mathscr{E}_v \to \mathscr{E}_v$ and $\varphi_1^v : F^*\mathscr{L}_v \to \mathscr{L}_v$ naturally induce isomorphisms $\tilde{\varphi}_0^v : F^*\tilde{\mathscr{E}}_v \to \tilde{\mathscr{E}}_v$ and $\tilde{\varphi}_1^v : F^*\tilde{\mathscr{L}}_v \to \tilde{\mathscr{L}}_v$ hence an isomorphism $\tilde{\varphi}_0^v\otimes\tilde{\varphi}_1^v : F^*\tilde{\mathscr{E}}_v\otimes F^*\tilde{\mathscr{L}}_v \to \tilde{\mathscr{E}}_v\otimes\tilde{\mathscr{L}}_v$. Clearly the variety $Y$ of \cref{pa:unip-sup-char-sheaves} is stable under $F$, (because $\bL_v$ is stable under $F$), so we have
\begin{equation*}
F^*K_v = \IC(\overline{Y},F^*\gamma_*(\tilde{\mathscr{E}}_v\otimes\tilde{\mathscr{L}}_v))[\dim Y] = \IC(\overline{Y},\gamma_*F^*(\tilde{\mathscr{E}}_v\otimes\tilde{\mathscr{L}}_v))[\dim Y]
\end{equation*}
because $\gamma_* = \gamma_!$ and $F^*\gamma_! = \gamma_!F^*$, (see \cref{pa:unip-sup-char-sheaves} and \cref{pa:direct-image-commutative}). We may now define a unique isomorphism $\phi : F^*K_v \to K_v$ by specifying that $\phi|_Y$ coincides with $\gamma_*(\tilde{\varphi}_0^v\otimes\tilde{\varphi}_1^v)$.
\end{pa}

\begin{pa}\label{pa:iso-K-v-compare}
Assume $\mathscr{L} = \Ql$ and let $K$ be the image of $(\bL,\mathcal{O}_0,\mathscr{E}_0,\Ql)$ under the map in \cref{eq:ind-w-P-construction} then for each $v \in W_{\bG}(\bL)$ we wish to consider an isomorphism $\phi_v : F^*K \to K$ which we will see is related to the isomorphism $\psi_v : F^*K_v \to K_v$ defined in \cref{pa:iso-K-v}. Recall that in \cref{pa:isomorphism-cuspidal-case,pa:iso-K-v} we have defined isomorphisms $\phi : F^*K\to K$, $\varphi_0 : F^*\mathscr{E}_0 \to \mathscr{E}_0$, $\varphi_1 : F^*\Ql \to \Ql$ and $\tilde{\varphi}_0 : F^*\tilde{\mathscr{E}}_0 \to \tilde{\mathscr{E}}_0$, (we simply set $v=1$ in the constructions). We obtain isomorphisms $\varphi_{v,0} := \theta_{F(v)}(\varphi_0\boxtimes\varphi_1) : F^*(\mathscr{E}_0\boxtimes\Ql) \to (\ad v)^*(\mathscr{E}_0\boxtimes
\Ql)$ and $\tilde{\varphi}_{v,0} := \tilde{\theta}_{F(v)}\tilde{\varphi}_0 : F^*\tilde{\mathscr{E}}_0 \to \gamma_v^*\tilde{\mathscr{E}}_0$ by composing with the endomorphism defined in \cref{pa:Lusztig-iso-A}. As in \cref{pa:iso-K-v} we define the isomorphism $\phi_v$ to be the unique extension of the isomorphism $\gamma_*\tilde{\varphi}_{v,0}$. It is clear from the construction that we have $\phi_v = \Theta_{F(v)}\phi$ and furthermore using the exact same argument as in \cite[(10.6.1)]{lusztig:1985:character-sheaves} we have for each $g \in G$ and $i \in \mathbb{Z}$ that
\begin{equation}\label{eq:trace-comparison}
\Tr(\phi_v,\mathscr{H}_g^iK) = \Tr(\Theta_{F(v)}\phi,\mathscr{H}_g^iK) = \Tr(\psi_v,\mathscr{H}_g^iK_v)
\end{equation}
\end{pa}

\subsection{Isomorphisms for Character Sheaves}
\begin{pa}\label{pa:iso-on-arbitrary-char-sheaf}
Let $w^{-1} \in Z_{\mathscr{L}}$ be the unique element of minimal length, (c.f.\ \cref{pa:generators-rel-Weyl}). Furthermore let $K$ be the image of $(\bL_w,\mathcal{O}_w,\mathscr{E}_w,\mathscr{L}_w)$ under the map in \cref{eq:ind-w-P-construction} and $\phi : F^*K \to K$ the isomorphism defined in \cref{pa:iso-K-v}. By \cref{prop:K-nu-L-iso-to-induction} we may assume that any summand $A$ of $\ind_{\bL_w}^{\bG}(D_w)$ is a summand of $K \cong \ind_{\bL_w}^{\bG}(D_w)$. In particular if $\mathcal{A}$ is the endomorphism algebra of $K$ then $A = K_E$ for some simple $\mathcal{A}$-module $E$, (c.f.\ \cref{pa:decomp-by-end-algebra}). By adapting the construction in \cite[\S10.3]{lusztig:1985:character-sheaves} we will show how $\phi_A$ is determined from $\phi$. Let us denote by $\sigma : \mathcal{A} \to \mathcal{A}$ the algebra automorphism given by $\sigma(\theta) = \phi \circ F^*\theta \circ \phi^{-1}$ where $F^*\theta : F^*K \to F^*K$ is the induced map. For any $\mathcal{A}$-module $E$ we denote by $E_{\sigma}$ the module obtained from $E$ by twisting the action with $\sigma^{-1}$, (i.e.\ if $\cdot$ denotes the original action then we define a new action $\star$ by setting $\theta\star e = \sigma^{-1}(\theta)\cdot e$). For each such $E$ we then have an induced isomorphism of $\mathcal{A}$-modules
\begin{equation*}
F^*A = \Hom_{\mathcal{A}}(E,F^*K) \to \Hom_{\mathcal{A}}(E_{\sigma},K)
\end{equation*}
given by $f \mapsto \phi\circ f$. Here $F^*K$ is an $\mathcal{A}$-module under the action $\theta\cdot k = (F^*\theta)(k)$ for all $k \in F^*K$ and the first equality is seen to hold by the construction given in \cref{eq:diag-Hom-MG}. Assume we have an $\mathcal{A}$-module isomorphism $\psi_E : E \to E_{\sigma}$ then $\phi_A : F^*K_E \to K_E$, given by $\phi_A(f) = \phi\circ f \circ \psi_E$, is an isomorphism of $\mathcal{A}$-modules which defines an isomorphism in $\mathscr{D}\bG$ under the construction in \cref{eq:diag-Hom-MG}. We now need only observe that when $K_E$ is simple, hence $E$ is simple, all such isomorphisms occur in this way. To see this note that any non-zero $f \in \Hom_{\mathcal{A}}(E,K)$ gives an isomorphism $E \cong \Image(f)$, because $E$ is a simple module, so $\psi_E = f^{-1}\circ\phi^{-1}\circ\phi_A\circ f$ is determined by $\phi_A$. This discussion shows that choosing the isomorphism $\phi_A$ is equivalent to choosing the isomorphism $\psi_E$.

Let us recall Lusztig's basis $\{\Theta_v \mid v \in W_{\bG}(\bL_w,\mathscr{L}_w)\}$ for the endomorphism algebra $\mathcal{A}$, (c.f.\ \cref{pa:Lusztig-iso-A}). From their definition one may readily check that the basis elements of $\mathcal{A}$ satisfy $\sigma(\Theta_v) = \Theta_{F^{-1}(v)}$ for all $v \in W_{\bG}(\bL_w,\mathscr{L}_w)$. Let $\widetilde{W}_{\bG}(\bL_w,\mathscr{L}_w)$ denote the semidirect product $W_{\bG}(\bL_w,\mathscr{L}_w)\rtimes \langle F\rangle$ where $\langle F\rangle$ is the finite cyclic group generated by the automorphism $F$. Let $E$ be a $W_{\bG}(\bL_w,\mathscr{L}_w)$-module then we can extend $E$ to a $\widetilde{W}_{\bG}(\bL_w,\mathscr{L}_w)$-module $\widetilde{E}$ by letting $F^{-1}$ act as $\psi_E$ under the isomorphism in \cref{eq:end-grp-alg-iso}.
\end{pa}

\begin{assumption}
We will now assume that $\psi_E$ is chosen such that $\widetilde{E}$ is Lusztig's preferred extension of $E$ defined in \cite[\S17.2]{lusztig:1986:character-sheaves-IV}.
\end{assumption}

\subsection{Describing Isomorphisms on Local Systems at Split Elements}
\begin{pa}\label{pa:iso-arbitrary-pair}
Before we continue we define here two integer values associated to a pair $\iota \in \mathcal{N}_{\bG}$. Recall that $[\bL_{\iota},\nu_{\iota}] \in \mathcal{M}_{\bG}$ is the orbit such that $\iota \in \mathscr{I}[\bL_{\iota},\nu_{\iota}]$, (c.f.\ \cref{pa:gen-spring-cor}). Let us denote $\nu_{\iota}$ by $(\mathcal{O}_0,\mathscr{E}_0)$ then we attach to $\iota$ the following two values
\begin{align*}
a_{\iota} &= -\dim\mathcal{O}_{\iota} - \dim Z^{\circ}(\bL_{\iota}),\\
b_{\iota} &= (\dim\bG -\dim\mathcal{O}_{\iota}) - (\dim\bL_{\iota} - \dim\mathcal{O}_0).
\end{align*}
We now consider how the above isomorphisms determine a canonical isomorphism $F^*\mathscr{E}_{\iota} \to \mathscr{E}_{\iota}$ for an arbitrary pair $\iota = (\mathcal{O}_{\iota},\mathscr{E}_{\iota}) \in \mathcal{N}_{\bG}^F$. Recall from \cite[Theorem 6.5(c)]{lusztig:1984:intersection-cohomology-complexes} that the perverse sheaf $K_{\iota}$ is such that $\mathscr{H}^{a_{\iota}}(K_{\iota})|_{\mathcal{O}_{\iota}} \cong \mathscr{E}_{\iota}$. In \cref{pa:iso-on-arbitrary-char-sheaf} we have defined a canonical isomorphism $\phi_{\iota} : F^*K_{\iota} \to K_{\iota}$ and this induces an isomorphism $F^*\mathscr{H}^{a_{\iota}}(K_{\iota})|_{\mathcal{O}_{\iota}} \to \mathscr{H}^{a_{\iota}}(K_{\iota})|_{\mathcal{O}_{\iota}}$ hence an isomorphism $\varphi_{\iota} : F^*\mathscr{E}_{\iota} \to \mathscr{E}_{\iota}$. To do computations we will need the following explicit result concerning $\varphi_{\iota}$ which follows the same line of argument as \cite[3.4]{lusztig:1986:on-the-character-values} and \cite[Lemma 3.6]{shoji:1997:unipotent-characters-of-finite-classical-groups}.
\end{pa}

\begin{prop}\label{prop:F-stable-pair-loc-sys-iso}
If $\bG$ is of type $\E_8$ then assume $q \equiv 1\pmod{3}$. For any $F$-stable pair $\iota \in \mathcal{N}_{\bG}^F$ and $v \in W_{\bG}(\bL_{\iota})$ the map $(\mathscr{E}_{\iota})_u \to (\mathscr{E}_{\iota})_u$ induced by $\Theta_v'\phi_{\iota}$ for some split element $u \in \mathcal{O}_{\iota}^F$ is $q^{(\dim\bG + a_{\iota})/2}$ times the identity. If $\bG$ is of type $\E_8$ and $q \equiv -1\pmod{3}$ then the same is true unless $\mathcal{O}_{\iota}$ is the class $\E_8(b_6)$ in which case the induced map is $-q^{(\dim\bG+a_{\iota})/2}$ times the identity.
\end{prop}

\begin{proof}
Let $K$ be the image of $(\bL,\mathcal{O}_0,\mathscr{E}_0,\Ql)$ under the map in \cref{eq:ind-w-P-construction} and let $\phi : F^*K \to K$ be the canonical isomorphism defined above and $\varphi_u$ the isomorphism $\mathscr{H}^{a_{\iota}}_u(K) \to \mathscr{H}^{a_{\iota}}_u(K)$ induced by $\Theta_v'\phi$. First of all let us note that under the isomorphism in \cref{eq:K-nu-L-decomp} we have
\begin{equation*}
F^*K \cong \bigoplus (E\otimes F^*K_E) \qquad\qquad \mathscr{H}_u^{a_{\iota}}(K) \cong \bigoplus (E \otimes \mathscr{H}_u^{a_{\iota}}(K_E)).
\end{equation*}
Both of these isomorphisms are given in the same way as \cref{eq:K-nu-L-decomp} by noticing that $F^*K_E = \Hom_{\mathcal{A}}(E,F^*K)$ and $\mathscr{H}_u^{a_{\iota}}(K_E) = \Hom_{\mathcal{A}}(E,\mathscr{H}_u^{a_{\iota}}(K))$. Assume $E \in \Irr(W_{\bG}(\bL))^F$ then under the isomorphism \cref{eq:K-nu-L-decomp} the restriction of $\Theta_v'\phi$ to the summand isomorphic to $E\otimes K_E$ corresponds to the isomorphism $\psi_E^{-1} \otimes (\Theta_v'\phi_A)$, (where $\psi_E$ and $\phi_A$ are as in \cref{pa:iso-on-arbitrary-char-sheaf}). With this in mind it is enough to show that $\varphi_u$ acts on an $\mathcal{A}$-submodule of $\mathscr{H}_u^{a_{\iota}}(K)$ isomorphic to $E$ as $q^{(\dim\bG+a_{\iota})/2}$ times $\psi_E^{-1}$, (resp.\ $-\psi_E^{-1}$ if $\mathcal{O}_{\iota}$ is $\E_8(b_6)$ and $q \equiv -1\pmod{3}$).

To prove this we follow the argument of \cite[Lemma 3.6]{shoji:1997:unipotent-characters-of-finite-classical-groups} which in turn is a modification of the arguments in \cite{lusztig:1986:on-the-character-values}. Let $\bP \leqslant \bG$ be a parabolic subgroup of $\bG$ such that $\bP = \bL_{\iota}\bU_{\bP}$ is a Levi decomposition of $\bP$. As $\iota \in \mathcal{N}_{\bG}^F$ is $F$-stable we may assume $\bL_{\iota}$ and $\bP$ are $F$-stable. We define $Z_u$ to be the variety $\{x\bP \in \bG/\bP \mid x^{-1}ux \in \mathcal{O}_0\bU_{\bP}\}$ and similarly we define $\widehat{Z}_u$ to be the variety $\{x \in \bG \mid x^{-1}ux \in \mathcal{O}_0\bU_{\bP}\}$. We have two natural morphisms $\psi_u : \widehat{Z}_u \to Z_u$, resp.\ $\lambda_u : \widehat{Z}_u \to \mathcal{O}_0$, given by $\psi_u(x) = x\bP$ and $\lambda_u(x) = (\mathcal{O}_0$-component of $x^{-1}ux \in \mathcal{O}_0\bU_{\bP})$. We then define a local system $\widehat{\mathscr{E}}_{\iota}$ on $Z_u$ by the condition that $\psi_u^*\widehat{\mathscr{E}}_{\iota} = \lambda_u^*\mathscr{E}_0$.

By the discussion in \cite[\S6.3]{lusztig:1992:a-unipotent-support} it follows that $b_{\iota} = a_{\iota} + \dim\supp K_{\iota}$ hence by \cite[24.2.5]{lusztig:1986:character-sheaves-V} we have an isomorphism $\Phi : \mathscr{H}^{a_{\iota}}_u(K) \to H_c^{b_{\iota}}(Z_u,\widehat{\mathscr{E}}_{\iota})$, (where this is the cohomology with compact support of $Z_u$ with coefficients in the local system $\widehat{\mathscr{E}}_{\iota}$). Note that we have a commutative diagram
\begin{center}
\begin{tikzcd}
Z_u \arrow{d}{\pi_v} & \widehat{Z}_u \arrow{d}{\hat{\pi}_v}\arrow{l}[swap]{\psi_u}\arrow{r}{\lambda_u} & \mathcal{O}_0 \arrow{d}{\ad \dot{v}}\\
Z_u & \widehat{Z}_u \arrow{l}{\psi_u}\arrow{r}[swap]{\lambda_u} & \mathcal{O}_0
\end{tikzcd}
\end{center}
where $\pi_v$ and $\hat{\pi}_v$ are defined by $\pi_v(x) = x\dot{v}^{-1}$ and $\hat{\pi}_v(x\bP) = x\dot{v}^{-1}\bP$. Let $\varphi_0 : F^*\mathscr{E}_0 \to \mathscr{E}_0$ be as in \cref{pa:isomorphism-cuspidal-case} then the composition $\theta_v'\varphi_0 : F^*\mathscr{E}_0 \to (\ad\dot{v})^*\mathscr{E}_0$ is an isomorphism such that the induced isomorphism at the stalk of our fixed split element $u_0 \in \mathcal{O}_0^F$ is the identity, (c.f.\ \cref{pa:Bonnafe-iso-A}). This isomorphism induces an isomorphism $\widehat{\varphi}_0 : F^*\widehat{\mathscr{E}}_{\iota} \to \pi_v^*\widehat{\mathscr{E}}_{\iota}$ which in turn induces a linear map $\widehat{\varphi}_0$ of $H_c^{b_{\iota}}(Z_u,\widehat{\mathscr{E}}_{\iota})$ satisfying $\Phi\circ\varphi_u = \widehat{\varphi}_0\circ\Phi$. With this we see that we are left with showing that $\widehat{\varphi}_0$ acts on an $\mathcal{A}$-submodule of $H_c^{b_{\iota}}(Z_u,\widehat{\mathscr{E}}_{\iota})$ isomorphic to $E$ as $q^{(\dim\bG+a_{\iota})/2}$ times $\psi_E^{-1}$, (resp.\ $-\psi_E^{-1}$ if $\mathcal{O}_{\iota}$ is $\E_8(b_6)$ and $q \equiv -1\pmod{3}$).

Assume $\bL_{\iota} = \bT_0$ so that $\iota$ is in the principal block then $(\mathcal{O}_0,\mathscr{E}_0) = (\{1\},\Ql)$ and $Z_u$ can canonically be identified with $\mathfrak{B}_u^{\bG}$. Furthermore we have $\widehat{\mathscr{E}}_{\iota}$ is simply the constant sheaf and $\widehat{\varphi}_0$ induces the identity at every stalk of $\widehat{\mathscr{E}}_{\iota}$. Hence the statement is simply that of \cref{prop:frob-action} so we are done in this case. The case where $\bL_{\iota} = \bG$, (i.e.\ $\iota$ is cuspidal), is trivial so we are left only with the case where $\bL_{\iota}$ is neither $\bG$ nor $\bT_0$. For this to be the case we must have $\bG$ is of type $\B_n$, $\C_n$ or $\D_n$ but these cases are dealt with by Shoji in \cite[Theorem 4.3]{shoji:2007:generalized-green-functions-II}, (see also the reduction arguments given in \cite[1.5]{shoji:2006:generalized-green-functions-I}). Note that to apply this theorem we need the fact that $\theta_v'\varphi_0$ induces the identity at the stalk of the split element $u_0 \in \mathcal{O}_0^F$.
\end{proof}

\begin{assumption}
From now on the following assumption is in place. If $A \in \widehat{\bG}^F$ is a unipotently supported character sheaf then we assume the isomorphism $\phi_A : F^*A \to A$ to be chosen as in \cref{pa:iso-on-arbitrary-char-sheaf}, otherwise we assume $\phi_A$ is any isomorphism satisfying the conditions mentioned in \cref{pa:characteristic-function}. If $\iota \in \mathcal{N}_{\bG}^F$ is any $F$-stable pair then we assume the isomorphism $\varphi_{\iota} : F^*\mathscr{E}_{\iota} \to \mathscr{E}_{\iota}$ to be chosen as in \cref{pa:iso-arbitrary-pair}, in particular the isomorphism is described explicitly by \cref{prop:F-stable-pair-loc-sys-iso}. With this in mind we will simply write $\chi_A$, (resp.\ $\chi_{K_{\iota}}$), for $\chi_{A,\phi_A}$, (resp.\ $\chi_{K_{\iota},\phi_{\iota}}$).
\end{assumption}

\begin{pa}\label{pa:X-and-Y}
Following \cite[\S24.2]{lusztig:1986:character-sheaves-V} we associate to each $F$-stable pair $\iota \in \mathcal{N}_{\bG}^F$ and each $v \in W_{\bG}(\bL)$ a unipotently supported class function of $G$ by setting
\begin{align*}\label{eq:def-X_i}
X_{\iota}^v(g) &= (-1)^{a_{\iota}}q^{-(\dim\bG + a_{\iota})/2}\chi_{K_{\iota},\phi_{\iota}^v}(g),
\intertext{for all $g \in G$, where $\phi_{\iota}^v = \Theta_{F(v)}\phi_{\iota}$. Furthermore for each $\iota \in \mathcal{N}_{\bG}^F$ we define a second unipotently supported class function of $G$ by setting}
Y_{\iota}(g) &= \begin{cases}
\Tr(\psi_{\iota},(\mathscr{E}_{\iota})_g) &\text{if }g \in \mathcal{O}_{\iota}^F,\\
0 &\text{otherwise},
\end{cases}
\end{align*}
for all $g \in G$. Here $\psi_{\iota} : F^*\mathscr{E}_{\iota} \to \mathscr{E}_{\iota}$ is an isomorphism such that for any split element $u \in \mathcal{O}_{\iota}^F$ we have the map induced by $\psi_{\iota}$ on the stalk $(\mathscr{E}_{\iota})_u$ is the identity. The sets $\mathcal{Y} = \{Y_{\iota} \mid \iota \in \mathcal{N}_{\bG}^F\}$ and $\mathcal{X}^v = \{X_{\iota}^v \mid \iota \in \mathcal{N}_{\bG}^F\}$ are bases for the subspace $\Centu{G}$ of unipotently supported class functions of $G$, (see \cite[24.2.7]{lusztig:1986:character-sheaves-V}). In particular for each $\iota$, $\iota' \in \mathcal{N}_{\bG}^F$ there exists an element $P_{\iota',\iota}^v \in \overline{\mathbb{Q}}_{\ell}$ such that
\begin{equation*}
X_{\iota}^v = \sum_{\iota' \in \mathcal{N}_{\bG}^F} P_{\iota',\iota}^vY_{\iota'}.
\end{equation*}
These coefficients have the following properties:
\begin{equation}\label{eq:P-properties}
P_{\iota'\iota}^v = 
\begin{cases}
0 &\text{if }\mathcal{O}_{\iota'} \not\subset \overline{\mathcal{O}_{\iota}}\\
0 &\text{if }\mathcal{O}_{\iota'} = \mathcal{O}_{\iota}\text{ and }\iota'\neq\iota\\
\gamma_{\bL_{\iota}}^{\bG}(F(v)) &\text{if }\iota'=\iota.
\end{cases}
\end{equation}
which follows from \cref{prop:F-stable-pair-loc-sys-iso,prop:bases-of-A} and \cite[24.1.2]{lusztig:1986:character-sheaves-V}. Note that if $v$ is the identity then we will suppress the superscript writing simply $X_{\iota}$ and $P_{\iota',\iota}$ for $X_{\iota}^v$ and $P_{\iota',\iota}^v$.
\end{pa}
%
\section{Restricting Character Sheaves to the Unipotent Variety}\label{sec:restricting-to-unip-variety}
\begin{pa}
Let us maintain the setup of the previous section, namely that $A \in \widehat{\bG}^F$ is an indecomposable summand of $K_w^{\mathscr{L}}$ which is the image of $(\bL_w,\mathcal{O}_w,\mathscr{E}_w,\mathscr{L}_w)$ under the map in \cref{eq:ind-w-P-construction}. If $\mathscr{L} = \Ql$ then we will denote the complex $K_w^{\mathscr{L}}$ simply by $K_w$. Recall from \cref{pa:X-and-Y} that $\mathcal{X}^w$ is a basis for the subspace $\Cent_U(G)$ of unipotently supported class functions of $G$. In particular we have
\begin{equation*}
\chi_A|_{G_{\uni}} = \sum_{\iota \in \mathcal{N}_{\bG}^F} m(A,\iota,\phi_{\iota}^w)\chi_{K_{\iota},\phi_{\iota}^w}
\end{equation*}
for some coefficients $m(A,\iota,\phi_{\iota}^w) \in \Ql$. It is the purpose of this section to describe the coefficients $m(A,\iota,\phi_{\iota}^w)$. We will do this following the method in \cite{lusztig:1986:on-the-character-values}, in particular we will now recall the sequence of isomorphisms constructed in \cite[\S2.6]{lusztig:1986:on-the-character-values}.
\end{pa}

\begin{pa}\label{pa:d_w-iso}
Assume $(\bL_w,\bQ) \in \mathcal{Z}$ then we will denote by $\ind_{\bL_w \subset \bQ}^{\bG}(A_{\mathscr{L}})$ the image of $(\bL_w,\bQ,\mathcal{O}_w,\mathscr{E}_w,\mathscr{L}_w)$ under the map in \cref{eq:ind-P-construction}. Let $D$ be the similarly named complex defined in \cref{pa:definition-of-induction} then by the discussion in \cref{pa:K-nu-L-ind-A-nu-L} we have $D = \IC(\tilde{X}',\overline{\mathscr{E}_w\boxtimes\mathscr{L}_w})[\dim \tilde{X}']$. If $\mathscr{L} = \Ql$ then we will denote the complex $D$ by $D_0$ and $\ind_{\bL_w \subset \bQ}^{\bG}(A_{\mathscr{L}})$ by $\ind_{\bL_w \subset \bQ}^{\bG}(A_0)$. Let $i : \overline{Y}_{\uni} \hookrightarrow \overline{Y}$ be the natural inclusion of the unipotent elements contained in $\overline{Y}$ then we have a commutative diagram
\begin{center}
\begin{tikzcd}
\overline{\mathcal{O}}_0 \arrow{d}[swap]{i} & \hat{X}_{\uni}' \arrow{l}[swap]{\pi}\arrow{d}[swap]{i}\arrow{r}{\sigma} & \tilde{X}_{\uni}' \arrow{d}[swap]{i}\arrow{r}{\tau} & \overline{Y}_{\uni} \arrow{d}[swap]{i}\\
\overline{\Sigma} & \hat{X}' \arrow{l}[swap]{\pi} \arrow{r}{\sigma} & \tilde{X}' \arrow{r}{\tau} & \overline{Y}
\end{tikzcd}
\end{center}
where $\tilde{X}_{\uni}' = \tau^{-1}(\overline{Y}_{\uni})$ and $\hat{X}_{\uni}' = \sigma^{-1}(\tilde{X}_{\uni}')$. For $\tilde{X}_{\uni}'$ and $\hat{X}_{\uni}'$ we have the action of $i$ is given by the natural action on the first factor. We will denote by $\tilde{Z}_{\uni}' \subset \tilde{X}_{\uni}'$ the subvariety given by $\{(g,h\bQ) \in \bG_{\uni} \times \bG/\bQ \mid h^{-1}gh \in \Sigma\bU_{\bQ}\}$ and by $\tilde{\imath} : \tilde{Z}_{\uni}' \hookrightarrow \tilde{X}_{\uni}'$ the inclusion map. It is clear that we have an isomorphism $\psi : i^*(\mathscr{E}_0\boxtimes\mathscr{L}) \to i^*(\mathscr{E}_0\boxtimes\Ql)$ of local systems and as $\pi$, $\tilde{\imath}$ and $j$, (c.f.\ \cref{pa:K-nu-L-ind-A-nu-L}), commute with $i$ we have an induced isomorphism
\begin{equation*}
\psi' = \tilde{\imath}^*j^*\pi^*(\psi) : i^*\tilde{\imath}^*\overline{\mathscr{E}_0\boxtimes\mathscr{L}} \to i^*\tilde{\imath}^*\overline{\mathscr{E}_0\boxtimes\Ql},
\end{equation*}
which in turn induces an isomorphism
\begin{equation*}
\psi'': i^*\tilde{\imath}^*D = \IC(\tilde{X}_{\uni}',i^*\tilde{\imath}^*\overline{\mathscr{E}_0\boxtimes\mathscr{L}})[\dim\tilde{X}'] \to \IC(\tilde{X}_{\uni}',i^*\tilde{\imath}^*\overline{\mathscr{E}_0\boxtimes\Ql})[\dim\tilde{X}'] = i^*\tilde{\imath}^*D_0.
\end{equation*}
According to \cite[6.6]{lusztig:1984:intersection-cohomology-complexes} we have $(\tau_!D)|_{\overline{Y}_{\uni}} = \tau_!i^*\tilde{\imath}^*D$ and $(\tau_!D_0)|_{\overline{Y}_{\uni}} = \tau_!i^*\tilde{\imath}^*D_0$ hence $\tau_!(\psi'')$ gives an isomorphism $\delta_{\bQ} : \ind_{\bL_w \subset \bQ}^{\bG}(A_{\mathscr{L}})|_{\bG_{\uni}} \to \ind_{\bL_w \subset \bQ}^{\bG}(A_0)|_{\bG_{\uni}}$ in $\mathscr{D}\bG_{\uni}$, (c.f.\ \cite[\S2.6(c)]{lusztig:1986:on-the-character-values}). Note that \cref{pa:direct-image-commutative} does not apply here because $i$ is a closed immersion.
\end{pa}

\begin{pa}
For each pair $(\bL_w,\bQ) \in \mathcal{Z}$ we will denote by
\begin{equation*}
\lambda_{\mathscr{L},\bQ} : \ind_{\bL_w \subset \bQ}^{\bG}(A_{\mathscr{L}})|_{\bG_{\uni}}\to K_w^{\mathscr{L}}|_{\bG_{\uni}} \qquad\text{and}\qquad \lambda_{\bQ} : \ind_{\bL_w \subset \bQ}^{\bG}(A_0)|_{\bG_{\uni}} \to K_w|_{\bG_{\uni}}
\end{equation*}
the restrictions to $\bG_{\uni}$ of the canonical isomorphisms described by \cref{prop:K-nu-L-iso-to-induction}. Putting these isomorphisms together with $\delta_{\bQ}$ we can now define an isomorphism
\begin{equation*}
\epsilon := \lambda_{\bQ}\circ\delta_{\bQ}\circ\lambda_{\mathscr{L},\bQ}^{-1} : K_w^{\mathscr{L}}|_{\bG_{\uni}} \to K_w|_{\bG_{\uni}},
\end{equation*}
(c.f.\ \cite[\S2.6(a)]{lusztig:1986:on-the-character-values}). We claim that $\epsilon$ does not depend upon the choice of parabolic subgroup $\bQ$ used to define it. Firstly, from the construction above, it is clear that $\delta_{\bQ}$ does not depend upon the choice of $\bQ$ hence we need only show that the same is true of $\lambda_{\bQ}$ and $\lambda_{\mathscr{L},\bQ}$. However using the construction in \cref{pa:K-nu-L-ind-A-nu-L} we see that this follows from the commutative diagram
\begin{equation*}
\begin{tikzcd}
\tilde{Y} \arrow{d}[swap]{\ID}\arrow{r}{\kappa_{\bQ}} & \tau_{\bQ}^{-1}(Y) \arrow{d}{\mu}\\
\tilde{Y} \arrow{r}{\kappa_{\bR}} & \tau_{\bR}^{-1}(Y)
\end{tikzcd}
\end{equation*}
where $\mu(g,h\bQ) = (g,h\bR)$ and $(\bL_w,\bR) \in \mathcal{Z}$. Note this statement is implicitly used in \cite[\S2.6]{lusztig:1986:on-the-character-values}.
\end{pa}

\begin{pa}
Let $v \in W_{\bG}(\bL_w,\mathscr{L}_w)$ then we denote by $\dot{v} \in N_{\bG}(\bL_w,\mathscr{L}_w)$ a representative of $v$. Following \cite[\S2.6(d)]{lusztig:1986:on-the-character-values} we define to each $v \in W_{\bG}(\bL_w,\mathscr{L}_w)$ isomorphisms $\tilde{\theta}_{\mathscr{L},v} : \ind_{\bL_w \subset {}^{\dot{v}}\bP_w}^{\bG}(A_{\mathscr{L}}) \to \ind_{\bL_w \subset \bP_w}^{\bG}(A_{\mathscr{L}})$ and $\tilde{\theta}_v : \ind_{\bL_w \subset {}^{\dot{v}}\bP_w}^{\bG}(A_0) \to \ind_{\bL_w \subset \bP_w}^{\bG}(A_0)$ in the following way. Denote by $\varphi_v : \tilde{X}_{\bL_w\subset{}^{\dot{v}}\bP_w}^{\bG} \to \tilde{X}_{\bL_w\subset\bP_w}^{\bG}$ the isomorphism given by $\varphi_v(g,h{}^{\dot{v}}\bP_w) = (g,h\dot{v}\bP_w)$. By \cref{pa:direct-image-commutative} we have $(\tau_{\bL_w\subset{}^{\dot{v}}\bP_w}^{\bG})_!\varphi_v^* = \ID^*(\tau_{\bL_w\subset\bP_w}^{\bG})_! = (\tau_{\bL_w\subset\bP_w}^{\bG})_!$ so we take $\tilde{\theta}_{\mathscr{L},v}$ and $\tilde{\theta}_v$ to be the isomorphisms induced by $\varphi_v$. Using the definition of $\theta_v$ one can check that we have a commutative diagram
\begin{equation*}
\begin{tikzcd}
K_w^{\mathscr{L}}|_{\bG_{\uni}} \arrow{d}{\theta_v}\arrow{r}{\lambda_{\mathscr{L},v}^{-1}} &
\ind_{\bL_w\subset{}^{\dot{v}}\bP_w}^{\bG}(A_{\mathscr{L}})|_{\bG_{\uni}} \arrow{d}{\tilde{\theta}_{\mathscr{L},v}}\arrow{r}{\delta_v} & \ind_{\bL_w\subset{}^{\dot{v}}\bP_w}^{\bG}(A_0)|_{\bG_{\uni}} \arrow{d}{\tilde{\theta}_v}\arrow{r}{\lambda_v} & K_w|_{\bG_{\uni}} \arrow{d}{\theta_v}\\
K_w^{\mathscr{L}}|_{\bG_{\uni}} \arrow{r}{\lambda_{\mathscr{L},1}^{-1}} & \ind_{\bL_w\subset\bP_w}^{\bG}(A_{\mathscr{L}})|_{\bG_{\uni}}\arrow{r}{\delta_1} & \ind_{\bL_w\subset\bP_w}^{\bG}(A_0)|_{\bG_{\uni}}\arrow{r}{\lambda_1} & K_w|_{\bG_{\uni}}
\end{tikzcd}
\end{equation*}
where $\lambda_{\mathscr{L},v} = \lambda_{\mathscr{L},{}^{\dot{v}}\bP_w}$ and $\lambda_v = \lambda_{{}^{\dot{v}}\bP_w}$. In particular this shows that the isomorphism $\epsilon$ is an isomorphism of $W_{\bG}(\bL_w,\mathscr{L}_w)$-modules, (recall that $\epsilon$ does not depend upon the choice of parabolic subgroup used to define it).

We now wish to check that the isomorphism $\epsilon$ respects the action of the Frobenius endomorphism. In other words let $\phi_{\mathscr{L}}^w : F^*K_w^{\mathscr{L}} \to K_w^{\mathscr{L}}$ and $\phi_{\Ql}^w : F^*K_w \to K_w$ be the isomorphisms defined in \cref{pa:iso-K-v} then we wish to show that $\epsilon\circ\phi_{\mathscr{L}}^w = \phi_{\Ql}^w\circ F^*\epsilon$. Recall that in \cref{pa:isomorphism-cuspidal-case} we fixed isomorphisms $F^*A_{\mathscr{L}} \to A_{\mathscr{L}}$ and $F^*A_0 \to A_0$ and that by \cref{lem:F-action-ind} these respectively induce isomorphisms $\psi_{\mathscr{L},\bQ} : F^*\ind_{\bL_w \subset F(\bQ)}^{\bG}(A_{\mathscr{L}}) \to \ind_{\bL_w \subset \bQ}^{\bG}(A_{\mathscr{L}})$ and $\psi_{\bQ} : F^*\ind_{\bL_w \subset F(\bQ)}^{\bG}(A_0) \to \ind_{\bL_w \subset \bQ}^{\bG}(A_0)$. With this we then have a commutative diagram
\begin{equation*}
\begin{tikzcd}[column sep=8ex]
F^*(K_w^{\mathscr{L}}|_{\bG_{\uni}}) \arrow{d}{\phi_{\mathscr{L}}^w}\arrow{r}{F^*\lambda_{\mathscr{L},F(\bQ)}^{-1}} &
F^*(\ind_{\bL_w\subset F(\bQ)}^{\bG}(A_{\mathscr{L}})|_{\bG_{\uni}}) \arrow{d}{\psi_{\mathscr{L},\bQ}}\arrow{r}{F^*\delta_{F(\bQ)}} & F^*(\ind_{\bL_w\subset F(\bQ)}^{\bG}(A_0)|_{\bG_{\uni}}) \arrow{d}{\psi_{\bQ}}\arrow{r}{F^*\lambda_{F(\bQ)}} & F^*(K_w|_{\bG_{\uni}}) \arrow{d}{\phi_{\Ql}^w}\\
K_w^{\mathscr{L}}|_{\bG_{\uni}} \arrow{r}{\lambda_{\mathscr{L},\bQ}^{-1}} & \ind_{\bL_w\subset \bQ}^{\bG}(A_{\mathscr{L}})|_{\bG_{\uni}}\arrow{r}{\delta_{\bQ}} & \ind_{\bL_w\subset \bQ}^{\bG}(A_0)|_{\bG_{\uni}}\arrow{r}{\lambda_{\bQ}} & K_w|_{\bG_{\uni}}
\end{tikzcd}
\end{equation*}
hence $\epsilon$ commutes with the isomorphisms $\phi_{\mathscr{L}}^w$ and $\phi_{\Ql}^w$ as desired, (note that we have used here that $F(\bG_{\uni}) = \bG_{\uni}$).
\end{pa}

\begin{pa}\label{pa:arbitrary-char-sheaf}
We now arrive at our ultimate isomorphism, (c.f.\ \cite[\S2.6(e)]{lusztig:1986:on-the-character-values}). For any simple $W_{\bG}(\bL_w,\mathscr{L}_w)$-module $E$ we can define a $W_{\bG}(\bL_w,\mathscr{L}_w)$-module isomorphism $\mathfrak{X} : K_{w,E}^{\mathscr{L}}|_{\bG_{\uni}} \to K_{w,\hat{E}}|_{\bG_{\uni}}$ by taking the composition
\begin{align*}
K_{w,E}^{\mathscr{L}}|_{\bG_{\uni}} &= \Hom_{W_{\bG}(\bL_w,\mathscr{L}_w)}(E,K_w^{\mathscr{L}}|_{\bG_{\uni}})\\
&\cong \Hom_{W_{\bG}(\bL_w,\mathscr{L}_w)}(E,K_w|_{\bG_{\uni}})\\
&\cong \Hom_{W_{\bG}(\bL_w,\mathscr{L}_w)}(E,\Hom_{W_{\bG}(\bL_w)}(\Ql W_{\bG}(\bL_w),K_w))|_{\bG_{\uni}}\\
&\cong \Hom_{W_{\bG}(\bL_w)}(\Ql W_{\bG}(\bL_w)\otimes E,K_w)|_{\bG_{\uni}}\\
&= K_{w,\hat{E}}|_{\bG_{\uni}}.
\end{align*}
where $\hat{E} = \Ind_{W_{\bG}(\bL_w,\mathscr{L}_w)}^{W_{\bG}(\bL_w)}(E)$ is the induced module. Here we have used $\epsilon$ and the standard isomorphisms given by (2.6) and (2.19) of \cite{curtis-reiner:1981:methods-vol-I}. Chasing through the isomorphisms we can see that
\begin{equation*}
\mathfrak{X}\circ\phi_{K_{w,E}^{\mathscr{L}}} \circ F^*f = \phi_{\Ql}^w\circ F^*\mathfrak{X} \circ F^*f\circ (1\otimes\psi_E).
\end{equation*}
for all $f \in K_{w,E}^{\mathscr{L}}|_{\bG_{\uni}}$.

Using \cref{eq:trace-comparison} and the fact that the modules $\Ql W_{\bG}(\bL_w)\otimes E_{\sigma}$ and $(\Ql W_{\bG}(\bL_w)\otimes E)_{\sigma}$ are isomorphic as $W_{\bG}(\bL_w)$-modules we see that we have an equality
\begin{equation}\label{eq:chi-A-G-uni}
\chi_A|_{\bG_{\uni}} = \sum_{\iota \in \mathscr{I}(\bL,\nu)^F} \langle \tilde{E}_{\iota}, \Ind_{W_{\bG}(\bL_w,\mathscr{L}_w).F}^{W_{\bG}(\bL_w).F}(\tilde{E})\rangle_{W_{\bG}(\bL_w).F}\chi_{K_{\iota},\phi_{\iota}^w}
\end{equation}
Note that in the above we assume that $\widetilde{E}_{\iota}$ and $\widetilde{E}$ are the restrictions of these characters to the appropriate coset. In particular we have
\begin{equation*}
m(A,\iota,\phi_{\iota}^w) = \langle \tilde{E}_{\iota}, \Ind_{W_{\bG}(\bL,\mathscr{L}).Fw}^{W_{\bG}(\bL).F}(\tilde{E})\rangle_{W_{\bG}(\bL).F}
\end{equation*}
for all $\iota \in \mathcal{N}_{\bG}^F$, (c.f.\ \cref{pa:conventions-finite-groups} and \cref{pa:conventions-coset-identification}).
\end{pa}
%
\section{\texorpdfstring{Weyl Groups of Type $\B_n$ and $\D_n$}{Weyl Groups of Type B and D}}\label{sec:weyl-groups}

\begin{pa}
In this section we wish to develop some notation concerning Weyl groups of type $\B_n$. In particular let $(\bW,\mathbb{S})$ be a finite Coxeter system of type $\B_n$ then we will denote by $s_i \in \mathbb{S}$, (with $1 \leqslant i \leqslant n$), the $i$th simple reflection of $\bW$ such that the labelling corresponds to the labelling of the Dynkin diagram given in \cite[Plate II(IV)]{bourbaki:2002:lie-groups-chap-4-6}. We will identify $\bW$ with the group $W_n \subset \GL_n(\mathbb{R})$ consisting of all matrices which have precisely one non-zero entry in each row and column, and where this non-zero entry is either $\pm 1$, (see \cite[1.4.1]{geck-pfeiffer:2000:characters-of-finite-coxeter-groups}). We do this via the isomorphism sending $s_i$, for $1 \leqslant i \leqslant n-1$, to the permutation matrix of the transposition $(i,i+1)$ and sending $s_n$ to the diagonal matrix whose $n$th diagonal entry is $-1$ and whose remaining diagonal entries are $1$. Let $\epsilon : W_n \to \{\pm 1\}$ be the character of $W_n$ defined by $\epsilon(s_i) = 1$ for all $1 \leqslant i \leqslant n-1$ and $\epsilon(s_n) = -1$. We denote by $W_n'$ the kernel of $\epsilon$, which is a Coxeter group of type $\D_n$ with set of Coxeter generators given by $\{s_1,\dots,s_{n-1},u\}$ where $u = s_ns_{n-1}s_n$. Assume $(\bW',\mathbb{S}')$ is a finite Coxeter system of type $\D_n$ then we will denote by $s_i' \in \mathbb{S}'$, (with $1 \leqslant i \leqslant n$), the $i$th simple reflection of $\bW'$ such that the labelling corresponds to the labelling of the Dynkin diagram given in \cite[Plate IV(IV)]{bourbaki:2002:lie-groups-chap-4-6}. We will identify $\bW'$ with $W_n'$ via the isomorphism sending $s_i'$ to $s_i$ for all $1 \leqslant i \leqslant n-1$ and $s_n'$ to $u$.
\end{pa}

\begin{pa}
We now wish to give a precise meaning to the notation we will use to label certain subgroups of $W_n$. Firstly, for each $0 \leqslant i \leqslant n-1$ we will denote by $t_i$ the element $s_{i+1}s_{i+2}\cdots s_{n-1}s_ns_{n-1}\cdots s_{i+2}s_{i+1}$ and by $s_0$ the element $t_0s_1t_0 = t_1s_1t_1$ then we set $\mathbb{S}_0 = \{s_0,s_1,\dots,s_{n-1},u\}$, (note that $t_{n-1} = s_n$ and $t_it_j = t_jt_i$ for all $0\leqslant i,j\leqslant n-1$). The element $s_0$, (resp.\ $t_0$), can be identified as the reflection of the highest root when the underlying root system of $\bW$ is of type $\B_n$, (resp.\ $\C_n$). Assume now that $a,b \in \mathbb{N}_0$ are such that $a+b = n$ then we define a maximal rank subgroup $W_a'\times W_b' \leqslant W_n' \leqslant W_n$ by setting
\begin{equation*}
W_a'\times W_b' = \begin{cases}
\langle\mathbb{S}_0\setminus\{s_0,s_1\}\rangle &\text{if }a = 1\\
\langle\mathbb{S}_0\setminus\{s_{n-1},u\}\rangle &\text{if }a = n-1\\
\langle\mathbb{S}_0\setminus\{u\}\rangle &\text{if }a = n\\
\langle\mathbb{S}_0\setminus\{s_a\}\rangle &\text{otherwise}.
\end{cases}
\end{equation*}
From this we obtain the following additional maximal rank subgroups of $W_n$ by setting
\begin{equation*}
W_a'\times W_b = \langle W_a'\times W_b',s_n \rangle
\qquad
W_a\times W_b' = \langle W_a'\times W_b',t_0 \rangle
\qquad
W_a\times W_b = \langle W_a'\times W_b',t_0,s_n \rangle.
\end{equation*}
\end{pa}

\begin{pa}\label{pa:longest-element-products}
We will see shortly that we will also need to have information concerning the longest elements of $W_n$ and the above subgroups. To obtain a uniform description we will assume that $n,a,b \geqslant 2$, (in the remaining cases we simply obtain the identity element). Let $w_0 \in W_n$, $w_0' \in W_n'$, $w_a' \in W_a'$, $w_b' \in W_b'$, $w_a \in W_a$ and $w_b \in W_b$ be the longest elements of these groups then we have
\begin{align*}
w_0 &= t_{n-1}\cdots t_0 & w_a &= t_0t_1\cdots t_{a-1} & w_b &= t_at_{a+1}\cdots t_{n-1},\\
w_0' &= t_{n-1}^{n-1}t_{n-2}t_{n-3} \cdots t_0  & w_a' &= t_0^{a-1}t_1\cdots t_{a-1} & w_b' &= t_at_{a+1}\cdots t_{n-2}t_{n-1}^{b-1}
\end{align*}
(see for instance \cite[Example 1.5.5]{geck-pfeiffer:2000:characters-of-finite-coxeter-groups}). Considering the products of these elements we see that
\begin{align*}
w_0w_aw_b &= 1 & w_0w_a'w_b &= t_0^a & w_0'w_a'w_b' &= (s_nt_0)^a.
\end{align*}
Assume now that $a,b\geqslant 2$ then with this in mind we will denote by $\gamma_a : W_a' \to W_a'$ and $\gamma_b : W_b' \times W_b'$ the automorphisms defined by $\gamma_a(x) = t_0xt_0$ and $\gamma_b(y) = s_nys_n$ for all $x \in W_a'$ and $y \in W_b'$.

Now let $\gamma : W_n'\to W_n'$ be the automorphism which fixes each $s_i$ with $1 \leqslant i \leqslant n-2$ and exchanges $s_{n-1}$ and $u$. We will identify the semidirect product $\widetilde{W}_n' = W_n' \rtimes \langle \gamma \rangle$ with the group $W_n$ under the map defined by $(v,\gamma) \mapsto vs_n$ for all $v \in W_n'$. Note that $\gamma$ stabilises the subgroup $W_a'\times W_b'$ hence fixes the element $w_0'w_a'w_b'$. Under the identification of $\widetilde{W}_n'$ with $W_n$ we have
\begin{equation*}
(W_a'\times W_b')\langle (\gamma(w_0'w_a'w_b'),\gamma) \rangle \mapsto \begin{cases}
W_a' \times W_b &\text{if }a\equiv 0 \pmod{2}\\
W_a \times W_b' &\text{if }a\equiv 1 \pmod{2}.
\end{cases}
\end{equation*}
\end{pa}

%
\section{Describing the Action of Frobenius}\label{sec:desc-frob-action}

\begin{assumption}
From now until the end of this article we assume that $\bG$ is such that $\bG/Z(\bG)$ is simple of type $\B_n$, $\C_n$ or $\D_n$.
\end{assumption}

\begin{pa}\label{pa:assumptions}
Let us assume that $\bL \in \mathcal{L}_{\std}$ is a standard Levi subgroup such that $\mathcal{N}_{\bL}^0 \neq \emptyset$ and let $\bL^{\star} \in \mathcal{L}_{\std}^{\star}$ be a dual Levi subgroup. We will denote by $s \in \bT_0^{\star}$ a semisimple element such that the series $\widehat{\bL}_s$ contains a unipotently supported cuspidal character sheaf $A_{\mathscr{L}}$ defined as in \cref{eq:cusp-uni-sup} with respect to a local system $\mathscr{L} \in \mathcal{S}(Z^{\circ}(\bL))$. We may replace $s$ by any $N_{\bG^{\star}}(\bT_0^{\star})$-conjugate, in particular we may (and will) assume that we have a set of Coxeter generators $\mathbb{I} \subseteq W_{\bG^{\star}}(s)$ satisfying $\mathbb{I} \subset \mathbb{T}_0$ where $\mathbb{T}_0 = \mathbb{T}\cup\{t_0\}$ and $t_0$ is the reflection of the highest root in $W_{\bG^{\star}}$, (see \cref{pa:duality}).

Let $\pi : \bG^{\star} \to \bG_{\ad}^{\star}$ be an adjoint quotient of $\bG^{\star}$ and let us denote by $F^{\star}$ a Frobenius endomorphism of $\bG_{\ad}^{\star}$ commuting with $\pi$, (note that $\bG_{\ad}^{\star}$ should be read as $(\bG^{\star})_{\ad}$). Let $\bT_{\ad}^{\star} \leqslant \bB_{\ad}^{\star}$ be the respective images of $\bT_0^{\star} \leqslant \bB_0^{\star}$ under $\pi$ then we have $\pi$ induces an isomorphism $W_{\bG^{\star}} \to W_{\bG_{\ad}^{\star}}$ where $W_{\bG_{\ad}^{\star}}$ is the Weyl group of $\bG_{\ad}^{\star}$ defined with respect to $\bT_{\ad}^{\star}$. We will denote by $\bar{s} = \pi(s)$ the image of $s$ under $\pi$ then we have $\pi$ induces an isomorphism $W_{\bG^{\star}}(s) \to W_{\bG_{\ad}^{\star}}(\bar{s})$.

\begin{table}[h!t]
\centering
\begin{tabular}{>{$}c<{$}>{$}c<{$}>{$}c<{$}>{$}c<{$}>{$}c<{$}>{$}c<{$}}
\hline\addlinespace
\text{Type of }\bG & \text{Type of }[\bL,\bL] & m & \text{Type of }W_{\bL^{\star}}(s) & W_{\bG^{\star}}(\bL^{\star})\\\addlinespace
\hline\addlinespace
\B_n & \B_{2t(t+1)} & n-2t(t+1) & \C_{t(t+1)}\times\C_{t(t+1)} & W_m\\\addlinespace
\C_n & \C_{2t(4t\pm 1)} & n-2t(4t\pm1) & \D_{4t^2}\times\B_{4t^2\pm 2t} & W_m\\\addlinespace
\D_n & \D_{8t^2} & n-8t^2 & \D_{4t^2}\times\D_{4t^2} & \begin{cases}W_m &\text{if }t\geqslant 1\\W_m' &\text{if }t=0\end{cases} \\\addlinespace
\hline
\end{tabular}
\caption{Location of Cuspidal Character Sheaves}
\label{tab:levis-1}
\end{table}

The condition that $\bL$ supports a unipotently supported cuspidal character sheaf is quite restrictive. For example the derived subgroup $[\bL,\bL]$ of $\bL$ must be simple of the same type as $\bG$ and the rank must be as in \cref{tab:levis-1} for some $t \geqslant 0$, (see \cite[(3.7)]{geck:1999:character-sheaves-and-GGGRs}). Furthermore we have by \cite[(17.12.4)]{lusztig:1985:character-sheaves} that the series $\widehat{\bL}_s$ containing the cuspidal character sheaf must be such that $s$ is an isolated element of $\bL^{\star}$, i.e.\ $C_{\bL^{\star}}(s)^{\circ}$ is not contained in any proper Levi subgroup of $\bL^{\star}$. In fact, as is described in \cite[Remark 5.4]{geck:1999:character-sheaves-and-GGGRs}, more can be said about $s$ in that the Weyl group of its centraliser must have the structure described in \cref{tab:levis-1}.

\begin{assumption}
From now until the end of this article we will maintain the assumptions of \cref{pa:assumptions}. Furthermore we will assume that $s$ is an isolated semisimple element of $\bG^{\star}$, (i.e.\ $C_{\bG^{\star}}(s)^{\circ}$ is not contained in any proper Levi subgroup of $\bG^{\star}$)
\end{assumption}

\noindent Let us recall from Proposition 2.3(b) and Table 2 of \cite{bonnafe:2005:quasi-isolated} that as $s$ is an isolated semisimple element of $\bG^{\star}$ we must have $\bar{s}^2 = 1$. Under the identification of $W_{\bG^{\star}}$ with $W_n$, (resp.\ $W_n'$), if $\bG$ is of type $\B_n$ or $\C_n$, (resp.\ $\D_n$), we will identify the subgroup $W_{\bG^{\star}}(s)$ with the reflection subgroup
\begin{equation*}
\bV_a \times \bW_b = \begin{cases}
W_a \times W_b &\text{if }\bG\text{ is of type }\B_n\\
W_a' \times W_b &\text{if }\bG\text{ is of type }\C_n\\
W_a' \times W_b' &\text{if }\bG\text{ is of type }\D_n
\end{cases}
\end{equation*}
where $a,b \in \mathbb{N}$ are such that $n = a+b$, (if $\bG$ is of type $\C_n$ or $\D_n$ then we also have $a\neq1$). The integers $a$ and $b$ are determined in the following way. First recall from \cref{pa:assumptions} that $W_{\bG^{\star}}(s)$ is naturally isomorphic under $\pi$ to $W_{\bG_{\ad}^{\star}}(\bar{s})$. Let us fix an isomorphism $\sigma : \bT_{\ad}^{\star} \to (\mathbb{K}^{\times})^n$ then as $\bar{s}$ satisfies $\bar{s}^2=1$ we have $\sigma(\bar{s}) = (t_1,\dots,t_n)$ with each $t_i \in \{\pm1\}$. We can then identify $a$ as the number of $t_i = -1$ and $b$ as the number of $t_i = 1$, (note that these values do not depend upon the choice of isomorphism $\sigma$).
\end{pa}

\begin{pa}\label{pa:frob-action-s}
We now wish to consider the action of $F^{\star}$ on the semisimple element $s$. Let us assume that $\widehat{\bG}_s$ contains an $F$-stable character sheaf then this implies that the $N_{\bG^{\star}}(\bT_0^{\star})$-orbit of $s$ is $F^{\star}$-stable. From the proof of \cite[Proposition 6.14]{taylor:2012:finding-characters-satisfying} and the choice of $s$ in \cref{pa:assumptions} we see that $\bar{s} \in \bG_{\ad}^{\star}$ is fixed by $F^{\star}$. Combining this with the fact that the $N_{\bG^{\star}}(\bT_0^{\star})$-orbit of $s$ is $F^{\star}$-stable we obtain that $F^{\star}(s) = sz_1$ for some $z_1 \in Z(\bG_{\der}^{\star})$ where $\bG_{\der}^{\star} \leqslant \bG^{\star}$ is the derived subgroup. Applying \cite[Corollary 2.8]{bonnafe:2005:quasi-isolated} we see that $|A_{\bG_{\ad}^{\star}}(\bar{s})| = |\{y \in Z(\bG_{\der}^{\star}) \mid sy$ and $s$ are conjugate in $\bG^{\star}\}|$, hence from the information in \cite[Table 2]{bonnafe:2005:quasi-isolated} we see that there are at most two possible choices for $z_1$ except when $\bG^{\star}$ is of type $\D_n$ and $a=b$, (c.f.\ \cref{pa:assumptions}), then there are at most four possible choices for $z_1$. Assume we are not in the latter case then the argument used in \cite[Proposition 6.14]{taylor:2012:finding-characters-satisfying} shows that either
\begin{equation}\label{eq:property-frob-s}
F^{\star}(s) = s\qquad\text{or}\qquad F^{\star}(s) = s^{n_0}
\end{equation}
where $n_0 \in N_{\bG^{\star}}(\bT_0^{\star})$ is a representative of the longest element. In the remaining case the argument used in \cite[Proposition 6.14]{taylor:2012:finding-characters-satisfying} shows that at least one element from $\{sy \mid y \in Z(\bG_{\der}^{\star})\}$ satisfies the condition in \cref{eq:property-frob-s}. As all such elements are conjugate, (in fact conjugate by elements centralising $\bar{s}$), we may replace $s$ by an element satisfying the condition in \cref{eq:property-frob-s}, hence we may assume that this condition holds.
\end{pa}

\begin{rem}\label{rem:F-action-on-s}
Let us note that we can also see from \cite[Proposition 6.14]{taylor:2012:finding-characters-satisfying} that if $q \equiv 1 \pmod{4}$ then we will always have $F^{\star}(s) = s$.
\end{rem}

\begin{pa}
As in \cref{pa:conj-levi-etc} we may now consider the subset
\begin{equation*}
Z_s = \{n \in N_{\bG^{\star}}(\bT_0^{\star}) \mid F^{\star}(s) = s^n\}/\bT_0^{\star} \subseteq W_{\bG^{\star}}
\end{equation*}
which is a right coset of $W_{\bG^{\star}}(s)$ in $W_{\bG^{\star}}$. Let us denote by $\Phi$ the roots of $\bG^{\star}$ with respect to $\bT_0^{\star}$ and $\Phi^+ \subset \Phi$ the set of positive roots determined by $\bB_0^{\star}$. Similarly we denote by $\Phi(s) \subset \Phi$ the root system of $C_{\bG^{\star}}(s)$ with respect to $\bT_0^{\star}$ and $\Phi^+(s) = \Phi(s) \cap \Phi^+$ the positive roots determining the Coxeter generators $\mathbb{I}$. Again by \cite[Lemma 1.9]{lusztig:1984:characters-of-reductive-groups} there exists a unique element $w_s \in Z_s$ of minimal length with respect to the length function of $(W_{\bG^{\star}},\mathbb{T})$ which satisfies the condition $w_s(\Phi^+(s)) \subset \Phi^+$. Let $\gamma : W_{\bG^{\star}}(s) \to W_{\bG^{\star}}(s)$ be the automorphism given by $\gamma(x) = {}^{w_s}F^{\star}(x)$ then to describe $\gamma$ we need to understand the element $w_s$. By \cref{eq:property-frob-s} we have either the identity or the longest element is contained in $Z_s$. If $Z_s$ contains the identity then clearly $w_s$ is the identity and $\gamma = F^{\star}$. Let us now assume that $Z_s$ contains the longest element $w_0 \in W_{\bG^{\star}}$ then we may write $w_s$ as a product $w_0x$ for some unique $x \in W_{\bG^{\star}}(s)$. As $w_s(\Phi^+(s)) \subset \Phi^+$ this implies that $x(\Phi^+(s)) \subset w_0(\Phi^+) = -\Phi^+$ but as $x$ stabilises $\Phi(s)$ this implies $x(\Phi^+(s)) = \Phi(s) \cap (-\Phi^+) = -\Phi^+(s)$ hence $x$ is the longest element of $W_{\bG^{\star}}(s)$.

Using the above and the discussion in \cref{sec:weyl-groups} we see that we have the following possibilities for the automorphism $\gamma$.
\begin{itemize}
	\item If $\bG$ is of type $\B_n$ then the automorphism $\gamma$ is the identity.
	\item If $\bG$ is of type $\C_n$ then $\gamma$ is an element of the set $\{\ID, \gamma_a\times \ID\}$.
	\item If $\bG$ is of type $\D_n$ and $F$ is the identity on $(W_{\bG},\mathbb{S})$ then $\gamma$ is an element of the set $\{\ID, \gamma_a\times\gamma_b\}$.
	\item If $\bG$ is of type $\D_n$ and $F$ is of order 2 on $(W_{\bG},\mathbb{S})$ then $\gamma$ is an element of the set $\{\gamma_a\times \ID, \ID\times\gamma_b\}$.
\end{itemize}
In the above $\ID$ denotes the appropriate identity automorphism.
\end{pa}

\begin{pa}
We will denote by $\mathscr{E} \in \mathcal{S}(\bT_0)$ the unique local system satisfying $\lambda_{\bT_0}(\mathscr{E}) = s$ where $s$ is as above, i.e.\ $s \in \bT_0^{\star}$ is such that $A_{\mathscr{L}} \in \widehat{\bG}_s$. We now wish to prove the following lemma which gives the first step in describing the relationship between the two labellings of character sheaves.
\end{pa}

\begin{lem}
We have $\mathscr{E}|_{Z^{\circ}(\bL)} = \mathscr{L}$.
\end{lem}

\begin{proof}
Let us denote by $\bL'$ the direct product $\bL_{\der} \times Z^{\circ}(\bL)$ where $\bL_{\der}$ is the derived subgroup of $\bL$. We have a natural surjective morphism $\tau : \bL' \to \bL$ given by $\tau(l,y) = ly$ whose kernel is $\{(l,l^{-1}) \mid l \in \bL_{\der} \cap Z^{\circ}(\bL)\}$. Let $\bT_0'$ be the maximal torus $\tau^{-1}(\bT_0)$ then there exists a maximal torus $\bS_{\der} \leqslant \bL_{\der}$ such that $\bT_0' = \bS_{\der} \times Z^{\circ}(\bL)$, (note that we will also consider $\tau$ to be a morphism $\tau : \bT_0' \to \bT_0$ without specific mention). Furthermore let us denote by $\mathscr{E}' \in \mathcal{S}(\bT_0')$ the inverse image $\tau^*\mathscr{E}$. Clearly the local system $\mathscr{E}'$ is a direct product $\mathscr{E}_1\boxtimes\mathscr{E}_2$ where $\mathscr{E}_1 = \mathscr{E}'|_{\bS_{\der}} \in \mathcal{S}(\bS_{\der})$ and $\mathscr{E}_2 = \mathscr{E}'|_{Z^{\circ}(\bL)} \in \mathcal{S}(Z^{\circ}(\bL))$. From the definition of character sheaves, (see \cite[Definition 2.10]{lusztig:1985:character-sheaves}), we see that
\begin{equation*}
\widehat{\bL}'_{\mathscr{E}'} = (\widehat{\bL}_{\der})_{\mathscr{E}_1} \boxtimes (\widehat{Z^{\circ}(\bL)})_{\mathscr{E}_2},
\end{equation*}
(see also \cite[(17.11)]{lusztig:1985:character-sheaves}). Furthermore the series $(\widehat{Z^{\circ}(\bL)})_{\mathscr{E}_2}$ contains only the perverse sheaf $\mathscr{E}_2[\dim Z^{\circ}(\bL)]$. Assume now that $A \in \widehat{\bL}_{\mathscr{E}}$ then by \cite[(17.16.1)]{lusztig:1985:character-sheaves} we have the summands of $\tau^*A$ are contained in $\widehat{\bL}'_{\mathscr{E}'}$. Let us consider the cuspidal character sheaf $A_{\mathscr{L}}$ then the inverse image is given by
\begin{equation*}
\tau^*A_{\mathscr{L}} = \bigoplus_i\IC(\overline{\mathcal{O}_0 \times Z^{\circ}(\bL)},\mathscr{E}_{0,i} \boxtimes \mathscr{L})[\dim \mathcal{O}_0\times Z^{\circ}(\bL)]
\end{equation*}
where $\mathscr{E}_{0,i}$ are cuspidal local systems on $\mathcal{O}_0 \subset \bL_{\der}$. Let $A_i$ be the summand of $\tau^*A_{\mathscr{L}}$ determined by $\mathscr{E}_{0,i}$ then by the above remarks we may write this as
\begin{equation*}
A_i = \overline{A}_i \boxtimes \mathscr{E}_2[\dim Z^{\circ}(\bL)]
\end{equation*}
for some unique character sheaf $\overline{A}_i \in (\widehat{\bL}_{\der})_{\mathscr{E}_1}$. As $A_i|_{\mathcal{O}_0 \times Z^{\circ}(\bL)} \cong (\mathscr{E}_{0,i} \boxtimes \mathscr{L})[\dim \mathcal{O}_0 \times Z^{\circ}(\bL)]$ we deduce that $\mathscr{E}_2 = \mathscr{L}$. Let $i : Z^{\circ}(\bL) \hookrightarrow \bT_0'$ and $j : Z^{\circ}(\bL) \hookrightarrow \bT_0$ be the natural inclusion maps then it is clear that we have $\tau \circ i = j$. Rephrasing the above we have $\mathscr{E}|_{Z^{\circ}(\bL)} = j^*\mathscr{E} = i^*\tau^*\mathscr{E} = \mathscr{L}$ as required.
\end{proof}

\begin{pa}\label{pa:a'-b'}
We now relate this property to the dual group. We will denote by $\bS^{\star} \leqslant \bT_0^{\star}$ a subtorus such that we have a direct product $\bT_0^{\star} = \bS^{\star} \times Z^{\circ}(\bL^{\star})$. Note that $Z^{\circ}(\bL^{\star}) \leqslant \bT_0^{\star}$ is dual to $Z^{\circ}(\bL) \leqslant \bT_0$ hence the local system $\mathscr{L} \in \mathcal{S}(Z^{\circ}(\bL))$ determines a unique element $z \in Z^{\circ}(\bL^{\star})$ under the map $\lambda_{Z^{\circ}(\bL)}$, (c.f. \cref{pa:duality}). The above result then says that there exists a unique element $\hat{s} \in \bS^{\star}$ such that $s = \hat{s}z$. Let $\bar{\bL}^{\star}$ be the Levi subgroup $\pi(\bL^{\star}) \leqslant \bG_{\ad}^{\star}$ and $\bar{z} = \pi(z)$ the image of $z$, (c.f.\ \cref{pa:assumptions}). Above we deduced that $\bar{s}^2=1$ hence this clearly implies that $\bar{z}^2 = 1$. Applying the anti-isomorphism $W_{\bG} \to W_{\bG^{\star}}$ we see that the subgroup $W_{\bG}(\bL,\mathscr{L})$ is mapped to $W_{\bG^{\star}}(\bL^{\star},z) = N_{\bG^{\star}}(\bL^{\star},z)/\bL^{\star}$, (note that $\bL^{\star} = C_{\bL^{\star}}(z)$), which is in turn mapped to $W_{\bG_{\ad}^{\star}}(\bar{\bL}^{\star},\bar{z}) = N_{\bG_{\ad}^{\star}}(\bar{\bL}^{\star},\bar{z})/\bar{\bL}^{\star}$ under the isomorphism $W_{\bG^{\star}} \to W_{\bG_{\ad}^{\star}}$.

We now consider how to describe the group $W_{\bG_{\ad}^{\star}}(\bar{\bL}^{\star},\bar{z})$, which we will do in a similar way as in \cref{pa:assumptions}. Under the identification of $W_{\bG^{\star}}(\bL^{\star})$ with $W_m$, we will identify $W_{\bG^{\star}}(\bL^{\star},z)$ with the subgroup
\begin{equation*}
\bV_{a'} \times \bW_{b'} = 
\begin{cases}
W_{a'}'\times W_{b'} &\text{if }\bG\text{ is of type }\C_n\text{ and }t=0,\\
W_{a'}'\times W_{b'}' &\text{if }\bG\text{ is of type }\D_n\text{ and }t=0,\\
W_{a'}\times W_{b'} &\text{otherwise}.
\end{cases}
\end{equation*}
where $a',b' \in \mathbb{N}_0$ are such that $a'+b' = m$ and $t$ is as in \cref{tab:levis-1}. The values $a'$ and $b'$ are determined in the following way. The isomorphism $\sigma : \bT_{\ad}^{\star} \to (\mathbb{K}^{\times})^n$ restricts to an isomorphism $Z^{\circ}(\bar{\bL}^{\star}) \to (\mathbb{K}^{\times})^m$, which we again denote by $\sigma$. As $\bar{z}^2 = 1$ we have $\sigma(\bar{z}) = (t_1,\dots,t_m)$ with each $t_i \in \{\pm1\}$, and we set $a'$ to be the number of entries $t_i=-1$ and $b'$ to be the number of entries $t_i=1$. Note that we necessarily have $a'\leqslant a$ and $b' \leqslant b$.
\end{pa}

\begin{pa}
We now end our discussion by considering the coset $Z_{\mathscr{L}}$, (see \cref{pa:conj-levi-etc}). Under the isomorphism $W_{\bG} \to W_{\bG^{\star}}$ this coset is taken to the coset
\begin{equation*}
Z_z = \{n \in N_{\bG^{\star}}(\bL^{\star}) \mid F^{\star}(z) = z^n\}/\bL^{\star} \subseteq W_{\bG^{\star}}(\bL^{\star}).
\end{equation*}
Above we showed that either $F^{\star}(s) = s$ or $F^{\star}(s) = s^{n_0}$ where $n_0 \in N_{\bG^{\star}}(\bT_0^{\star})$ is a representative of the longest element in $W_{\bG^{\star}}$. The decomposition $s = \hat{s}z$ is unique hence we must have either $F^{\star}(\hat{s}) = \hat{s}$ and $F^{\star}(z) = z$ or $F^{\star}(\hat{s}) = \hat{s}^{n_0}$ and $F^{\star}(z) = z^{n_0}$. It is easy to see that $n_0 \in N_{\bG^{\star}}(\bL^{\star})$ and that $n_0$ is also a representative for the longest element of $W_{\bG^{\star}}(\bL^{\star})$, hence we either have $Z_z$ contains the identity of $W_{\bG^{\star}}(\bL^{\star})$ or the longest element. As the longest elements are in bijective correspondence under the anti-isomorphism $W_{\bG}(\bL) \to W_{\bG^{\star}}(\bL^{\star})$ it is clear that either the identity or the longest element of $W_{\bG}(\bL)$ is contained in $Z_{\mathscr{L}}$.

Let $w^{-1} \in Z_{\mathscr{L}}$ be the unique element of minimal length then the argument used in \cref{pa:frob-action-s} shows that the automorphism $F_w$ of the Coxeter system $(W_{\bG}(\bL,\mathscr{L}),\mathbb{J})$, (see \cref{pa:generators-rel-Weyl}), can be computed from the action of the longest element in $W_{\bG}(\bL)$ and the longest element in $W_{\bG}(\bL,\mathscr{L})$. In particular when $\bL \neq \bT_0$ we see that $w=1$ and the automorphism induced by $F_w$ is simply the identity. If $\bL = \bT_0$ then $z = s$ and the automorphism $F_w$ is as in \cref{pa:generators-rel-Weyl}.
\end{pa}
%
\section{Parameterisation of Character Sheaves in Isolated Series}\label{sec:parameterisation-of-char-sheaves}
\subsection{Combinatorics}
\begin{pa}\label{pa:symbols-unipotent}
We say $B$ is a $\beta$-set if it is a finite (possibly empty) subset of $\mathbb{N}_0$. If $X = \{x_1,\dots,x_s\}$ is any finite subset of $\mathbb{N}_0$ then we will assume that the elements are labelled such that $x_1 < x_2 < \cdots < x_s$. We define the rank of a $\beta$-set $B$ to be $\rk(B) = \sum_{b \in B} b - \binom{|B|}{2}$, (note that this is 0 if $B = \emptyset$). We will denote by $\widetilde{\mathbb{W}}_N$ the set of all $\beta$-sets of rank $N \geqslant 0$. For each $k \in \mathbb{N}_0$ we define a shift operation ${}^{+k} : \mathbb{W}_N \to \mathbb{W}_N$ given by $B^{+k} = \{0,1,\dots,k-1\}\cup\{b+k\mid b \in B\}$; note that this is just $\{0,1,\dots,k-1\}$ if $B$ is empty and simply $B$ if $k=0$. We can then define an equivalence relation $\sim$ on $\widetilde{\mathbb{W}}_N$ by setting $A \sim B$ if and only if there exists $k \in \mathbb{N}$ such that $A^k = B$ or $A = B^k$. We denote by $\mathbb{W}_N$ the resulting set of equivalence classes.
\end{pa}

\begin{pa}
We say $\Lambda = \sbpair{A}{B}$ is a $\beta$-pair if it is an ordered pair of $\beta$-sets. We define the defect of $\Lambda$ to be $d(\Lambda) = |A| - |B|$ and the rank of $\Lambda$ to be
\begin{equation*}
\rk(\Lambda) = \rk(A) + \rk(B) + \left\lfloor \left( \frac{d(\Lambda)}{2} \right)^2 \right\rfloor,
\end{equation*}
where for any $x \in \mathbb{R}$ we have $\lfloor x \rfloor = \sup\{k \in \mathbb{N}_0 \mid k\leqslant x\}$. For each $N \in \mathbb{N}_0$ and $d \in \mathbb{Z}$ we denote by $\widetilde{\mathbb{V}}_N$ the set of all $\beta$-pairs of rank $N$ and by $\widetilde{\mathbb{V}}_N^d \subseteq \widetilde{\mathbb{V}}_N$ the $\beta$-pairs of rank $N$ and defect $d$. As for $\beta$-sets we define for each $k \in \mathbb{N}_0$ a shift operation ${}^{+k} : \widetilde{\mathbb{V}}_N^d \to \widetilde{\mathbb{V}}_N^d$ defined by $\sbpair{A}{B}^{+k} = \sbpair{A^{+k}}{B^{+k}}$. With this we can define an equivalence relation $\sim$ on $\widetilde{\mathbb{V}}_N^d$ given by $\Lambda\sim\Xi$ if and only if there exists $k \in \mathbb{N}_0$ such that $\Lambda^k = \Xi$ or $\Lambda = \Xi^k$. We denote by $\mathbb{V}_N^d$ the resulting equivalence classes and by $\mathbb{V}_N$ the union $\cup_{d \in \mathbb{Z}} \mathbb{V}_N^d$, (we call these equivalence classes symbols of rank $N$ and defect $d$).

Assume now that $\ssymb{A}{B} \in \mathbb{V}_N^0$ and $\ssymb{C}{D} \in \mathbb{V}_N^1$ then for each $d \in \mathbb{N}_0$ the maps given by
\begin{equation*}\label{eq:bij-shift-def}
\symb{A}{B} \mapsto \symb{A^{+d}}{B} \qquad\qquad \symb{C}{D} \mapsto \symb{C^{+(d-1)}}{D}
\end{equation*}
define natural bijections $\mathbb{V}_N^0 \to \mathbb{V}_{N'}^d$ and $\mathbb{V}_N^1 \to \mathbb{V}_{N'}^d$ where $N' = N + \lfloor (d/2)^2 \rfloor$, (see \cite[Proposition 3.2]{lusztig:1977:irreducible-representations-of-finite-classical-groups}). It will also be convenient for us to define the notion of a cuspidal symbol, (see \cite[\S8.1]{lusztig:1984:characters-of-reductive-groups}). Specifically we say $\ssymb{A}{\emptyset}\in\mathbb{V}_N$ is a cuspidal symbol if
\begin{equation*}
A = \begin{cases}
\{0,1,2,\dots,2d\} &N = d^2+d\text{ for some }d\geqslant 1,\\
\{0,1,2,\dots,2d-1\} &N = 2d^2\text{ for some }d\geqslant 2.
\end{cases}
\end{equation*}
\end{pa}

\begin{pa}\label{pa:def-labeling-sets}
For each $d \in \mathbb{Z}\setminus\{0\}$ we will denote by $\overline{\mathbb{V}}_N^d$ the union $\mathbb{V}_N^d \sqcup \mathbb{V}_N^{-d}$ where we identify $\mathbb{V}_N^d$ with $\mathbb{V}_N^{-d}$ via the map $\binom{A}{B} \mapsto \binom{B}{A}$, in particular we consider $\binom{A}{B}$ and $\binom{B}{A}$ to be the same element of $\overline{\mathbb{V}}_N^d$. Assume $\Lambda = \binom{A}{B} \in \mathbb{V}_N^0$ is a symbol of defect 0 then we say $\Lambda$ is \emph{degenerate} if $A = B$. We denote by $\widehat{\mathbb{V}}_N^0 \subseteq \mathbb{V}_N^0$ the set of all non-degenerate symbols. Furthermore we denote by $\overline{\mathbb{V}}_N^0$ the set of all non-degenerate symbols, (this time with the ordering forgotten), together with two symbols $\Lambda_+$ and $\Lambda_-$ for each degenerate symbol $\Lambda \in \mathbb{V}_N^0$. For each $N \in \mathbb{N}_0$ we now define the following four sets of symbols
\begin{align*}
\Phi_N^+ &= \bigsqcup_{\substack{d\in\mathbb{N}\\d\text{ odd}}}\mathbb{V}_N^d & \Phi_N^- &= \widehat{\mathbb{V}}_N^0\bigsqcup_{\substack{d\in\mathbb{Z}\setminus\{0\}\\d \equiv 0 \pmod{4}}}\mathbb{V}_N^d\\
\Omega_N^+ &= \bigsqcup_{\substack{d\in\mathbb{N}\\d \equiv 0 \pmod{4}}}\overline{\mathbb{V}}_N^d & \Omega_N^- &= \bigsqcup_{\substack{d\in\mathbb{N}\\d \equiv 2 \pmod{4}}}\overline{\mathbb{V}}_N^d
\end{align*}
which will form our indexing sets.

We now wish to define an equivalence relation $\equiv$ on $\Phi_N^{\pm}$ and $\Omega_N^{\pm}$ known as \emph{similarity}. Assume $\Lambda = \ssymb{A_1}{B_1}$ and $\Xi = \ssymb{A_2}{B_2} \in \Phi_N^{\pm}$ then we write $\Lambda \equiv \Xi$ if they can be represented so that $A_1 \cup B_1 = A_2 \cup B_2$ and $A_1 \cap B_2 = A_2 \cap B_2$, (c.f.\ \cite[4.14]{lusztig:1986:on-the-character-values}). If $\Lambda$, $\Xi \in \Omega_N^{\pm}$ are non-degenerate then we define the equivalence relation $\equiv$ in the same way. If $\Lambda_{\pm} \in \Omega_N^+$ is degenerate then we additionally set $\Lambda_{\pm} \equiv \Xi$ if and only if $\Xi = \Lambda_{\pm}$, (i.e.\ $\{\Lambda_+\}$ and $\{\Lambda_-\}$ form equivalence classes of cardinality 1). Note that we will call the equivalence class of $\Lambda$ under $\equiv$ the similarity class of $\Lambda$.
\end{pa}

\begin{pa}\label{pa:weyl-group-chars-symbols}
Let $\bW$ be a finite irreducible Coxeter group of classical type then we wish to give a labelling to the characters of $\bW$. We will do this using the following bijections which are described in \S4.4 - \S4.6 and \S4.18 of \cite{lusztig:1984:characters-of-reductive-groups}.
\begin{equation*}
\Irr(\bW) \longleftrightarrow \begin{cases}
\mathbb{W}_{n+1} &\text{if }\bW\text{ is of type }\A_n\\
\mathbb{V}_n^{\delta} &\text{if }\bW\text{ is of type }\B_n/\C_n\\
\overline{\mathbb{V}}_n^0 &\text{if }\bW\text{ is of type }\D_n
\end{cases}
\end{equation*}
where $\delta \in \{0,1\}$. Note that the two parameterisations in the case of type $\B_n/\C_n$ are related to each other via the natural bijection $\mathbb{V}_n^0 \to \mathbb{V}_n^1$ given by \cref{eq:bij-shift-def}. In what follows we will often identify, using these bijections, the set of irreducible characters $\Irr(\bW)$ with the corresponding labelling set without explicit mention.
\end{pa}

\begin{pa}\label{pa:a-value-symbol-weyl}
We introduce here two functions $\sigma : \mathbb{V}_N \to \mathbb{N}$ and $a : \mathbb{V}_N \to \mathbb{N}$ which will play a role later on. Given a symbol $\Lambda \in \mathbb{V}_N$ we write the entries of $\Lambda$ in a single row in increasing order $y_{\epsilon} \leqslant y_{\epsilon+1} \leqslant \cdots \leqslant y_{2k}$, where $\epsilon \in \{0,1\}$ is such that $d(\Lambda)+1 \equiv \epsilon \pmod{2}$. Then we define
\begin{equation*}
\sigma(\Lambda) = \sum_{\epsilon\leqslant i < j \leqslant 2k} y_i = \sum_{i=\epsilon}^{2k} (2k-i)y_i
\end{equation*}
Let $y_{\epsilon}^0 \leqslant y_{\epsilon+1}^0 \leqslant \cdots \leqslant y_{2k}^0$ denote the sequence
\begin{equation*}
\begin{aligned}
&0 \leqslant 0 \leqslant 1 \leqslant 1 \leqslant \cdots \leqslant k-1 \leqslant k-1 &&\text{if }\epsilon = 1\\
&0 \leqslant 0 \leqslant 1 \leqslant 1 \leqslant \cdots \leqslant k-1 \leqslant k-1 \leqslant k &&\text{if }\epsilon = 0
\end{aligned}
\end{equation*}
then we define
\begin{equation*}
a(\Lambda) = \sum_{\epsilon\leqslant i < j \leqslant 2k} (y_i - y_i^0) = \sum_{i=\epsilon}^{2k} (2k-i)(y_i - y_i^0).
\end{equation*}
We call $a(\Lambda)$ the $a$-value of the symbol. One readily checks that this definition agrees with the that given in \cite[Definition 6.4.3]{geck-pfeiffer:2000:characters-of-finite-coxeter-groups}, (see also \cite[4.11]{lusztig:1986:on-the-character-values}).
\end{pa}

\subsection{Two Classifications of Character Sheaves}
\begin{pa}
We denote by
\begin{equation}
\begin{aligned}
\overline{\bOmega}_a \times \overline{\bPhi}_b =  \begin{cases}
\Phi_a^+ \times \Phi_b^+ &\text{if }\bG\text{ is of type }\B\\
\Omega_a^+ \times \Phi_b^+ &\text{if }\bG\text{ is of type }\C\\
\Omega_a^+ \times \Omega_b^+ &\text{if }\bG\text{ is of type }\D
\end{cases}
\qquad
\overline{\bOmega}_a^1 \times \overline{\bPhi}_b^1 =  \begin{cases}
\mathbb{V}_a^1 \times \mathbb{V}_b^1 &\text{if }\bG\text{ is of type }\B\\
\overline{\mathbb{V}}_a^0 \times \mathbb{V}_b^1 &\text{if }\bG\text{ is of type }\C\\
\overline{\mathbb{V}}_a^0 \times \overline{\mathbb{V}}_b^0 &\text{if }\bG\text{ is of type }\D
\end{cases}
\end{aligned}
\end{equation}
a set of symbols corresponding to $W_{\bG^{\star}}(s)$, (note that we have $\overline{\bOmega}_a^1 \times \overline{\bPhi}_b^1$ is naturally a subset of $\overline{\bOmega}_a \times \overline{\bPhi}_b$). Now let $(\Lambda_1,\Xi_1) \in \overline{\bOmega}_a^1 \times \overline{\bPhi}_b^1$ then by the bijection described in \cref{pa:weyl-group-chars-symbols} these symbols determine an irreducible character $E \in \Irr(\bV_a\times\bW_b)$. Assume $\mathscr{E} \in \mathcal{S}(\bT_0)$ is the local system such that $\lambda_{\bT_0}(\mathscr{E}) = s$ then we will denote by $\mathscr{R}_s(\Lambda_1,\Xi_1)$ the object of $\mathscr{K}_0(\bG)\otimes\Ql$, (c.f.\ \cref{pa:grothendieck-group}), defined in \cite[(14.10.3)]{lusztig:1985:character-sheaves} and there denoted $R_E^{\mathscr{L}}$. With this we may now formulate the classification result of \cite{lusztig:1985:character-sheaves} in the following way.
\end{pa}

\begin{thm}[{}{Lusztig, \cite[Theorem 23.1]{lusztig:1986:character-sheaves-V}}]\label{thm:char-sheaves-main-theorem}
There exists a bijection $\overline{\bOmega}_a \times \overline{\bPhi}_b \to \widehat{\bG}_s$ denoted by $(\Lambda,\Xi) \mapsto A(\Lambda,\Xi)$ which satisfies the following condition. For any $(\Lambda,\Xi) \in \overline{\bOmega}_a \times \overline{\bPhi}_b$ and any $(\Lambda_1,\Xi_1) \in \overline{\bOmega}_a^1 \times \overline{\bPhi}_b^1$ we have
\begin{equation*}
(A(\Lambda,\Xi):\mathscr{R}_s(\Lambda_1,\Xi_1)) = \begin{cases}
(-1)^{\langle (\Lambda,\Xi), (\Lambda_1,\Xi_1) \rangle}|V|^{-1/2}|W|^{-1/2} &\text{if $\Lambda \equiv \Lambda_1$ and $\Xi \equiv \Xi_1$}\\
0 &\text{otherwise}.
\end{cases}
\end{equation*}
Here $V$, (resp.\ $W$), denotes the similarity class of $\Lambda$, (resp.\ $\Xi$), and $\langle-,-\rangle$ denotes the symplectic form on $V\times W$ defined by Lusztig in \cite[\S4.5-4.6]{lusztig:1984:characters-of-reductive-groups}. Note that this is the sum of the symplectic forms defined on $V$ and $W$.
\end{thm}

\begin{rem}
The above can be extracted from \S4.5 and \S4.6 of \cite{lusztig:1984:characters-of-reductive-groups}. One only needs to note that the set denoted $\mathcal{M}_{\mathcal{F}}$ in \cite{lusztig:1984:characters-of-reductive-groups} is the similarity class of the symbol. Note that by Proposition 18.5, Lemmas 23.3 and 23.5 of \cite{lusztig:1985:character-sheaves} we have the value denoted $\hat{\epsilon}_A = \epsilon_A$ in \cite{lusztig:1985:character-sheaves} is always 1.
\end{rem}

\begin{pa}\label{pa:label-embed}
The statement of \cref{thm:char-sheaves-main-theorem} gives a classification for the whole series $\widehat{\bG}_s$. We would now like to consider a labelling of the subset of $\widehat{\bG}_s$ consisting of those character sheaves occurring as summands of the induced complex $\ind_{\bL}^{\bG}(A_{\mathscr{L}})$. From the discussion in \cref{pa:weyl-group-chars-symbols} we have $\Irr(W_{\bG}(\bL,\mathscr{L}))$ is in bijection with the set
\begin{equation*}
\mathbb{A}_{a'} \times \mathbb{B}_{b'} = \begin{cases}
\overline{\mathbb{V}}_{a'}^0 \times \mathbb{V}_{b'}^1 &\text{if }\bG\text{ is of type }\C_n\text{ and }t=0\\
\overline{\mathbb{V}}_{a'}^0 \times \overline{\mathbb{V}}_{b'}^0 &\text{if }\bG\text{ is of type }\D_n\text{ and }t=0\\
\mathbb{V}_{a'}^1 \times \mathbb{V}_{b'}^1 &\text{otherwise.}
\end{cases}
\end{equation*}
where $t$ is as in \cref{tab:levis-1}. With this we may now describe a combinatorial bijection
\begin{equation}\label{eq:lab-embedd}
\varphi : \mathbb{A}_{a'} \times \mathbb{B}_{b'} \to \begin{cases}
\mathbb{V}_{a' + t(t+1)}^{2t+1} \times \mathbb{V}_{b'+ t(t+1)}^{2t+1} &\text{if }\bG\text{ is of type }\B_n\\
\mathbb{V}_{a'+4t^2}^{4t} \times \mathbb{V}_{b'+4t^2\pm 2t}^{4t\pm1} &\text{if }\bG\text{ is of type }\C_n\text{ and }t\neq 0\\
\overline{\mathbb{V}}_{a'}^0 \times \mathbb{V}_{b'}^1 &\text{if }\bG\text{ is of type }\C_n\text{ and }t=0\\
\mathbb{V}_{a'+4t^2}^{4t} \times \mathbb{V}_{b'+4t^2}^{4t} &\text{if }\bG\text{ is of type }\D_n\text{ and }t\neq 0\\
\overline{\mathbb{V}}_{a'}^0 \times \overline{\mathbb{V}}_{b'}^0 &\text{if }\bG\text{ is of type }\D_n\text{ and }t=0\\
\end{cases}
\end{equation}
defined by \cref{eq:bij-shift-def} unless $\bG$ is of type $\C_n$ or $\D_n$ and $t = 0$, in which case $\varphi$ is simply the identity. It is clear that $\varphi$ gives us an embedding of $\mathbb{A}_{a'}\times \mathbb{B}_{b'}$ into $\overline{\bOmega}_a \times \overline{\bPhi}_b$. Assume $(\Lambda,\Xi) \in \overline{\bOmega}_a \times \overline{\bPhi}_b$ is in the image of $\varphi$ then we denote by $\widetilde{A}(\Lambda,\Xi)$ the character sheaf occurring with non-zero multiplicity in $\ind_{\bL}^{\bG}(A_{\mathscr{L}})$ parameterised by $\varphi^{-1}(\Lambda,\Xi) \in \mathbb{A}_{a'} \times \mathbb{B}_{b'}$.
\end{pa}

\subsection{Action of the Frobenius Endomorphism}
\begin{pa}\label{pa:action-frob-char-sheaves}
It will be useful for us to understand the action of the Frobenius endomorphism on the character sheaves $\widehat{\bG}_s$. The multiplicities uniquely determine the bijection in \cref{thm:char-sheaves-main-theorem}, (see for instance \cite[Proposition 6.3]{digne-michel:1990:lusztigs-parametrization}), hence we can conclude the following. Let $\gamma$ be the automorphism of $\bV_a\times\bW_b$ induced by $F_w$ then we define an action of the cyclic group $\langle \gamma\rangle$ on $\overline{\bOmega}_a\times\overline{\bPhi}_b$ in the following way. We have $\gamma$ acts trivially unless $\gamma$ restricted to $\bV_a$, (resp.\ $\bW_b$), is the automorphism $\gamma_a$, (resp.\ $\gamma_b$), in which case $\gamma$ permutes the degenerate symbols in $\overline{\bOmega}_a$, (resp.\ in $\overline{\bPhi}_b$), and fixes all other symbols in $\overline{\bOmega}_a$, (resp.\ in $\overline{\bPhi}_b$). With this we have the bijection $\overline{\bOmega}_a\times\overline{\bPhi}_b \to \widehat{\bG}_s$ restricts to a bijection $(\overline{\bOmega}_a\times\overline{\bPhi}_b)^{\gamma} \to \widehat{\bG}_s^F$.
\end{pa}
%
\section{The Generalised Springer Correspondence}\label{sec:generalised-springer}
\begin{assumption}
Throughout the next two sections we will denote by $\bW$ the relative Weyl group $W_{\bG}(\bL)$, (with $\bL$ as in \cref{pa:assumptions}), and $\widetilde{\bW}$ to be the semidirect product $\bW \rtimes \langle F\rangle$. Furthermore we will denote by $\bH$ the subgroup $W_{\bG}(\bL,\mathscr{L})$ of $\bW$, (where $\mathscr{L}$ is also as in \cref{pa:assumptions}), and by $\widetilde{\bH}$ the subgroup $\bH\langle (F(w),F) \rangle \leqslant \widetilde{\bW}$, (c.f.\ \cref{pa:semidirect-products}), where $w^{-1} \in Z_{\mathscr{L}}$ is the unique element of minimal length.
\end{assumption}

\begin{pa}\label{pa:gen-spring-symbols}
For any even $N \geqslant 0$ we denote by $\widetilde{\mathbb{Y}}_N$ the set of all ordered pairs $\sbpair{X}{Y}$ of finite (possibly empty) sets $X\subset \mathbb{N}_0$ and $Y \subset \mathbb{N}$ satisfying the following conditions:
\begin{enumerate}[label=(\roman*)]
	\item $|X| + |Y|$ is odd
	\item Neither $X$ nor $Y$ contain consecutive integers
	\item $\sum_{x\in X} x + \sum_{y\in Y}y = n + \frac{1}{2}(|X|+|Y|)(|X|+|Y|-1)$.
\end{enumerate}
For any $k \geqslant 0$ we define a shift operation ${}^{+k} : \widetilde{\mathbb{Y}}_N \to \widetilde{\mathbb{Y}}_N$ in the following way. If $\Lambda = \sbpair{X}{Y} \in \widetilde{\mathbb{Y}}_N$ then we define $\Lambda^{+k}$ to be $\binom{X'}{Y'}$ where
\begin{align*}
X' &= \{0,2,\dots,2k-2\}\cup\{x+2k \mid x \in X\},\\
Y' &= \{1,3,\dots,2k-1\}\cup\{y+2k \mid y \in Y\}
\end{align*}
We now define an equivalence relation denoted $\sim$ on $\widetilde{\mathbb{Y}}_N$ by setting $\Lambda \sim \Xi$ if and only if there exists $k \geqslant 0$ satisfying $\Lambda^k = \Xi$ or $\Lambda = \Xi^k$. We denote by $\mathbb{Y}_N$ the resulting equivalence classes.

For any $N \geqslant 0$ we denote by $\widetilde{\mathbb{X}}_N$ the set of all unordered pairs $\sbpair{A}{B}$ of finite (possibly empty) sets $A, B \subset \mathbb{N}_0$ which satisfy condition (ii) above and the condition
\begin{enumerate}
	\item[(iii')] $\sum_{a\in A} a + \sum_{b\in B}b = \frac{1}{2}N + \frac{1}{2}((|A|+|B| -1)^2-1)$.
\end{enumerate}
For any $k \geqslant 0$ we define a shift operation ${}^{+k} : \widetilde{\mathbb{X}}_N \to \widetilde{\mathbb{X}}_N$ given by $\Lambda^{+k} = \sbpair{A'}{B'}$ where
\begin{align*}
A' &= \{0,2,\dots,2k-2\}\cup\{a+2k \mid a \in A\},\\
B' &= \{0,2,\dots,2k-2\}\cup\{b+2k \mid b \in B\}
\end{align*}
We now define an equivalence relation denoted $\sim$ on $\widetilde{\mathbb{X}}_N$ by setting $\Lambda \sim \Xi$ if and only if there exists $k \geqslant 0$ satisfying $\Lambda^k = \Xi$ or $\Lambda = \Xi^k$. We denote by $\mathbb{X}_N$ the resulting equivalence classes.

It will be understood that by $\Psi_N$ we mean one of the sets $\mathbb{Y}_N$ or $\mathbb{X}_N$, where $N$ is always even in the first case. As in the case of symbols we would like to define an equivalence relation $\equiv$ on $\Psi_N$ which we will again refer to as similarity. Assume $\ssymb{A_1}{B_1},\ssymb{A_2}{B_2} \in \Psi_N$ are two pairs then we say they are \emph{similar} if they can be represented so that $A_1 \cup B_1 = A_2 \cup B_2$ and $A_1 \cap B_1 = A_2 \cap B_2$.
\end{pa}

\begin{pa}
Assume $\Lambda = \ssymb{A}{B} \in \Psi_N$ then we call $d(\Lambda) = |A| - |B|$ the defect of $\Lambda$, this is well defined as it is invariant under shift. Let $\Psi_N^d$ be the set of all symbols $\Lambda \in \Psi_N$ satisfying $d(\Lambda) = d$ then we have disjoint unions
\begin{equation*}
\mathbb{Y}_N = \bigsqcup_{\substack{d\in \mathbb{Z}\\ d\text{ odd}}}\mathbb{Y}_N^d \qquad\qquad \mathbb{X}_N = \bigsqcup_{\substack{d \geqslant 0\\ d\text{ even}}}\mathbb{X}_N^d \qquad\qquad \mathbb{X}_N = \bigsqcup_{\substack{d \geqslant 1\\ d\text{ odd}}}\mathbb{X}_N^d.
\end{equation*}
Assume $\ssymb{A}{B} \in \Psi_N^1$ then the map given by
\begin{equation}\label{eq:bijection-defect-d}
\begin{aligned}
\symb{A}{B} &\mapsto \symb{\{0,2,\dots,2d-4\} \cup \{a+2d-2 \mid a \in A\} }{B} &&\text{if }d\geqslant 1,\\
\symb{A}{B} &\mapsto \symb{A}{\{1,3,\dots,1-2d\} \cup \{b-2d+2 \mid b \in B\}} &&\text{if }d \leqslant -1
\end{aligned}
\end{equation}
defines natural bijections $\mathbb{Y}_N^1 \to \mathbb{Y}_{N+d(d-1)}^d$ for all odd integers $d \in \mathbb{Z}$ and bijections $\mathbb{X}_N^1 \to \mathbb{X}_{N + d^2 - 1}^d$ for all $d \in \mathbb{N}$ which we denote by $\Pi_N^d$. Note that if $d=1$ then this map should be interpreted as the identity.
\end{pa}

\begin{pa}
For each $n \geqslant 0$ we define two bijections $\Theta : \mathbb{V}_n^0 \to \mathbb{Y}_{2n}^1$ and $\Theta' : \mathbb{V}_n^0 \to \mathbb{X}_{2n+1}^1$ and a surjective map $\overline{\Theta}: \overline{\mathbb{V}}_n^0 \to \mathbb{X}_{2n}^0$ in the following way. Let $\ssymb{A}{B} \in \mathbb{V}_n^0$ be a symbol of defect 0 such that $A = \{a_1<a_2<\dots<a_s\}$ and $B = \{b_1<b_2<\dots<b_s\}$ then we define subsets $X$, $X'$ $Y$, $Y' \subset \mathbb{N}_0$ by setting
\begin{equation}\label{eq:Irr-Wn-bij}
\begin{aligned}
X &= \{0 < a_1 + 2 < a_2 + 3 < \cdots < a_s + (s+1)\},\\
X' &= \{a_1 < a_2 + 1 < \cdots < a_s + (s-1)\},\\
Y &= \{b_1 + 1 < b_2 + 2 < \cdots < b_s + s\},\\
Y'&= \{b_1 < b_2 + 1 < \cdots < b_s + (s - 1)\}.
\end{aligned}
\end{equation}
The bijections $\Theta$ and $\Theta'$ are then defined by setting $\Theta\ssymb{A}{B} = \ssymb{X}{Y}$ and $\Theta'\ssymb{A}{B} = \ssymb{X}{Y'}$. The map $\overline{\Theta}$ is defined by setting $\overline{\Theta}\ssymb{A}{B} = \ssymb{X'}{Y'}$, note that if $\ssymb{A}{B}$ is degenerate then both $\ssymb{A}{B}_{\pm}$ are mapped to the same element of $\mathbb{X}_{2n}^1$ hence the map is not bijective in general. We also define a variant of the bijection $\Theta$ denoted ${}^T\Theta : \mathbb{V}_n^0 \to \mathbb{Y}_{2n}^1$ which is defined by ${}^T\Theta\ssymb{A}{B} = \Theta\ssymb{B}{A}$. Note that these bijections are obtained by combining the bijections given in \cref{pa:weyl-group-chars-symbols} together with the maps given in (12.2.4), (13.2.5) and (13.2.6) of \cite{lusztig:1984:intersection-cohomology-complexes}.

Assume $m$ and $t$ are as in \cref{tab:levis-1} then we will denote by $d$ the integer $2t+1$, (resp.\ $4t+1$, $1-4t$, $4t$) if $\bG$ is of type $\B_n$, (resp.\ $\C_n$, $\C_n$, $\D_n$), and $t \geqslant 0$, (resp.\ $t\geqslant 0$, $t\geqslant 1$, $t \geqslant 1$). We now define a map
\begin{equation*}
\Delta_m^d = \begin{cases}
\Pi_{2m+1}^{2t+1} \circ \Theta' : \mathbb{V}_m^0 \to \mathbb{X}_{2m+1}^1 \to \mathbb{X}_{2n+1}^{2t+1} &\text{if }\bG\text{ of type }\B_n\text{ and }t\geqslant 0,\\
\Pi_{2m}^{4t+1}\circ\Theta : \mathbb{V}_m^0 \to \mathbb{Y}_{2m}^1 \to \mathbb{Y}_{2n}^{4t+1} &\text{if }\bG\text{ of type }\C_n\text{ and }t\geqslant 0,\\
\Pi_{2m}^{1-4t}\circ{}^T\Theta : \mathbb{V}_m^0 \to \mathbb{Y}_{2m}^1 \to \mathbb{Y}_{2n}^{1-4t} &\text{if }\bG\text{ of type }\C_n\text{ and }t\geqslant 1,\\
\Pi_{2m}^{4t} \circ \Theta' : \mathbb{V}_m^0 \to \mathbb{X}_{2m}^1 \to \mathbb{X}_{2n}^{4t} &\text{if }\bG\text{ of type }\D_n\text{ and }t\geqslant 1,\\
\overline{\Theta} : \overline{\mathbb{V}}_m^0 \to \mathbb{X}_{2n}^0 &\text{if }\bG\text{ of type }\D_n\text{ and }t=0.
\end{cases}
\end{equation*}
This map defines in most cases an embedding of $\Irr(\bW)$ into $\Psi_N$, (the only case which is not injective is in type $\D_n$ when $t=0$). Lusztig has also defined explicitly a relationship between $\Psi_N$ and $\mathcal{N}_{\bG}$ and the composition of these descriptions gives the generalised Springer correspondence. Here we will only need the map $\Delta_m^d$ and the following property of the generalised Springer correspondence, (which is found in part (c) of Corollaries 12.4 and 13.4 of \cite{lusztig:1984:intersection-cohomology-complexes}).

\begin{lem}[Lusztig]\label{lem:gen-spring-same-class}
Assume $\iota$, $\iota' \in \mathcal{N}_{\bG}$ are two pairs such that $\mathcal{O}_{\iota}$ and $\mathcal{O}_{\iota'}$ are non-degenerate. Let $\Lambda_{\iota}$, $\Lambda_{\iota'} \in \Psi_N$ be the corresponding symbols then $\mathcal{O}_{\iota} = \mathcal{O}_{\iota'}$ if and only if $\Lambda_{\iota} \equiv \Lambda_{\iota'}$.
\end{lem}
\end{pa}

\begin{pa}\label{pa:symb-a-value}
Given a symbol $\Lambda \in \Psi_N$ we write the entries of $\Lambda$ in a single row in increasing order $x_{\epsilon} \leqslant x_{\epsilon+1} \leqslant \cdots \leqslant x_{2k}$, where $\epsilon \in \{0,1\}$ is such that $d(\Lambda) + 1 \equiv \epsilon \pmod{2}$, then we define
\begin{equation}\label{eq:def-sigma}
\sigma(\Lambda) = \sum_{\epsilon \leqslant i < j \leqslant 2k} x_i = \sum_{i = \epsilon}^{2k} (2k-i)x_i,
\end{equation}
which gives us a function $\sigma : \Psi_N \to \mathbb{N}$. Let $x_{\epsilon}^0 \leqslant x_{\epsilon+1}^0 \leqslant \cdots \leqslant x_{2k}^0$ denote the sequence
\begin{equation}\label{eq:seq-x0}
\begin{aligned}
&0 \leqslant 1 \leqslant 2 \leqslant 3 \leqslant \cdots \leqslant 2k-1 \leqslant 2k &&\text{if }\Psi_N = \mathbb{Y}_{2n}\\
&0 \leqslant 0 \leqslant 2 \leqslant 2 \leqslant \cdots \leqslant 2(k-1) \leqslant 2(k-1) &&\text{if }\Psi_N = \mathbb{X}_{2n}\\
&0 \leqslant 0 \leqslant 2 \leqslant 2 \leqslant \cdots \leqslant 2(k-1) \leqslant 2(k-1) \leqslant 2k & &\text{if }\Psi_N = \mathbb{X}_{2n+1}
\end{aligned}
\end{equation}
then we define
\begin{equation*}
b(\Lambda) = \sum_{\epsilon\leqslant i < j \leqslant 2k} (x_i-x_i^0) = \sum_{i = \epsilon}^{2k} (2k-i)(x_i-x_i^0),
\end{equation*}
which gives us a second function $b : \Psi_N \to \mathbb{N}_0$ which we refer to as Lusztig's $b$-function. We now define $\sigma^d$, (resp.\ $b^d$), to be the composition $\sigma\circ\Delta_m^d$, (resp.\ $b\circ\Delta_m^d$). Note that it is clear from the definitions that for any $\Lambda,\Xi \in \Psi_N$ we have
\begin{equation}\label{eq:order-implications}
\sigma(\Lambda) \geqslant \sigma(\Xi) \Leftrightarrow b(\Lambda) \geqslant b(\Xi).
\end{equation}
Lusztig's $b$-function has the following interpretation under the generalised Springer correspondence. Assume $\Lambda \in \Irr(\bW)$ corresponds to $\iota \in \mathcal{N}_{\bG}$ under the generalised Springer correspondence, (c.f.\ \cref{eq:gen-spring-cor}), then we have
\begin{equation}\label{eq:b-val-dimBu}
b^d(\Lambda) = \dim\mathfrak{B}_u^{\bG}
\end{equation}
where $u \in \mathcal{O}_{\iota}$ is a class representative, (see \cite[\S11-\S13]{lusztig:1984:intersection-cohomology-complexes} and \cite[\S5]{shoji:1997:unipotent-characters-of-finite-classical-groups}). As we will use it often we recall here that for any unipotent class $\mathcal{O}$ of $\bG$ we have the following dimension formula, (see \cite[II - 2.8]{spaltenstein:1982:classes-unipotentes}),
\begin{equation}\label{eq:dim-formula}
\dim\mathcal{O} = |\Phi| - 2\dim\mathfrak{B}_u^{\bG}
\end{equation}
where $u \in \mathcal{O}$ is any class representative and $\Phi$ are the roots of $\bG$ defined with respect to our choice of $\bT_0$.
\end{pa}
%
\section{\texorpdfstring{Minimising Lusztig's $b$-function}{Minimising Lusztig's b-function}}
\begin{assumption}
Unless otherwise stated we will assume that $\Irr(W_n)$ is parameterised by $\mathbb{V}_n^0$.
\end{assumption}

\begin{pa}
Assume $n \in \mathbb{N}$ and $a,b \in \mathbb{N}_0$ are such that $n = a+b$ then the symmetric group $\mathfrak{S}_n$ has a natural maximal rank parabolic subgroup $\mathfrak{S}_a\times\mathfrak{S}_b$. Assume $X \in \mathbb{W}_a$ and $Y \in \mathbb{W}_b$ then for any $Z \in \mathbb{W}_n$ we define
\begin{equation}\label{ind:S_n}
c_{XY}^Z = \langle \Ind_{\mathfrak{S}_a \times \mathfrak{S}_b}^{\mathfrak{S}_n}(X\boxtimes Y), Z\rangle_{\mathfrak{S}_n}.
\end{equation}
The multiplicities $c_{XY}^{Z}$ are called Littlewood--Richardson coefficients, (or LR-coefficients for short). To better understand these coefficients we need to introduce the following partial order on $\mathbb{W}_n$. Assume $A$, $B \in \mathbb{W}_n$ are fixed representatives of their equivalence classes so that $A = \{a_1 < \cdots < a_s\}$ and $B = \{b_1 < \cdots < b_s\}$. Then we write $B \unlhd A$ if for all $k \in \{1,\dots,s\}$ we have
\begin{equation*}
\sum_{i=k}^s b_i \leqslant \sum_{i=k}^s a_i.
\end{equation*}
Note that $B \unlhd A$ if and only if $B^{+1}\unlhd A^{+1}$ hence this is well defined on equivalence classes. The relation $\unlhd$ is known as the \emph{dominance ordering} on $\mathbb{W}_n$.

We will now define an operation $\odot : \mathbb{W}_a \times \mathbb{W}_b \to \mathbb{W}_n$ in the following way. Firstly let $A = \{a_1,\dots,a_s\} \in \mathbb{W}_a$ and $B = \{b_1,\dots,b_s\} \in \mathbb{W}_b$ be fixed representatives for their equivalence classes then we define
\begin{equation*}
A\odot B = \{a_1+b_1< a_2+b_2-1 < \cdots < a_s+b_s-(s-1)\}.
\end{equation*}
This also defines a similar operation $\odot : \mathbb{V}_a^0 \times \mathbb{V}_b^0 \to \mathbb{V}_n^0$ by setting $\ssymb{A_1}{B_1} \odot \ssymb{A_2}{B_2} \to \ssymb{A_1\odot A_2}{B_1 \odot B_2}$ as long as we assume $|A_1| = |A_2| = |B_1| = |B_2|$. With this in hand we have the following result concerning the LR-coefficients.
\end{pa}

\begin{lem}[{}{see \cite[Lemma 6.1.2]{geck-pfeiffer:2000:characters-of-finite-coxeter-groups}}]\label{ind:sym-group}
Assume $X \in \mathbb{W}_a$ and $Y \in \mathbb{W}_b$ then for any $Z \in \mathbb{W}_n$ we have $c_{XY}^Z \neq 0$ implies $Z \unlhd X\odot Y$. Moreover if $Z = X\odot Y$ then $c_{XY}^Z = 1$.
\end{lem}

\begin{pa}
In what follows we will only talk about the set $\mathbb{V}_m^0$ however when $\bG$ is of type $\D_n$ and $t = 0$ the domain of $\Delta_m^d$ is $\overline{\mathbb{V}}_m^0$. In this case we interpret the results by simply forgetting the ordering of the symbols to obtain an element of $\overline{\mathbb{V}}_m^0$, (note that if the symbol is degenerate then it doesn't matter how we choose the $\pm$ as $\Delta_m^d(\Lambda_+) = \Delta_m^d(\Lambda_-)$ for all degenerate symbols $\Lambda \in \mathbb{V}_m^0$).
\end{pa}

\begin{prop}[{}{Aubert, {\cite[Theorem 4.4]{aubert:2003:characters-sheaves-and-generalized}}}]\label{prop:aubert}
Assume $\ssymb{A_1}{A_2}$, $\ssymb{B_1}{B_2} \in \mathbb{V}_m^0$ are such that $\rk(A_i) = \rk(B_i)$ for $i \in \{1,2\}$. If $A_i \unlhd B_i$ for $i \in \{1,2\}$ then we have $b^d\ssymb{B_1}{B_2} \leqslant b^d\ssymb{A_1}{A_2}$ with equality if and only if $A_i = B_i$.
\end{prop}

\begin{rem}
The above result of Aubert was proved under the assumption that $t \neq 0$ but it is easily seen that her proof works unchanged in the case $t=0$, thus we do not reproduce it here. For additional details see also \cite[Annexe A]{hezard:2004:thesis}.
\end{rem}

\begin{definition}
Assume $E \in \Irr(\bH)^{Fw}$ is an $F_w$-invariant character of $\bH$ then we say a character $X \in \Irr(\bW)$ is \emph{$d$-minimal for $E$} if the following properties hold:
\begin{enumerate}[label=(\roman*)]
	\item $X$ is $F$-stable,
	\item $\langle \Ind_{\bH.Fw}^{\bW.F}(\widetilde{E}),\widetilde{X} \rangle_{\bW.F} \neq 0$,
	\item for all $Y \in \Irr(\bW)^F$ satisfying $\langle \Ind_{\bH.Fw}^{\bW.F}(\widetilde{E}),\widetilde{Y}\rangle_{\bW.F} \neq 0$ we have $b^d(Y) \leqslant b^d(X)$.
\end{enumerate}
\end{definition}

\begin{prop}\label{cor:multiplicity-1}
Let $E = A \boxtimes B = \ssymb{A_1}{A_2} \boxtimes \ssymb{B_1}{B_2} \in \Irr(\bH)^{Fw}$ and denote by $\Lambda$, (resp.\ $\Lambda'$), the character $\ssymb{A_1 \odot B_1}{A_2 \odot B_2} \in \Irr(\bW)$, (resp.\ $\ssymb{A_2 \odot B_1}{A_1 \odot B_2} \in \Irr(\bW)$). If $X \in \Irr(\bW)$ is $d$-minimal for $E$ then the following hold:
\begin{enumerate}[label=(\alph*)]
	\item Assume either $t\neq 0$ or $\bG$ is of type $\B_n$ then $X = \Lambda$.
	
	\item Assume $t=0$ and $\bG$ is of type $\C_n$. If $F_w$ is the identity and $A$ is non-degenerate or $F_w$ is the automorphism $\gamma_a\times\ID$ then $X \in \{\Lambda,\Lambda'\}$. If $F_w$ is the identity and $A$ is degenerate then $X = \Lambda$.
	
	\item Assume $t=0$, $\bG$ is of type $\D_n$ and $F$ is the identity.
	\begin{itemize}
		\item Assume $F_w$ is the identity. If one of $A$ or $B$ is non-degenerate then $X \in \{\Lambda,\Lambda'\}$. If both $A$ and $B$ are degenerate then $X = \Lambda$.
		\item Assume $F_w$ is the automorphism $\gamma_a\times\gamma_b$ then $X \in \{\Lambda,\Lambda'\}$.
	\end{itemize}
	
	Assume $Y \in \{\Lambda,\Lambda'\}$ is degenerate then we have the following. If $A$ or $B$ is degenerate then either $Y_+$ or $Y_-$ is $d$-minimal for $E$ and if neither $A$ nor $B$ is degenerate then both $Y_+$ and $Y_-$ is $d$-minimal for $E$.
\end{enumerate}
In all of the above cases if $X \in \Irr(\bW)$ is $d$-minimal for $E$ then we have $\langle \Ind_{\bH.Fw}^{\bW.F}(\widetilde{E}),\widetilde{X}\rangle_{\bW.F} = 1$.
\begin{enumerate}[resume,label=(\alph*)]
	\item Assume $t=0$, $\bG$ is of type $\D_n$ and $F$ is the automorphism $\gamma_n$, (c.f. \cref{sec:weyl-groups}). Furthermore let us assume that one of $\Lambda$ or $\Lambda'$ is non-degenerate then $X \in \{\Lambda,\Lambda'\}$ and $\langle \Ind_{\bH.Fw}^{\bW.F}(\widetilde{E}),\widetilde{X}\rangle_{\bW.F} = \pm1$.
\end{enumerate}
\end{prop}

\begin{proof}
(a). This is simply \cite[Corollary 4.3]{aubert:2003:characters-sheaves-and-generalized} but we recall it here for the other parts. By \cite[Lemma 6.1.3]{geck-pfeiffer:2000:characters-of-finite-coxeter-groups} we have the induced character is explicitly given by
\begin{equation}\label{eq:type-C-weyl-induction}
\Ind_{\bH}^{\bW}(E) = \sum c_{A_1B_1}^{X_1}c_{A_2B_2}^{X_2}\symb{X_1}{X_2},
\end{equation}
where the sum is over all symbols $\ssymb{X_1}{X_2} \in \mathbb{V}_n^0$ satisfying $\rk(X_i) = \rk(A_i)+\rk(B_i)$. Assume the coefficient of $\ssymb{X_1}{X_2}$ is non-zero then by \cref{ind:sym-group} we have $X_i \unlhd A_i + B_i$ for $i \in \{1,2\}$. The result then follows immediately from \cref{prop:aubert}.

(b). Let us first treat the case when $F_w$ is the identity. Let $\widetilde{E}_1 = \ssymb{A_1}{A_2} \boxtimes \ssymb{B_1}{B_2}$, $\widetilde{E}_2 = \ssymb{A_2}{A_1} \boxtimes \ssymb{B_1}{B_2}$ be the two characters of $\Irr(W_a\times W_b)$ extending $E$. We have the following sequence of subgroups $\bH \leqslant W_a \times W_b \leqslant \bW$ and applying transitivity of induction we see that
\begin{equation*}
\Ind_{\bH}^{\bW}(E) = \begin{cases}
\Ind_{W_a\times W_b}^{\bW}(\widetilde{E}_1)+\Ind_{W_a\times W_b}^{\bW}(\widetilde{E}_2 ) &\text{if }A_1 \neq A_2,\\
\ind_{W_a\times W_b}^{\bW}(\widetilde{E}_1) &\text{otherwise}.
\end{cases}
\end{equation*}
In the degenerate case (b) follows from (a) so let us assume that $A_1 \neq A_2$ then we claim $\Lambda'$, (resp.\ $\Lambda$), does not occur in the induction of $\widetilde{E}_1$, (resp.\ $\widetilde{E}_2$). We will only prove the statement for $\Lambda'$, the other case is identical. Assume for a contradiction that $\Lambda'$ does occur in the induction of $\widetilde{E}_1$ then we must have $c_{A_1B_1}^{A_2\odot B_1}c_{A_2B_2}^{A_1\odot B_2} \neq 0$, in particular
\begin{equation*}
\left.\begin{aligned}
A_2 \odot B_1 &\unlhd A_1 \odot B_1\\
A_1 \odot B_2 &\unlhd A_2 \odot B_2
\end{aligned}\right\} \Rightarrow
\left.\begin{aligned}
A_2 &\unlhd A_1\\
A_1 &\unlhd A_2
\end{aligned}\right\} \Rightarrow
A_1 = A_2
\end{equation*}
but this is a contradiction as we assumed $A_1 \neq A_2$, hence this proves the claim. Now let us assume $F_w$ is $\gamma_a\times\ID$ then $\widetilde{\bH}$ is the subgroup $W_a\times W_b$ of $\bW$ and this case follows immediately from (a).

(c). Formally the argument is the same as that in part (a) except now we must use \cref{lem:typeD-GP-Lemm6.1.3} instead of \cite[Lemma 6.1.3]{geck-pfeiffer:2000:characters-of-finite-coxeter-groups}. We deal first with the case where $F_w$ is the identity. With the notation as in \cref{lem:typeD-GP-Lemm6.1.3} we have the induced character is given by
\begin{equation*}
\Ind_{W_a'\times W_b'}^{W_n'}(E) = \sum_{X_1\neq X_2} a_{AB}^{X_1X_2}\symb{X_1}{X_2} + \sum_{X_1 = X_2} a_{AB}^{X_1,\pm}\symb{X_1}{X_2}_{\pm}
\end{equation*}
where the sum is over all $\ssymb{X_1}{X_2} \in \overline{\mathbb{V}}_n^0$ such that there exist indexing sets $\{i_1,i_2\} = \{j_1,j_2\} = \{1,2\}$ satisfying $\rk(X_k) = \rk(A_{i_k})+\rk(B_{j_k})$ for $k \in \{1,2\}$. Assume the coefficient of $\ssymb{X_1}{X_2}$ is non-zero then there exist permutations $\sigma$, $\rho \in \mathfrak{S}_2$ such that $X_i \unlhd A_{\sigma(i)} + B_{\rho(i)}$. The exact same argument as used in (a) shows that $\Lambda$ or $\Lambda'$ is $d$-minimal for $E$. We now consider the multiplicity of this character. We will treat this case by case as follows, (note that it is sufficient to show that $a_{AB}^{X_1X_2}=1$).
\begin{itemize}
	\item Firstly it is clear that $a_{AB}^{X_1X_2} = 1$ unless one of $A$ or $B$ is non-degenerate. Assume that $A$ is non-degenerate but $B$ is degenerate then
\begin{equation*}
a_{AB}^{X_1X_2} = c_{A_{i_1}B_{j_1}}^{X_1}c_{A_{i_2}B_{j_2}}^{X_2} + \delta_Ac_{A_{i_2}B_{j_1}}^{X_1}c_{A_{i_1}B_{j_2}}^{X_2},
\end{equation*}
with $i_1$, $i_2$, $j_1$ and $j_2$ as above. Now assume both terms are non-zero then the argument used in (b) shows $A_{i_1} = A_{i_2}$, a contradiction. The case of $A$ degenerate but $B$ non-degenerate is treated similarly.
	\item Assume now that $A$ and $B$ are both non-degenerate then we have
	\begin{equation*}
	a_{AB}^{X_1X_2} = c_{A_{i_1}B_{j_1}}^{X_1}c_{A_{i_2}B_{j_2}}^{X_2} + \delta_Ac_{A_{i_2}B_{j_1}}^{X_1}c_{A_{i_1}B_{j_2}}^{X_2} + \delta_Bc_{A_{i_1}B_{j_2}}^{X_1}c_{A_{i_2}B_{j_1}}^{X_2} + \delta_{A\odot B}c_{A_{i_2}B_{j_2}}^{X_1}c_{A_{i_1}B_{j_1}}^{X_2}.
	\end{equation*}
	Assume the first and fourth terms, (resp.\ the second and third terms), are both non-zero then this implies $A_{i_1} + B_{j_1} = X_1 = X_2 = A_{i_2} + B_{j_2}$, (resp.\ $A_{i_1}+B_{j_2} = X_1 = X_2 = A_{i_2}+B_{j_2}$), hence $a_{AB}^{X_1X_2} = 2$ and $\ssymb{X_1}{X_2}$ is degenerate. As $A$ and $B$ are non-degenerate we have by \cref{lem:typeD-GP-Lemm6.1.3} that the coefficient of the degenerate character is $\frac{1}{2}a_{AB}^{X_1X_2} = 1$ as required. Assume the second and fourth terms, (resp.\ third and fourth terms), are simultaneously non-zero then we can argue as in (b) to deduce that $B$, (resp.\ $A$), is degenerate, which is a contradiction. This finishes the proof of the multiplicity statement.
\end{itemize}
Assume now that $F_w$ is the automorphism $\gamma_a\times\gamma_b$ then $\widetilde{\bH}$ is the subgroup $(W_a'\times W_b')\langle s_nt_0 \rangle \leqslant W_n'$. The extension $\widetilde{E} \in \Irr(\widetilde{\bH})$ of $E \in \Irr(\bH)^{Fw}$ is a non-degenerate character of $\widetilde{\bH}$, (in the sense of \cref{pa:the-subgroup-H}), hence the arguments above together with \cref{lem:subgroup-H-W_n'} easily yield the result.

(d). This case is easily delt with using the previous methods. We simply recall from \cref{sec:weyl-groups} that we may identify $\widetilde{\bW}$ with $W_n$ and $\widetilde{\bH}$ with either $W_a \times W_b'$ or $W_a'\times W_b$.
\end{proof}

\begin{rem}\label{rem:C_4-exmp}
Note that the `or' statements in the proposition are not in general exclusive. Indeed one can easily construct the following example. Let $\bG$ be of type $\C_4$, take $\bH \leqslant \bW$ to be the subgroup $W_4' \leqslant W_4$ and let $F_w$ be the identity. Take $E$ to be the irreducible character $\ssymb{1,2}{0,3}$ then clearly the induced character is given by $\ssymb{1,2}{0,3} + \ssymb{0,3}{1,2}$. As $t= 0$ we have $d=1$ and the corresponding symbols are given by
\begin{equation*}
\Delta_4^1\symb{1,2}{0,3} = \symb{0,3,5}{1,5}\qquad\qquad
\Delta_4^1\symb{0,3}{1,2} = \symb{0,2,6}{2,4}.
\end{equation*}
It is readily checked that $b^1\ssymb{1,2}{0,3} = b^1\ssymb{0,3}{1,2} = 4$. Note that this is not an isolated case, indeed this phenomenon only becomes more frequent as we increase the semisimple rank of $\bG$. However we also notice that $a\ssymb{1,2}{0,3}=3 < 4$ which we will see plays a role later.
\end{rem}
%
\section{Unipotent Supports}
Note that in this section we use heavily the conditions that $\bG$ has a connected centre and $p$ is good for $\bG$. When $\bG$ has a disconnected centre the constructions we consider here are more subtle, (see \cite[\S10]{lusztig:1992:a-unipotent-support}).

\begin{pa}\label{pa:unip-supp-construction}
Let us denote by $H$ any reflection subgroup of $W_{\bG^{\star}}$ then $H$ is a disjoint union of two-sided cells $\mathfrak{C} \subseteq H$, defined as in \cite[Definition 2.1.15]{geck-jacon:2011:hecke-algebras}. Note that $W_{\bG^{\star}}$ acts naturally on the two-sided cells of $H$ by conjugation. Each two-sided cell $\mathfrak{C}$ defines a corresponding cell module $M_{\mathfrak{C}}$ of $H$, (see \cite[\S1.6.2]{geck-jacon:2011:hecke-algebras}), and we define
\begin{equation*}
\Irr(H|\mathfrak{C}) := \{E \in \Irr(H) \mid \langle E, M_{\mathfrak{C}}\rangle \neq 0\}
\end{equation*}
to be the \emph{family of characters} determined by $\mathfrak{C}$. This gives us a partition
\begin{equation*}
\Irr(H) = \bigsqcup \Irr(H|\mathfrak{C}),
\end{equation*}
where the union runs over all two-sided cells of $H$. Each family $\Irr(H|\mathfrak{C})$ contains a unique \emph{special character}, (see \cite[(4.1.4)]{lusztig:1984:characters-of-reductive-groups}), which we denote by $E_{\mathfrak{C}}$. Let $j_H^{W_{\bG^{\star}}} : \Cent(H) \to \Cent(W_{\bG^{\star}})$ denote Lusztig--Macdonald--Spaltenstein induction, defined as in \cite[(4.1.9)]{lusztig:1984:characters-of-reductive-groups}, then for each special character $E_{\mathfrak{C}} \in \Irr(H|\mathfrak{C})$ of $H$ we have $j_H^{W_{\bG^{\star}}}(E_{\mathfrak{C}}) \in \Irr(W)$ is irreducible, (see \cite[1.3]{lusztig:2009:unipotent-classes-and-special-Weyl}). Using the duality isomorphism $W_{\bG} \to W_{\bG^{\star}}$ we may identify $\Irr(W_{\bG})$ with $\Irr(W_{\bG^{\star}})$. Under this identification we have by \cite[Theorem 1.5]{lusztig:2009:unipotent-classes-and-special-Weyl} that $j_H^{W_{\bG^{\star}}}(E_{\mathfrak{C}})$ corresponds to some $E_{\iota} \in \Irr(W_{\bG})$, (c.f.\ \cref{pa:gen-spring-cor}), with $\iota = (\mathcal{O},\Ql) \in \mathscr{I}[\bT_0,\Ql] \subseteq \mathcal{N}_{\bG}$; we denote by $\mathcal{O}_{\mathfrak{C}}$ the class $\mathcal{O}$. With this in hand we have defined a map
\begin{equation}\label{eq:cells-to-classes}
\begin{aligned}
\{\text{two-sided cells of }H\} &\to \{\text{unipotent conjugacy classes of }\bG\}\\
\mathfrak{C} &\mapsto \mathcal{O}_{\mathfrak{C}}
\end{aligned}
\end{equation}
\end{pa}

\begin{pa}\label{pa:unip-supp-desc}
For each semisimple element $s \in \bT_0^{\star}$ the group $W_{\bG^{\star}}(s) = \{w \in W_{\bG^{\star}} \mid {}^{\dot{w}}s = s\}$ is a reflection subgroup of $W_{\bG^{\star}}$ and by \cite[(17.13.2)]{lusztig:1985:character-sheaves} we have a natural partition
\begin{equation*}
\widehat{\bG}_s = \bigsqcup\widehat{\bG}_{s,\mathfrak{C}}
\end{equation*}
where the union is over all the $W_{\bG^{\star}}$-orbits of pairs $(s,\mathfrak{C})$ under the action $w\cdot(s,\mathfrak{C})=(\dot{w}^{-1}s\dot{w},w^{-1}\mathfrak{C}w)$. Assume $A$ lies in $\widehat{\bG}_{s,\mathfrak{C}}$ then we denote by $\mathcal{O}_A$ the unipotent conjugacy class of $\bG$ determined by $\mathfrak{C}$ under the map in \cref{eq:cells-to-classes}. We call $\mathcal{O}_A$ the \emph{unipotent support} of $A$. Denote by $a : \Irr(W_{\bG^{\star}}(s)) \to \mathbb{N}$ Lusztig's $a$-function as defined in \cite[4.1]{lusztig:1984:characters-of-reductive-groups} then for any $A \in \widehat{\bG}_{s,\mathfrak{C}}$ we note here that
\begin{equation}\label{eq:a-val-dimBu}
a(E_{\mathfrak{C}}) = \dim\mathfrak{B}_u^{\bG}
\end{equation}
for any $u \in \mathcal{O}_A$, (see \cite[\S13.3]{lusztig:1984:characters-of-reductive-groups}).

We now turn our attention to the irreducible characters of $G$. Assume the $W_{\bG^{\star}}$-orbit containing $s \in \bT_0^{\star}$ is $F^{\star}$-stable then we denote by $\mathcal{E}(G,s,\mathfrak{C})$ the set of all $\psi \in \Irr(G)$ such that there exists an $F$-stable $A \in \widehat{\bG}_{s,\mathfrak{C}}^F$ satisfying $\langle \psi, \chi_A \rangle \neq 0$. Using \cite[Main Theorem 4.23]{lusztig:1984:characters-of-reductive-groups} and \cite[II - Theorem 3.2]{shoji:1995:character-sheaves-and-almost-characters} we obtain a partition
\begin{equation*}
\Irr(G) = \bigsqcup \mathcal{E}(G,s,\mathfrak{C})
\end{equation*}
where the union is taken over all the $F^{\star}$-stable $W_{\bG^{\star}}$-orbits of pairs $(s,\mathfrak{C})$. The union $\mathcal{E}(G,s) = \sqcup\mathcal{E}(G,s,\mathfrak{C})$ over all $F^{\star}$-stable $W_{\bG^{\star}}$-orbits of two sided cells $\mathfrak{C} \subseteq W_{\bG^{\star}}(s)$ is a \emph{Lusztig series} of $G$ in the usual sense. For each $\chi \in \mathcal{E}(G,s,\mathfrak{C})$ we denote by $\mathcal{O}_{\chi}$ the class determined by $\mathfrak{C}$ under the map in \cref{eq:cells-to-classes} and we call this the \emph{unipotent support} of $\chi$. Note that $\mathcal{O}_{\chi}$ is always $F$-stable. The unipotent support can be characterised in the following more natural way.
\end{pa}

\begin{thm}[{}{Lusztig \cite{lusztig:1992:a-unipotent-support}, Geck \cite{geck:1996:on-the-average-values}, Aubert \cite{aubert:2003:characters-sheaves-and-generalized}}]\label{thm:description-unipotent-support}
For each $X \in \widehat{\bG} \sqcup \Irr(\bG^F)$ we denote by $\mathcal{D}(X)$ the set of all unipotent conjugacy classes $\mathcal{O} \subset \bG$ satisfying
\begin{equation*}
\begin{cases}
X|_{\mathcal{O}} \neq 0 &\text{if }X \in \widehat{\bG}\\
F(\mathcal{O}) = \mathcal{O}\text{ and }\sum_{g\in\mathcal{O}^F}X(g) \neq 0 &\text{if }X \in \Irr(\bG^F)
\end{cases}
\end{equation*}
then the following hold:
\begin{enumerate}[label=(\roman*)]
	\item $\mathcal{O}_X \in \mathcal{D}(X)$,
	\item for all $\mathcal{O} \in \mathcal{D}(X)$ we have $\dim\mathcal{O} \leqslant \dim\mathcal{O}_X$ with equality if and only if $\mathcal{O} = \mathcal{O}_X$.
\end{enumerate}
\end{thm}

%
\section{Evaluating Characteristic Functions of Character Sheaves at Unipotent Elements}
\begin{pa}
Assume $A \in \widehat{\bG}^F$ is a unipotently supported character sheaf then we denote by $\widetilde{\mathcal{O}}_A$ an $F$-stable class of maximal dimension satisfying $\chi_A|_{\widetilde{\mathcal{O}}_A^F} \neq 0$. From the description of the unipotent support given in \cref{thm:description-unipotent-support} it is clear that $\dim \widetilde{\mathcal{O}}_A \leqslant \dim\mathcal{O}_A$ but it is possible that $\widetilde{\mathcal{O}}_A \neq \mathcal{O}_A$. Note that this is starkly different to the situation for irreducible characters of $G$, (see \cite[Theorem 11.2(v)]{lusztig:1992:a-unipotent-support}). Our goal in this section is to describe those character sheaves satisfying $\widetilde{\mathcal{O}}_A = \mathcal{O}_A$, however this will only be clear after having proved part (a) of \cref{prop:labels-coincide}. This description will be akin to that given by Lusztig in \cite[4.15(c)]{lusztig:1986:on-the-character-values}, where he showed that such character sheaves are described by a canonical lagrangian subspace of the usual parameterising set.
\end{pa}

\subsection{A Combinatorial Description}
\begin{pa}
With the above we now define two maps
\begin{equation*}
\Psi_N \leftarrow \Irr(\bV_{a'}\times\bW_{b'}) \rightarrow \overline{\bOmega}_a \times \overline{\bPhi}_b
\end{equation*}
where the right hand side map is defined by \cref{eq:lab-embedd}. Note that we may first need to apply the natural map $\mathbb{V}_n^0 \to \mathbb{V}_n^1$, (c.f.\ \cref{eq:bij-shift-def}). Keeping in mind \cref{cor:multiplicity-1} we may define the left hand side map by
\begin{equation*}
\symb{A_1}{A_2}\boxtimes\symb{B_1}{B_2} \mapsto \Delta_m^d\symb{A_1\odot B_1}{A_2\odot B_2}
\end{equation*}
These maps define directly a map $\overline{\bOmega}_a\times\overline{\bPhi}_b \to \Psi_N$ without reference to $\Irr(\bV_{a'}\times\bW_{b'})$. We would now like to describe this map explicitly. Given a pair of symbols $(\Lambda,\Xi) \in \overline{\bOmega}_a \times \overline{\bPhi}_b$ we denote by $\varepsilon(\Lambda,\Xi) = |d(\Lambda)-d(\Xi)|$ the modulus of the difference of the defects. Assume $(\Lambda,\Xi)$ satisfies $\epsilon(\Lambda,\Xi) \leqslant 1$ then, after possibly permuting the rows, there exists $t \in \mathbb{N}_0$ such that we are in one of the following cases:
\begin{enumerate}[label=(\alph*)]
	\item[(B)] $\bG$ is of type $\B_n$ and $(d(\Lambda),d(\Xi)) = (2t+1,2t+1)$,
	\item[(C.1)] $\bG$ is of type $\C_n$ and $(d(\Lambda),d(\Xi)) = (4t,4t+1)$,
	\item[(C.2)] $\bG$ is of type $\C_n$ and $(d(\Lambda),d(\Xi)) = (4t,4t-1)$,
	\item[(D)] $\bG$ is of type $\D_n$ and $(d(\Lambda),d(\Xi)) = (4t,4t)$,
\end{enumerate}
Let us write $\Lambda = \ssymb{A}{B}$ and $\Xi = \ssymb{A'}{B'}$. Applying shifts we may assume that there exists $p \geqslant 0$ such that $B = \{b_1,\dots,b_p\}$, $B' = \{b_1',\dots,b_p'\}$ and
\begin{equation}\label{eq:define-index}
\begin{aligned}
A &= \begin{cases}
\{a_1,\dots,a_{p+2t+1}\} &\text{case (B)}\\
\{a_1,\dots,a_{p+4t}\} &\text{case (C.1)}\\
\{a_0,\dots,a_{p+4t-1}\} &\text{case (C.2)}\\
\{a_1,\dots,a_{p+4t}\} &\text{case (D)}\\
\end{cases}
\end{aligned}
\qquad\qquad
\begin{aligned}
A' &= \begin{cases}
\{a_1',\dots,a_{p+2t+1}'\} &\text{case (B)}\\
\{a_0',\dots,a_{p+4t}'\} &\text{case (C.1)}\\
\{a_1',\dots,a_{p+4t-1}'\} &\text{case (C.2)}\\
\{a_1',\dots,a_{p+4t}'\} &\text{case (D)}
\end{cases}
\end{aligned}
\end{equation}
We now define an element $\Lambda\oplus\Xi \in \Psi_N$ by setting
\begin{equation*}
\Lambda \oplus \Xi = \begin{cases}
\symb{a_1+a_1',\dots,a_{p+2t+1}+a_{p+2t+1}'}{b_1+b_1',\dots,b_p+b_p'} &\text{case (B)}\\[15pt]
\symb{a_0',a_1+a_1'+1,\dots,a_{p+4t}+a_{p+4t}'+1}{b_1+b_1'+1,\dots,b_p+b_p'+1} &\text{case (C.1)}\\[15pt]
\symb{b_1+b_1',\dots,b_p+b_p'}{a_1+a_1',\dots,a_{p+4t-1}+a_{p+4t-1}'} &\text{case (C.2)}\\[15pt]
\symb{a_1+a_1',\dots,a_{p+4t}+a_{p+4t}'}{b_1+b_1',\dots,b_p+b_p'} &\text{case (D)}
\end{cases}
\end{equation*}
then our desired map $\overline{\bOmega}_a\times\overline{\bPhi}_b \to \Psi_N$ is defined by $(\Lambda,\Xi) \mapsto \Lambda\oplus\Xi$.
\end{pa}

\begin{rem}
In case (C.2) we have first applied the shift operation to both $\Lambda$ and $\Xi$ to obtain $\Lambda\oplus\Xi$. This will not affect our results because $\Lambda\oplus\Xi$ remains in the same equivalence class under shift.
\end{rem}

\begin{pa}\label{pa:symbol+partition}
One readily checks that for each $(\Lambda,\Xi) \in \overline{\bOmega}_a \times \overline{\bPhi}_b$ such that $\epsilon(\Lambda,\Xi) \leqslant 1$ we have $a(\Lambda) + a(\Xi) \leqslant b(\Lambda \oplus \Xi)$. Furthermore we have $a(\Lambda) + a(\Xi) = b(\Lambda \oplus \Xi)$ if and only if the following condition is verified:
\begin{equation}\label{eq:condition-equality}
\boxed{
\begin{gathered}
\text{For any $i\geqslant 1$ and $j \geqslant 1$, with the notation as in \cref{eq:define-index}, we have either}\\
a_i \leqslant b_j\quad\text{and}\quad a_i' \leqslant b_j' \qquad\qquad\text{or}\qquad\qquad a_i \geqslant b_j\quad\text{and}\quad a_i' \geqslant b_j'
\end{gathered}
}
\end{equation}
(see also \cite[4.12(b)]{lusztig:1986:on-the-character-values}). The reader easily convinces himself of the above condition after checking that $a(\Lambda)+a(\Xi) = b(\Lambda\oplus\Xi)$ holds if and only if $\sigma(\Lambda)+\sigma(\Xi) = \sigma(\Lambda\oplus\Xi)$ holds.
\end{pa}

\begin{pa}
We now assume that $\bG$ is of type $\C_n$ or $\D_n$ and $t=0$, i.e.\ $(\Lambda,\Xi) \in \overline{\bOmega}_a^1\times\overline{\bPhi}_b^1$. In this case the result of \cref{cor:multiplicity-1} provides another possibility for the map considered above. Let us keep the notation as in \cref{eq:define-index} then we also have the following additional alternatives
\begin{equation*}
\Lambda \boxplus \Xi = \begin{cases}
\symb{a_0',b_1+a_1'+1,\dots,b_p+a_p'+1}{a_1+b_1'+1,\dots,a_p+b_p'+1} &\text{case (C.1)}\\[15pt]
\symb{b_1+a_1',\dots,b_p+a_p'}{a_1+b_1',\dots,a_p+b_p'} &\text{case (D).}
\end{cases}
\end{equation*}
As before we see that $a(\Lambda) + a(\Xi) = b(\Lambda \boxplus \Xi)$ if and only if the following condition is verified:
\begin{equation}\label{eq:condition-equality-2}
\boxed{
\begin{gathered}
\text{For any $i\geqslant 1$ and $j \geqslant 1$, with the notation as in \cref{eq:define-index}, we have either}\\
b_i \leqslant a_j\quad\text{and}\quad a_i' \leqslant b_j' \qquad\qquad\text{or}\qquad\qquad b_i \geqslant a_j\quad\text{and}\quad a_i' \geqslant b_j'
\end{gathered}
}
\end{equation}

We now wish to consider those symbols $(\Lambda,\Xi) \in \overline{\bOmega}_a^1\times\overline{\bPhi}_b^1$ satisfying $b(\Lambda\oplus\Xi) = a(\Lambda) + a(\Xi) = b(\Lambda\boxplus\Xi)$. We claim that this is the case if and only if the following condition is verified:
\begin{equation}\label{eq:condition-equality-3}
\boxed{
\begin{gathered}
\text{For any $1 \leqslant i \leqslant p$, with the notation as in \cref{eq:define-index}, we have either}\\
a_i = b_i \qquad\qquad\text{or}\qquad\qquad a_i' = b_i'
\end{gathered}
}
\end{equation}
Certainly if \cref{eq:condition-equality-3} holds then from the definition, (c.f.\ \cref{eq:def-sigma}), we see that $b(\Lambda\oplus\Xi) = b(\Lambda\boxplus\Xi)$. Conversely taking $i = j$ in \cref{eq:condition-equality,eq:condition-equality-2} we see that the condition in \cref{eq:condition-equality-3} is satisfied. This proves the claim.
\end{pa}

\subsection{Lusztig's Lagrangian Subspace of $V\times W$}
\begin{pa}\label{pa:def-X}
Let $V \subseteq \overline{\bOmega}_a$ and $W \subseteq \overline{\bPhi}_b$ be similarity classes of symbols. Following Lusztig we define a subspace $\mathcal{J}$ of the symplectic vector space $V \times W$, (see \cite[4.15]{lusztig:1986:on-the-character-values} where there the subspace is denoted by $X$). Let $(\Lambda,\Xi) \in V \times W$ be any pair then, after possibly permuting the rows and applying shifts, we may assume that $\epsilon(\Lambda,\Xi) \leqslant 1$ and that each row of $\Lambda$ and $\Xi$ contains a zero. We now write the entries of $\Lambda$ as a weakly increasing sequence
\begin{equation}\label{eq:sequence-first-factor}
\begin{cases}
c_1 \leqslant c_2 \leqslant \cdots \leqslant c_{2k+1} &\text{type }\B_n\\
c_1 \leqslant c_2 \leqslant \cdots \leqslant c_{4k} &\text{type }\C_n\text{ and }d(\Xi) \equiv 1 \pmod{4}\\
c_0 \leqslant c_1 \leqslant \cdots \leqslant c_{4k-1} &\text{type }\C_n\text{ and }d(\Xi) \equiv -1 \pmod{4}\\
c_1 \leqslant c_2 \leqslant \cdots \leqslant c_{4k} &\text{type }\D_n.
\end{cases}
\end{equation}
Similarly we write the entries of $\Xi$ as a weakly increasing sequence
\begin{equation}\label{eq:sequence-second-factor}
\begin{cases}
d_1 \leqslant d_2 \leqslant \cdots \leqslant d_{2k+1} &\text{type }\B_n\\
d_0 \leqslant d_2 \leqslant \cdots \leqslant d_{4k} &\text{type }\C_n\text{ and }d(\Xi) \equiv 1 \pmod{4}\\
d_1 \leqslant d_2 \leqslant \cdots \leqslant d_{4k-1} &\text{type }\C_n\text{ and }d(\Xi) \equiv -1 \pmod{4}\\
d_1 \leqslant d_2 \leqslant \cdots \leqslant d_{4k} &\text{type }\D_n.
\end{cases}
\end{equation}
We denote by $\mathbb{L}$ the indexing set of the elements in \cref{eq:sequence-first-factor} with 0 removed. For instance if $\bG$ is of type $\C_n$ we have $\mathbb{L} = \{1,\dots,4k\}$ if $d(\Xi) \equiv 1 \pmod{4}$ and $\{1,\dots,4k-1\}$ if $d(\Xi) \equiv -1 \pmod{4}$. We will denote by $\approx$ the equivalence relation on $\mathbb{L}$ generated by the condition $i \approx i+1$ if $c_i = c_{i+1}$ or $d_i = d_{i+1}$ and $\overline{\mathbb{L}}$ the resulting equivalence classes. It is clear to see that any $I \in \overline{\mathbb{L}}$ is an interval. We will say $I \in \overline{\mathbb{L}}$ is \emph{degenerate} if $I = \{i,i+1\}$ and we have $c_i = c_{i+1}$ and $d_i = d_{i+1}$, otherwise we say $I$ is non-degenerate. Finally we denote by $I_0 \in \overline{\mathbb{L}}$ the unique equivalence class containing 1 and we set
\begin{equation*}
\widetilde{\mathbb{L}} = \begin{cases}
\{I \in \overline{\mathbb{L}} \mid |I|\text{ is odd}\} \cup \{I_0\} &\text{if }\bG\text{ is of type }\C_n\text{, }d(\Xi) \equiv -1\pmod{4}\text{ and }|I_0|\text{ is even},\\
\{I \in \overline{\mathbb{L}} \mid |I|\text{ is odd}\} &\text{otherwise.}
\end{cases}
\end{equation*}

Assume $f : \mathbb{L} \to \mathbb{F}_2$ is any function, with values in the finite field of order two, which is constant on equivalence classes. If $\bG$ is of type $\C_n$ or $\D_n$ then we will also assume that $f$ satisfies the condition
\begin{equation}\label{eq:def-4-condition}
|\{I \in \widetilde{\mathbb{L}}\mid f(i) = 1\text{ for some }i \in I\}| \equiv 0 \pmod{2}
\end{equation}
with this we define
\begin{align*}
S &= \{0\}\cup\{c_i \mid i \in \mathbb{L}\text{ is odd and }f(i) = 0\} \cup \{c_i \mid i \in \mathbb{L}\text{ is even and }f(i) = 1\}\\
T &= \{0\}\cup\{c_i \mid i \in \mathbb{L}\text{ is odd and }f(i) = 1\} \cup \{c_i \mid i \in \mathbb{L}\text{ is even and }f(i) = 0\}\\
S' &= \{0\}\cup\{d_i \mid i \in \mathbb{L}\text{ is odd and }f(i) = 0\} \cup \{d_i \mid i \in \mathbb{L}\text{ is even and }f(i) = 1\}\\
T' &= \{0\}\cup\{d_i \mid i \in \mathbb{L}\text{ is odd and }f(i) = 1\} \cup \{d_i \mid i \in \mathbb{L}\text{ is even and }f(i) = 0\}.
\end{align*}
Note that by assumption 0 is already contained in one of the latter two sets unless $\bG$ is of type $\C_n$ in which case precisely one of $S$, $T$, $S'$ or $T'$ does not contain 0. We claim the following hold:
\begin{itemize}
	\item If $\bG$ is of type $\B_n$ or $\C_n$ then precisely one of $(\ssymb{S}{T},\ssymb{S'}{T'})$ or $(\ssymb{T}{S},\ssymb{T'}{S'})$ defines an element of $V\times W$, the former if $|S'| > |T'|$ and the latter if $|S'|<|T'|$.
	\item If $\bG$ is of type $\D_n$ then $(\ssymb{S}{T},\ssymb{S'}{T'}) = (\ssymb{T}{S},\ssymb{T'}{S'})$ defines a single element of $V\times W$.
\end{itemize}
If $S=T$, (resp.\ $S' = T'$), then $V$, (resp.\ $W$), is either $\{\ssymb{S}{T}_+\}$ or $\{\ssymb{S}{T}_-\}$, (resp.\ $\{\ssymb{S'}{T'}_+\}$ or $\{\ssymb{S'}{T'}_-\}$). In particular there is a unique sign, which we assume chosen, so that the element lies in $V \times W$.

Let us prove the claim. Assume $f_0 : \mathbb{L} \to \mathbb{F}_2$ is the function defined by $f_0(i) = 0$ for all $i \in \mathbb{L}$ and let $S_0$, $T_0$, $S_0'$ and $T_0'$ be respectively the sets $S$, $T$, $S'$ and $T'$ defined with respect to $f_0$. It is clear that $f_0$ satisfies the condition in \cref{eq:def-4-condition} and $(\ssymb{S_0}{T_0},\ssymb{S_0'}{T_0'}) \in \overline{\bOmega}_a \times \overline{\bPhi}_b$ unless $\bG$ is of type $\C_n$ and $d(\Xi) \equiv 1 \pmod{4}$ in which case $(\ssymb{T_0}{S_0},\ssymb{T_0'}{S_0'}) \in \overline{\bOmega}_a \times \overline{\bPhi}_b$. Fix an equivalence class $I \in \overline{\mathbb{L}}$ and let $f_I$ be such that $f_I(i) = 1$ if $i \in I$ and 0 otherwise. If $I \in \overline{\mathbb{L}} - \widetilde{\mathbb{L}}$ then we have $|S| = |S_0|$ and $|T| = |T_0|$ but if $I \in \widetilde{\mathbb{L}}$ then we have $|S| = |S_0| \pm 1$ and $|T| = |T_0| \mp 1$ for some sign. If $\bG$ is of type $\B_n$ then we see that one of the pairs $(\ssymb{S}{T},\ssymb{S'}{T'})$ or $(\ssymb{T}{S},\ssymb{T'}{S'})$ contains symbols of odd positive defect, hence is an element of $\overline{\bOmega}_a \times \overline{\bPhi}_b$. However when $\bG$ is of type $\C_n$ or $\D_n$ and $I \in \widetilde{\mathbb{L}}$ then the defect of the symbols $\ssymb{S}{T}$ and $\ssymb{T}{S}$ will be congruent to $2\pmod{4}$, hence the resulting pairs will not be elements of $\overline{\bOmega}_a \times \overline{\bPhi}_b$. As every function $f$ is obtained from $f_0$ via steps of this kind we see that the condition in \cref{eq:def-4-condition} precisely ensures that one of the pairs is an element of $\overline{\bOmega}_a \times \overline{\bPhi}_b$. We denote by $\mathcal{J}$ the subset of $V \times W$ determined by all such $f$ as above.
\end{pa}

\begin{prop}\label{lem:lagrangian-subspace}
We have $\mathcal{J} = \{(\Lambda,\Xi) \in V \times W \mid \epsilon(\Lambda,\Xi)\leqslant 1$ and $a(\Lambda)+a(\Xi) = b(\Lambda\oplus\Xi)\}$ and 
\begin{equation*}
|\mathcal{J}| = \begin{cases}
2^{\ell-1} &\text{if }\bG\text{ is of type }\B_n\\
2^{\ell-2} &\text{if }\bG\text{ is of type }\C_n\text{ or }\D_n
\end{cases}
\end{equation*}
where $\ell$ is the number of non-degenerate equivalence classes in $\overline{\mathbb{L}}$. Furthermore no element of $\mathcal{J}\cap(\overline{\bOmega}_a^1\times\overline{\bPhi}_b^1)$ satisfies the condition in \cref{eq:condition-equality-3} unless $n=0$ or $\overline{\mathbb{L}} = \{\mathbb{L}\}$.
\end{prop}

\begin{proof}
Let $(\Lambda,\Xi) \in \mathcal{J}$ then as shown above $(\Lambda,\Xi) \in \overline{\bOmega}_a \times \overline{\bPhi}_b$ and by design we have $\epsilon(\Lambda,\Xi) \leqslant 1$. Define $I \subset \mathbb{L}$, (resp.\ $J \subset \mathbb{L}$), to be the set of all $i \in \mathbb{L}$, (resp.\ $j \in \mathbb{L}$), such that $c_i$ occurs in the top row of $\Lambda$ and $d_i$ occurs in the top row of $\Xi$, (resp.\ $c_j$ occurs in the bottom row of $\Lambda$ and $d_j$ occurs in the bottom row of $\Xi$). We then have $\mathbb{L}$ is a disjoint union $I \sqcup J$ and the condition in \cref{eq:condition-equality} can be recast as: For any $i \in I$ and $j  \in J$ we have either
\begin{equation*}
c_i \leqslant c_j\quad\text{and}\quad d_i \leqslant d_j \qquad\qquad\text{or}\qquad\qquad c_i \geqslant c_j\quad\text{and}\quad d_i \geqslant d_j.
\end{equation*}
However this clearly holds as both the statements $c_i \leqslant c_j$ and $d_i \leqslant d_j$, (resp.\ $c_i \geqslant c_j$ and $d_i \geqslant d_j$) are equivalent to $i \leqslant j$, (resp.\ $i \geqslant j$). In particular we have $a(\Lambda)+a(\Xi) = b(\Lambda\oplus\Xi)$.

Conversely assume $(\Lambda,\Xi) \in V \times W$ satisfies $\epsilon(\Lambda,\Xi) \leqslant 1$ and $a(\Lambda)+a(\Xi) = b(\Lambda\oplus\Xi)$. Let $(c_p)$ be the entries of $\Lambda$ arranged as a weakly increasing sequence and let $(d_q)$ be the entries of $\Xi$ arranged as a weakly increasing sequence. With the notation as in \cref{eq:define-index} we see that $a(\Lambda)+a(\Xi) = b(\Lambda\oplus\Xi)$ implies the condition in \cref{eq:condition-equality} holds. In particular we may arrange the sequences $(c_p)$ and $(d_q)$ such that for any $k$ we have either $c_k = a_i$ and $d_k = a_i'$ or $c_k = b_j$ and $d_k = b_j'$. This shows that over an equivalence class of $\approx$ the elements of the sequence $(c_p)$, (resp.\ $(d_q)$), alternate between the $a_i$ and $b_j$, (resp.\ $a_i'$ and $b_j'$). From this one easily deduces that there exists an $f$ admitting $(\Lambda,\Xi)$, hence $(\Lambda,\Xi) \in \mathcal{J}$ as required.

Assume now that $(\Lambda,\Xi) \in \mathcal{J}$ satisfies the condition in \cref{eq:condition-equality-3}. From the previous discussion concerning the sequences $(c_p)$ and $(d_q)$ we see that either every element of $\overline{\mathbb{L}}$ is degenerate or $\overline{\mathbb{L}} = \{\mathbb{L}\}$, however the former can only happen when $n=0$.

We now compute the order of $\mathcal{J}$. Let us first consider the case of type $\B_n$. Clearly the total number of possible functions $f$ is $2^s$ where $s = |\overline{\mathbb{L}}|$. However if $I\in\overline{\mathbb{L}}$ is degenerate then changing the value of $f$ on elements of $I$ does not affect the resulting element of $V \times W$, so the order of $\mathcal{J}$ is at most $2^{\ell}$. Let $\tilde{f}$ be the function defined by
\begin{equation*}
\tilde{f}(x) = \begin{cases}
0 &\text{if }f(x) = 1\\
1 &\text{if }f(x) = 0
\end{cases}
\end{equation*}
then both $\tilde{f}$ and $f$ determine the same element of $V \times W$, (note this argument also applies when $\bG$ is of type $\C_n$ or $\D_n$). This argument shows that the order of $\mathcal{J}$ is $2^{\ell-1}$ as required. We now consider the case when $\bG$ is of type $\C_n$ or $\D_n$. Let $\mu = |\overline{\mathbb{L}}| - |\widetilde{\mathbb{L}}|$ and $\nu = |\widetilde{\mathbb{L}}|$. The condition in \cref{eq:def-4-condition} together with the above discussion implies that the order is bounded above by $2^{\ell-1} = 2^{\mu}2^{\nu-1}$ because
\begin{equation*}
\sum_{x\text{ even}}\binom{\nu}{x} = \sum_{z = 0}^{\nu-1} \binom{\nu-1}{z} = 2^{\nu-1}.
\end{equation*}
Applying the parity swapping argument we obtain the order of $\mathcal{J}$ in these cases.
\end{proof}

\subsection{Computing the Values of Character Sheaves}
\begin{pa}\label{pa:the-subspace-J}
We end this section with the following conclusion. Let $V \subseteq \overline{\bOmega}_a$ and $W \subseteq \overline{\bPhi}_b$ be similarity classes of symbols and let $\mathcal{J} \subseteq V \times W$ be as in \cref{pa:def-X}. The following statements are easily checked using the combinatorial expression of $\mathcal{J}$ given above. Firstly if $(\Lambda_1,\Xi_1)$ and $(\Lambda_2,\Xi_2)$ are both in $\mathcal{J}$ then we have $\Lambda_1\oplus\Xi_1 \equiv \Lambda_2 \oplus \Xi_2$., i.e.\ they are in the same similarity class. Assume now that $(\Lambda,\Xi) \in \mathcal{J}$ then there exists a pair $\iota \in \mathcal{N}_{\bG}$ such that $\iota \mapsto \Lambda\oplus\Xi$ under the map $\mathcal{N}_{\bG} \to \Psi_N$, (note that this pair is unique unless the symbol is degenerate). We then define
\begin{equation*}
m(\Lambda\oplus\Xi) = \rank(\bL_{\iota}/Z^{\circ}(\bL_{\iota})) = \dim\bT_0/Z^{\circ}(\bL_{\iota}) \qquad\qquad \mathcal{O}(\Lambda\oplus\Xi) = \mathcal{O}_{\iota}.
\end{equation*}
We will denote by $\mathcal{O}$ the unipotent conjugacy class $\mathcal{O}_{\iota}$. Assume that $\mathcal{O}$ is non-degenerate then by \cref{lem:gen-spring-same-class} we see that $\mathcal{O}$ is independent of the choice of $(\Lambda,\Xi) \in \mathcal{J}$ used to define it. With this in hand we have the following result.
\end{pa}

\begin{prop}\label{prop:eval-char-function}
Assume $\mathcal{O}$ is non-degenerate and $u \in \mathcal{O}^F$ is a split element. Let $(\Lambda,\Xi) \in V \times W$ be such that $\epsilon(\Lambda,\Xi) \leqslant 1$ and $\widetilde{A}(\Lambda,\Xi)$ is $F_w$-stable then we have
\begin{equation*}
\chi_{\widetilde{A}(\Lambda,\Xi)}(u) = \begin{cases}
\epsilon(-1)^nq^{\dim\mathfrak{B}_u^{\bG} + \frac{1}{2}m(\Lambda\oplus\Xi)} &\text{if }(\Lambda,\Xi) \in \mathcal{J},\\
0 &\text{otherwise}.
\end{cases}
\end{equation*}
Here $\epsilon \in \{\pm 1\}$ is a sign, determined by \cref{cor:multiplicity-1}, which we see is 1 unless $\bG$ is of type $\D_n$, $F$ acts non-trivially on $W_{\bG}$ and $d(\Lambda) = d(\Xi) = 0$.
\end{prop}

\begin{proof}
Using \cref{eq:chi-A-G-uni} and the definitions of $X_{\iota}^w$ and $Y_{\iota}$ in \cref{pa:X-and-Y} we obtain
\begin{align*}
\chi_{\widetilde{A}(\Lambda,\Xi)}(u) &= \sum_{\iota \in \mathscr{I}(\bL,\nu)^F}m(A,\iota,\phi_{\iota}^w)\chi_{K_{\iota},\phi_{\iota}^w}(u),\\
&= \sum_{\iota \in \mathscr{I}(\bL,\nu)^F}m(A,\iota,\phi_{\iota}^w)(-1)^{a_{\iota}}q^{(\dim\bG + a_{\iota})/2}X_{\iota}^w(u)\\
&= \sum_{\iota \in \mathscr{I}(\bL,\nu)^F}\sum_{\iota' \in \mathcal{N}_{\bG}^F} m(A,\iota,\phi_{\iota}^w)(-1)^{a_{\iota}}q^{(\dim\bG + a_{\iota})/2}P_{\iota',\iota}^wY_{\iota'}(u).
\end{align*}
Assume that the term in the sum is non-zero then we have $Y_{\iota'}(u) \neq 0$ which implies $\mathcal{O}_{\iota'} = \mathcal{O}$ by definition. Using \cref{eq:b-val-dimBu,eq:dim-formula} we see that $m(A,\iota,\phi_{\iota}^w)\neq 0$ implies $\dim(\mathcal{O}_{\iota}) \leqslant \dim(\mathcal{O})$ and $P_{\iota',\iota}^w \neq 0$ implies that $\dim(\mathcal{O}) = \dim(\mathcal{O}_{\iota'}) \leqslant \dim(\mathcal{O}_{\iota})$, which gives us $\dim(\mathcal{O}_{\iota}) = \dim(\mathcal{O})$. In particular this shows that $\chi_{\widetilde{A}(\Lambda,\Xi)}(u)\neq0$ implies $(\Lambda,\Xi) \in \mathcal{J}$. By \cref{cor:multiplicity-1,lem:lagrangian-subspace} there is a unique pair $\iota \in \mathcal{N}_{\bG}^F$ satisfying $\dim(\mathcal{O}_{\iota}) = \dim(\mathcal{O})$ and $m(A,\iota,\phi_{\iota}^w)\neq 0$.

For any pair $\kappa \in \mathcal{N}_{\bG}$ we claim that
\begin{equation}\label{eq:aiota-n}
a_{\kappa} \equiv n \pmod{2}.
\end{equation}
As the dimension of a unipotent class is even, (c.f.\ \cref{eq:dim-formula}), it suffices to check that when $\bL$ is a Levi subgroup supporting a cuspidal pair we have $\dim Z^{\circ}(\bL) \equiv n \pmod{2}$ but this is clear using the information given in \cref{tab:levis-1}. As $A_{\bG}(u)$ is abelian we have $\mathscr{E}_{\iota'}$ is of rank 1, hence by definition we have $Y_{\iota'}(u) = 1$ because $u$ is split. Applying \cref{cor:multiplicity-1,eq:P-properties} we obtain
\begin{equation*}
\chi_{\widetilde{A}(\Lambda,\Xi)}(u) = \epsilon(-1)^n\gamma_{\bL_{\iota}}^{\bG}(F(w))q^{\dim\mathfrak{B}_u^{\bG}+\frac{1}{2}m(\Lambda\oplus\Xi)}.
\end{equation*}
We claim that we always have $\gamma_{\bL_{\iota}}^{\bG}(F(w))=1$. If $\bL$ is a torus then this follows from \cite[Corollary 6.9]{bonnafe:2004:actions-of-rel-Weyl-grps-I} and if $\bL$ is not a torus then this follows from the fact that $w$ is the identity, (see \cref{sec:weyl-groups}). This now gives the required statement.
\end{proof}
%
\section{Recalling Lusztig's Conjecture}
\begin{pa}\label{pa:lusztigs-conj-generic-case}
In \cite[\S13.6 - \S13.7]{lusztig:1984:characters-of-reductive-groups} Lusztig conjectured that the characteristic functions of character sheaves coincide with the almost characters of $\bG^F$ up to multiplication by roots of unity. The main result of this paper is an explicit determination of this relationship for those character sheaves of $\bG$ whose support contains unipotent elements. To state this result precisely we must first introduce some notation and recall some facts from \cite[Chapter 4]{lusztig:1984:characters-of-reductive-groups}. Let $\gamma$ be the automorphism of $W_{\bG^{\star}}(s) = \bV_a\times\bW_b$ induced by $F_w$ then we define
\begin{equation*}
\overline{X}(s,\gamma) = \begin{cases}
\Phi_a^+ \times \Phi_b^+ &\text{if }\bG\text{ of type }\B_n\\
\Omega_a^{\pm} \times \Phi_b^+ &\text{if }\bG\text{ of type }\C_n\\
\Omega_a^{\pm} \times \Omega_b^{\pm} &\text{if }\bG\text{ of type }\C_n
\end{cases}
\end{equation*}
where the sign in $\Omega_a^{\pm}$ is $+$ if $\gamma|_{\bV_a} = \ID$ and $-$ if $\gamma|_{\bV_a} = \gamma_a$, (similarly for $\Omega_b^{\pm}$). We also define $X(s,\gamma)$ to be the set $\bOmega_a \times \bPhi_b$ where we have
\begin{equation*}
\bOmega_a = \begin{cases}
\overline{\bOmega}_a &\text{if }\gamma|_{\bV_a} = \ID\\
\Phi_a^- &\text{if }\gamma|_{\bV_a} = \gamma_a
\end{cases}
\qquad
\bPhi_b = \begin{cases}
\overline{\bPhi}_b &\text{if }\gamma|_{\bW_b} = \ID\\
\Phi_b^- &\text{if }\gamma|_{\bW_b} = \gamma_b
\end{cases}
\end{equation*}
We now define $M$ to be the direct product $X \times Y$ where $X$, (resp.\ $Y$), is the trivial group if $\gamma|_{\bV_a} = \ID$, (resp.\ $\gamma|_{\bW_b} = \ID$), and the cyclic subgroup $\langle-1\rangle\leqslant \Ql^{\times}$ if $\gamma|_{\bV_a} = \gamma_a$, (resp.\ $\gamma|_{\bV_b} = \gamma_b$). We let $X$ act on $\bOmega_a$ by permuting the rows of the symbols if $X = \langle-1\rangle$ and trivially otherwise. We similarly define an action of $Y$ on $\bPhi_b$ so that we get an action of $M$ on $X(s,\gamma)$.

Let $\overline{X}(s) := \overline{\bOmega}_a \times \overline{\bPhi}_b$ then we have a natural map from $X(s,\gamma) \to \overline{X}(s)$ whose image is contained in the fixed point set $\overline{X}(s)^{\gamma}$, (see \cref{pa:action-frob-char-sheaves}). We also define subsets
\begin{equation*}
\bOmega_a^1 = \begin{cases}
\overline{\bOmega}_a^1 &\text{if }\gamma|_{\bV_a} = \ID\\
\widehat{\mathbb{V}}_a^0 &\text{if }\gamma|_{\bV_a} = \gamma_a
\end{cases}
\qquad
\bPhi_b^1 = \begin{cases}
\overline{\bPhi}_b^1 &\text{if }\gamma|_{\bW_b} = \ID\\
\widehat{\mathbb{V}}_b^0 &\text{if }\gamma|_{\bW_b} = \gamma_b.
\end{cases}
\end{equation*}
Let $\widetilde{W}_{\bG^{\star}}(s)$ be the semidirect product $W_{\bG^{\star}}(s) \rtimes\langle\gamma\rangle$ then we denote by $\Irr_{\ex}(\widetilde{W}_{\bG^{\star}}(s))$ those irreducible characters of $\widetilde{W}_{\bG^{\star}}(s)$ whose restriction to $W_{\bG^{\star}}(s)$ is irreducible. Using the parameterisation in \cref{pa:weyl-group-chars-symbols} we see that the natural inclusion map $\bOmega_a^1 \times \bPhi_b^1 \hookrightarrow X(s,\gamma)$ defines an $M$-equivariant embedding of $\Irr_{\ex}(\widetilde{W}_{\bG^{\star}}(s))$ into $X(s,\gamma)$.
\end{pa}

\begin{pa}\label{pa:lusztigs-conj-def-almost}
By \cite[Theorem 4.23]{lusztig:1984:characters-of-reductive-groups} we have a bijection $\overline{X}(s,\gamma) \to \mathcal{E}(G,s)$ denoted by $(\Lambda_1,\Xi_1) \mapsto \rho(\Lambda_1,\Xi_1)$. Let $\{-,-\} : \overline{X}(s,\gamma) \times X(s,\gamma) \to \Ql$ be the pairing defined by Lusztig in \cite[\S4.15 and \S4.18]{lusztig:1984:characters-of-reductive-groups} then for any $(\Lambda,\Xi) \in X(s,\gamma)$ we define the corresponding \emph{almost character} by setting
\begin{equation*}
R_s^{\bG}(\Lambda,\Xi) = (-1)^{\ell(w)}\sum_{(\Lambda_1,\Xi_1) \in \overline{X}(s,\gamma)} \{(\Lambda_1,\Xi_1),(\Lambda,\Xi)\}\rho(\Lambda_1,\Xi_1).
\end{equation*}
The element $R_s^{\bG}(\Lambda,\Xi)$ depends only on the $M$-orbit of $(\Lambda,\Xi)$ up to multiplication by $\pm 1$. With this notation in hand we can now recall the following result of Shoji.
\end{pa}

\begin{thm}[{}{Shoji, \cite[II - Theorem 3.2]{shoji:1995:character-sheaves-and-almost-characters}}]\label{thm:shoji}
For each $(\Lambda,\Xi) \in X(s,\gamma)$ there exists a scalar $\zeta_{(\Lambda,\Xi)} \in \Ql^{\times}$ of absolute value 1 such that
\begin{equation*}
R_s^{\bG}(\Lambda,\Xi) = \zeta_{(\Lambda,\Xi)}\chi_{A(\Lambda,\Xi)},
\end{equation*}
where $A(\Lambda,\Xi)$ is parameterised by the image of $(\Lambda,\Xi)$ under the map $X(s,\gamma) \to \overline{X}(s)$.
\end{thm}

The argument we will employ to compute these scalars will be based on induction. For this to work we will need a result which is completely analogous to \cite[6.4]{shoji:1997:unipotent-characters-of-finite-classical-groups}.

\begin{lem}[Shoji]\label{lem:shoji}
Assume $A_0 \in \bL_s^F$ is a unipotently supported $F$-stable cuspidal character sheaf. Let $R_0^{\bL}(s)$ be an almost character corresponding to $\chi_{A_0}$ and $\zeta_0 \in \Ql^{\times}$ the scalar of absolute value 1 such that $R_0^{\bL}(s) = \zeta_0\chi_{A_0}$. Assume $(\Lambda,\Xi) \in X(s,\gamma)$ is such that $A(\Lambda,\Xi)$ occurs in the induced complex $\ind_{\bL}^{\bG}(A_0)$ then after possibly replacing $(\Lambda,\Xi)$ by an element in the same $M$-orbit we have
\begin{equation*}
R_s^{\bG}(\Lambda,\Xi) = (-1)^n\zeta_0\chi_{A(\Lambda,\Xi)}.
\end{equation*}
\end{lem}

\begin{rem}
Note that here we have used \cref{eq:aiota-n}. The assumptions of \cref{thm:main-theorem} imply that $\gamma$ is the identity, (see \cref{rem:F-action-on-s}). In this case there is no freedom in the choice of $R_0^{\bL}(s)$ and each $M$-orbit of $X(s,\gamma)$ contains a unique element. Hence under the assumptions of \cref{thm:main-theorem} we have \cref{lem:shoji}, (together with the argument in \cite[5.20]{shoji:1995:character-sheaves-and-almost-characters}), reduces us to the case where $R_0^{\bG}(s)$ is cuspidal and $s$ is isolated. We hope to remove the ambiguity in \cref{lem:shoji} in the following article.
\end{rem}
%
\section{Lusztig's Conjecture for Unipotently Supported Character Sheaves}
We now wish to prove the following result which is a direct analog of \cite[Proposition 5.3]{lusztig:1986:on-the-character-values}.

\begin{thm}\label{prop:labels-coincide}
Assume $(\Lambda,\Xi) \in \overline{\bOmega}_a \times \overline{\bPhi}_b$ and let $\epsilon(\Lambda,\Xi) = |d(\Lambda) - d(\Xi)|$ be the modulus of the difference of the defects then the following hold:
\begin{enumerate}[label=(\alph*)]
	\item if $\epsilon(\Lambda,\Xi) \leqslant 1$ then $\widetilde{A}(\Lambda,\Xi) = A(\Lambda,\Xi)$.
	\item if $\epsilon(\Lambda,\Xi) > 1$ then $\supp(A(\Lambda,\Xi)) \cap \bG_{\uni} = \emptyset$.
	\item if $A(\Lambda,\Xi)$ is a cuspidal character sheaf then we have $\zeta_{(\Lambda,\Xi)} = 1$ assuming $F$ is split and $q \equiv 1 \pmod{4}$ when $\bG$ is of type $\C_n$ or $\D_n$.
\end{enumerate}
\end{thm}

\begin{pa}
Before proving \cref{prop:labels-coincide} we first give a series of lemmas concerning character sheaves and irreducible characters. Assume the Dynkin diagram of $\bG$ is labelled as in \cite[Plates I-IV]{bourbaki:2002:lie-groups-chap-4-6} and let $\bM \in \mathcal{L}_{\std}$ be such that the Dynkin diagram of $\bM$ is obtained from that of $\bG$ by deleting the node labelled 1. In other words $\bM$ is of type $\B_{n-1}$, $\C_{n-1}$ or $\D_{n-1}$ depending upon whether $\bG$ is respectively of type $\B_n$, $\C_n$ or $\D_n$. Recalling the notation from \cref{pa:assumptions} we denote by $\bM^{\star} \in \mathcal{L}_{\std}^{\star}$ a dual Levi subgroup and by $\bM_{\ad}^{\star}$ the image of $\bM$ under $\pi : \bG^{\star} \to \bG_{\ad}^{\star}$. Now $\bG_{\ad}^{\star} = \bM_{\ad}^{\star} \times Z^{\circ}(\bM_{\ad}^{\star})$ and we denote by $s_{\ad}$ the restriction of $s$ to $Z^{\circ}(\bM_{\ad}^{\star})$, it is clear that $s_{\ad} = \pm1$. Under the identification of $W_{\bG^{\star}}(s)$ with $W_a \times W_b$ we have
\begin{equation*}
W_{\bM^{\star}}(s) = \begin{cases}
\bV_{a-1} \times \bW_b &\text{if }s_{\ad} = 1,\\
\bV_a \times \bW_{b-1} &\text{if }s_{\ad} \neq 1.
\end{cases}
\end{equation*}
This together with \cref{thm:char-sheaves-main-theorem} shows that the character sheaves in $\widehat{\bM}_s$ are parameterised by $\overline{\bOmega}_{a-1}\times \overline{\bPhi}_b$, (resp.\ $\overline{\bOmega}_a\times\overline{\bPhi}_{b-1}$), if $s_{\ad}=1$, (resp.\ $s_{\ad}\neq1$). Let us denote by $\pi_s : \mathscr{K}_0(\bG) \otimes \mathbb{Q} \to \mathscr{K}_0(\bM) \otimes \mathbb{Q}$ the $\mathbb{Q}$-linear map defined by
\begin{equation*}
\pi_s(A) = \begin{cases}
A &\text{if }A \in \widehat{\bM}_s,\\
0 &\text{if }A \in \widehat{\bM}\text{ is not in }\widehat{\bM}_s,
\end{cases}
\end{equation*}
then with this we can give the following restriction statements for character sheaves.
\end{pa}

\begin{lem}\label{lem:rest-char-sheaf-LS}
For any $A(\Lambda,\Xi) \in \widehat{\bG}_s$ we have
\begin{equation}\label{eq:rest-char-sheaf-LS}
\pi_s\res_{\bM}^{\bG}A(\Lambda,\Xi) = \begin{cases}
\sum A(\Lambda_*,\Xi) &\text{if }s_{\ad} = 1,\\
\sum A(\Lambda,\Xi_*) &\text{if }s_{\ad} \neq 1.
\end{cases}
\end{equation}
where the first, (resp.\ second), sum is taken over all symbols $\Lambda_* \in \bOmega_{a-1}$, (resp.\ $\Xi_* \in \bPhi_{b-1}$), which are obtained by subtracting 1 from an entry in $\Lambda$, (resp.\ $\Xi$). If $\Lambda_*$ is degenerate then the sum will include both $(\Lambda_{*})_+$ and $(\Lambda_{*})_-$, similarly for $\Xi_*$.
\end{lem}

\begin{proof}
Using \cref{thm:char-sheaves-main-theorem} this is proved in exactly the same way as \cite[Lemma 5.5]{lusztig:1986:on-the-character-values}, (this was also noted by Shoji in \cite[Lemma 5.10]{shoji:1995:character-sheaves-and-almost-characters}). Note that the branching rules for type $\D_n$ are easily deduced from the corresponding branching rules for type $\B_n$, (see \cite[6.1.9]{geck-pfeiffer:2000:characters-of-finite-coxeter-groups}).
\end{proof}

\begin{lem}\label{lem:rest-char-sheaf-HC}
For any $\widetilde{A}(\Lambda,\Xi) \in \widehat{\bG}_s$ we have
\begin{equation}\label{eq:rest-char-sheaf-HC}
\pi_s\res_{\bM}^{\bG}\widetilde{A}(\Lambda,\Xi) = \begin{cases}
\sum \widetilde{A}(\Lambda_*,\Xi) &\text{if }s_{\ad} = 1,\\
\sum \widetilde{A}(\Lambda,\Xi_*) &\text{if }s_{\ad}\neq1,
\end{cases}
\end{equation}
where the sums are as in \cref{lem:rest-char-sheaf-LS}.
\end{lem}

\begin{proof}
This is proved in exactly the same way as \cite[Lemma 5.6]{lusztig:1986:on-the-character-values}.
\end{proof}

\begin{lem}\label{lem:princpal-series-Z}
For any $(\Lambda_1,\Xi_1) \in \overline{X}(s,\gamma)$ we have $\rho(\Lambda_1,\Xi_1)(v) \in \mathbb{Z}$ for any unipotent element $v \in G$.
\end{lem}

\begin{proof}
For any $m \in \mathbb{N}$ we denote by $\mathbb{Q}_m$ the subfield $\mathbb{Q}(\xi_m) \subset \Ql$ where $\xi_m \in \Ql$ is a primitive $m$th root of unity. Let $H$ be the automorphism group $\Aut(\Ql/\mathbb{Q}_{|G|_{p'}}) \leqslant \Aut(\Ql/\mathbb{Q})$ of the field extension then to obtain the statement of the lemma it is sufficient to show that $\alpha \circ \rho = \rho$ for all $\alpha \in H$, where $\rho := \rho(\Lambda_1,\Xi_1)$.

Recall the notation of \cref{pa:lusztigs-conj-generic-case} then for any $E \in \Irr(W_{\bG^{\star}}(s))$ we choose an extension $\widetilde{E} \in \Irr_{\ex}(\widetilde{W}_{\bG^{\star}}(s))$ such that $\widetilde{E}$ has rational values, (see \cite[Proposition 3.2]{lusztig:1984:characters-of-reductive-groups}). Let $R_s^{\bG}({\tilde{E}})$ be the corresponding almost character then by the definition in \cite[(3.7.1)]{lusztig:1984:characters-of-reductive-groups} we see that $R_s^{\bG}(\widetilde{E})$ is a $\mathbb{Q}$-linear combination of Deligne--Lusztig virtual characters. As is mentioned in the beginning of \cite[\S7.6]{carter:1993:finite-groups-of-lie-type} the value of any Deligne--Lusztig virtual character at a unipotent element is integral, hence we have $R_s^{\bG}(\widetilde{E})(v) \in \mathbb{Q}$ for any unipotent element $v \in G_{\uni}$ which proves the claim. With this we have for any $\alpha \in H$ that $\langle \alpha\circ\rho, R_s^{\bG}(\widetilde{E}) \rangle = \langle \rho, R_s^{\bG}(\widetilde{E}) \rangle$ for all $E \in \Irr(W_{\bG^{\star}}(s))$. However by \cite[Proposition 6.3]{digne-michel:1990:lusztigs-parametrization} $\rho$ is uniquely determined by its multiplicities in such almost characters, hence $\alpha\circ\rho = \rho$ for all $\alpha \in H$ as required.
\end{proof}

\begin{assumption}
From now until the end of this article we assume that when $\bG$ is of type $\C_n$ or $\D_n$ that $F$ is split and $q \equiv 1 \pmod{4}$. In particular $F_w$ is the identity, (see \cref{rem:F-action-on-s}).
\end{assumption}

\begin{proof}[of \cref{prop:labels-coincide}]
Let $\bG$ be of type $\B_n$, $\C_n$ or $\D_n$. If $n = 0$ then the statement of \cref{prop:labels-coincide} is trivial. We will now assume that $n \geqslant 1$ and the result holds for all such groups of type $\B_k$, $\C_k$ or $\D_k$ with $k < n$. Throughout $\bM$ will denote a Levi subgroup of $\bG$ of type $\B_{n-1}$, $\C_{n-1}$ or $\D_{n-1}$ appropriately. We will now show that (b) holds for all non-cuspidal character sheaves.
\begin{itemize}
	\item Let us assume that $\Xi$ is not cuspidal then $b \geqslant 1$ and we assume $\bM$ to be chosen such that $s_{\ad}\neq 1$. Let $A(\Lambda,\Xi) \in \widehat{\bG}_s$ then by \cref{lem:rest-char-sheaf-LS} we have $A(\Lambda,\Xi)$ occurs in $\ind_{\bM}^{\bG} A(\Lambda,\Xi')$ where $\Xi' \in \bPhi_{b-1}$ is a symbol occurring in the sum $\Xi_*$ in \cref{eq:rest-char-sheaf-LS}. By the induction hypothesis $\supp (A(\Lambda,\Xi')) \cap \bM_{\uni} = \emptyset$ occurs only when $\epsilon(\Lambda,\Xi') >1$ but this is true if and only if $\epsilon(\Lambda,\Xi)>1$. This proves (b) unless $\Xi$ is cuspidal.
\end{itemize}
The argument above also works when $\Lambda$ is not cuspidal, hence we have the validity of (b) unless $A(\Lambda,\Xi)$ is cuspidal. We will return to this case after first considering (a).

Assume $A(\Lambda,\Xi)$ is such that $\epsilon(\Lambda,\Xi)\leqslant 1$ and neither $\Lambda$ nor $\Xi$ is cuspidal. The same argument as used above shows that $\supp (A(\Lambda,\Xi)) \cap \bG_{\uni} \neq \emptyset$, which means $A(\Lambda,\Xi) = \widetilde{A}(\Lambda',\Xi')$ for some $(\Lambda',\Xi') \in \bOmega_a \times \bPhi_b$ with $\epsilon(\Lambda',\Xi') = \epsilon(\Lambda,\Xi)$, (see \cref{pa:label-embed}). Applying $\pi_s\res_{\bM}^{\bG}$ to this equality we obtain from \cref{lem:rest-char-sheaf-LS,lem:rest-char-sheaf-HC} that
\begin{equation}\label{eq:res-equality}
\begin{cases}
\sum A(\Lambda,\Xi_*) = \sum A(\Lambda',\Xi_*') &\text{if }s_{\ad} = 1\\
\sum A(\Lambda_*,\Xi) = \sum A(\Lambda_*',\Xi') &\text{if }s_{\ad} \neq 1
\end{cases}
\end{equation}
where we have used the induction hypothesis in $\bM$. In both cases the first sum is non-empty hence we deduce that $\Xi = \Xi'$ and $\Lambda = \Lambda'$. This proves (a) unless one of $\Lambda$ or $\Xi$ is cuspidal.

Assume now that $\Lambda$ is cuspidal and $\Xi$ is not cuspidal. We leave it to the reader to see that the following arguments apply equally well when $\Xi$ is cuspidal and $\Lambda$ is not cuspidal. With this assumption we have $a'=0$ and $b' > 0$ and the equalities in \cref{eq:res-equality} tell us only that $\Lambda = \Lambda'$ and $\{\Xi_*\} = \{\Xi_*'\}$, (where these are the sets of entries occurring in the sum in \cref{eq:res-equality}). Let $\Pi$, $\Pi' \in \mathbb{V}_k^1$, (resp.\ $\overline{\mathbb{V}}_k^0$), with $k \geqslant 1$ be two symbols and let $\{\Pi_*\}$, $\{\Pi'_*\}\subseteq \mathbb{V}_{k-1}^1$, (resp.\ $\overline{\mathbb{V}}_{k-1}^0$), be the set of all possible symbols obtained by subtracting 1 from an entry. It is easily seen that if $k \geqslant 3$ then $\{\Pi_*\}=\{\Pi'_*\}$ implies that $\Pi = \Pi'$. Using the notation of \cref{pa:label-embed} we denote by $\varphi_{b'} : \mathbb{B}_{b'} \to \bPhi_b$ the restriction of $\varphi$ to the second factor. Using the definition of the map in \cref{eq:bij-shift-def} we can see that for any $\Pi \in \mathbb{B}_{b'}$ we have $\{\varphi_{b'}(\Pi_*)\} = \{\varphi_{b'}(\Pi)_*\}$. This observation together with the equality in \cref{eq:res-equality} shows that (a) holds whenever $b'\geqslant 3$. We must now deduce that $\Xi = \Xi'$ when $b' \leqslant 2$.

Assume $t \geqslant 1$ so that we have $\mathbb{B}_{b'} = \mathbb{V}_{b'}^1$. As in \cite[5.7]{lusztig:1986:on-the-character-values} we define
\begin{equation*}
b' = 2 \left\{
\begin{aligned}
\Xi^1 &= \varphi_{b'}\symb{2}{\emptyset}\\
\Xi^2 &= \varphi_{b'}\symb{1,2}{0}\\
\Xi^3 &= \varphi_{b'}\symb{0,1}{2}
\end{aligned}
\right.
\qquad
b' = 2 \left\{
\begin{aligned}
\Xi^4 &= \varphi_{b'}\symb{0,1,2}{1,2}\\
\Xi^5 &= \varphi_{b'}\symb{0,2}{1}
\end{aligned}
\right.
\qquad
b' = 1 \left\{
\begin{aligned}
\Xi^6 &= \varphi_{b'}\symb{1}{\emptyset}\\
\Xi^7 &= \varphi_{b'}\symb{0,1}{1}
\end{aligned}
\right.
\end{equation*}
Assume now that $t=0$ then from \cref{pa:label-embed,pa:a'-b'} we see that $n = a'+b' \leqslant 2$. If $\bG$ is of type $\D_n$ then under this this assumption it is not necessarily the case that $\mathbb{B}_{b'} = \mathbb{V}_{b'}^1$. If $\bG$ is of type $\D_2$ then there are two Levi subgroups of type $\A_1$ and applying the above restriction arguments to both these Levi subgroups distinguishes the character sheaves and if $\bG$ is of type $\D_1$ it is isomorphic to a group of type $\B_1$. In particular we may simply assume in this situation that $\bG$ is of type $\B$ so we can take the symbols $\Xi^i$, ($1 \leqslant i \leqslant 7$), to be defined as above. Note also that in the situation $b'=0$ and $a'> 0$ we must have $a'=2$ if $\bG$ is of type $\C_n$ or $\D_n$, (c.f.\ \cref{pa:assumptions}). For each $i \in \{0,\dots,7\}$ we now define
\begin{equation*}
A_i = A(\Lambda^0,\Xi^i)
\qquad\qquad
\widetilde{A}_i = \widetilde{A}(\Lambda^0,\Xi^i)
\end{equation*}
where $\Lambda^0$ and $\Xi^0$ denote the appropriate cuspidal symbol of defect $2t+1$, $4t\pm 1$ or $4t$, depending on the type of $\bG$.

The equality in \cref{eq:res-equality} gives us only the following information
\begin{equation*}
\underbrace{
\begin{aligned}
\{\widetilde{A}_1,\widetilde{A}_2\}&=\{A_1,A_2\}\\
\{\widetilde{A}_3,\widetilde{A}_4\}&=\{A_3,A_4\}\\
\{\widetilde{A}_5\}&=\{A_5\}
\end{aligned}}_{b'=2}
\qquad\qquad
\underbrace{
\begin{aligned}
\{\widetilde{A}_6,\widetilde{A}_7\}=\{A_6,A_7\}
\end{aligned}}_{b'=1}
\end{equation*}
We must now distinguish these pairs of character sheaves, which we will do by comparing their characteristic functions at certain unipotent elements. After some mildly laborious calculations with symbols one checks that we have the following inequalities:
\begin{equation*}
b(\Lambda^0\oplus\Xi^1) < b(\Lambda^0\oplus\Xi^2)\qquad\qquad
b(\Lambda^0\oplus\Xi^3) < b(\Lambda^0\oplus\Xi^4)\qquad\qquad
b(\Lambda^0\oplus\Xi^6) < b(\Lambda^0\oplus\Xi^7)
\end{equation*}
For any $i \in \{0,\dots,7\}$ we will denote by $\mathcal{O}_i$ the class $\mathcal{O}(\Lambda^0\oplus\Xi^i)$, (c.f.\ \cref{pa:the-subspace-J}), and by $u_i$ a split element of $\mathcal{O}_i^F$. The above inequalities together with \cref{cor:multiplicity-1} imply that
\begin{align*}
\chi_{\widetilde{A}_i}(u_i) \neq 0 \qquad\text{and}\qquad \chi_{\widetilde{A}_{i+1}}(u_i)=0
\end{align*}
for each $i \in \{1,3,6\}$. In fact the non-zero value is $\pm$ a power of $q$, (see \cref{prop:eval-char-function}), hence an odd integer. We will now prove the following:
\begin{enumerate}[label=(\Alph*)]
	\item $\chi_{A_i}(u_i) \neq 0$ if $i \in \{0,1,6\}$.
	\item $\chi_{A_4}(u_3)$ is an even integer.
\end{enumerate}
It is clear from the above remarks that this is enough to distinguish the character sheaves.

For any $i \in \{0,\dots,7\}$ we denote by $V\times W_i \subseteq \bOmega_a\times\bPhi_b$ the product of similarity classes containing $(\Lambda^0,\Xi^i)$. Let $\rho_i := \rho(\Lambda_{\star}^0,\Xi_{\star}^i) \in \mathcal{E}(G,s)$ be such that $(\Lambda_{\star}^0,\Xi_{\star}^i) \in V\times W_i$ parameterises the unique special character of $\bV_a\times\bW_b$. Assume now that $t \geqslant 1$ then by \cite[Main Theorem 4.23]{lusztig:1984:characters-of-reductive-groups} and \cite[4.14.2]{lusztig:1984:characters-of-reductive-groups} we have for each $1 \leqslant i \leqslant 7$ that
\begin{equation*}
2^{\ell}\rho_i = \sum_{(M,N) \in V \times W_i} R_s^{\bG}(M,N) = \sum_{(M,N) \in V\times W_i} \zeta_{(M,N)}\chi_{A(M,N)}
\end{equation*}
where the sum is over $V\times W_i$ and
\begin{equation*}
\ell = \begin{cases}
2t &\text{if }\bG\text{ is of type }\B_n\\
4t-1 &\text{if }\bG\text{ is of type }\C_n\text{ and }d=4t+1\\
4t-2 &\text{if }\bG\text{ is of type }\C_n\text{ and }d=4t-1\\
4t-2 &\text{if }\bG\text{ is of type }\D_n,
\end{cases}
\end{equation*}
see \cite[\S4.15]{lusztig:1984:characters-of-reductive-groups}. For each $i \in \{0,1,4,6\}$ we see that the statement of the theorem holds for all $(M,N) \in V\times W_i$ except for a single pair, namely $(\Lambda^0,\Xi^i)$. Let us denote by $\zeta_i$ the scalar $\zeta_{(\Lambda^0,\Xi^i)}$ then using the inductive hypothesis and \cref{lem:shoji} we have for each $i \in \{0,1,4,6\}$ and unipotent element $v\in G$ that
\begin{equation}\label{eq:special-character-sum}
2^{\ell}\rho_i(v) = \zeta_i\chi_{A_i}(v) + \sum_{\substack{(M,N) \in V \times W_i\\ \epsilon(M,N) \leqslant 1\text{ and }N\neq \Xi^i}} (-1)^n\chi_{\widetilde{A}(M,N)}(v),
\end{equation}
Note also that by \cref{lem:shoji} and the inductive hypothesis we know $\zeta_i=1$ unless $i=0$.

Assume that in \cref{eq:special-character-sum} we have $v = u_i$, (for $i \in \{0,1,6\}$), then using \cref{prop:eval-char-function} we see that the sum in \cref{eq:special-character-sum} can be written as
\begin{equation*}
\sum_{(M,N) \in \mathcal{J}\setminus\{(\Lambda^0,\Xi^i)\}}q^{\dim\mathfrak{B}_u^{\bG} + \frac{1}{2}m(\Lambda\oplus\Xi)}
\end{equation*}
As $\chi_{A_i}(u_i) \in \{\chi_{\widetilde{A}_i}(u_i),\chi_{\widetilde{A}_{i+1}}(u_i)\}$ we have, by \cref{lem:princpal-series-Z,prop:eval-char-function}, that all the terms in \cref{eq:special-character-sum} are integral. Using \cref{lem:lagrangian-subspace} we see that $|\mathcal{J}| = 2^{\ell}$, (with $\ell$ as above), hence we obtain that
\begin{equation}\label{eq:cong-mod-2}
\zeta_i\chi_{A_i}(u_i) \equiv 1 \pmod{2}
\end{equation}
in $\mathbb{Z}$ so $\chi_{A_i}(u_i) \neq 0$. This proves (A) when $t \geqslant 1$. Taking $i = 4$ and $v = u_3$ in \cref{eq:special-character-sum} we see that the sum is zero and we are left with $2^{\ell}\rho_4(u_3) = \chi_{A_4}(u_3)$, which proves (B) when $t \geqslant 1$.

We now consider (A) and (B) when $t=0$. We do not need to prove (A) when $i=0$ because this would imply $n=0$, which we have assumed is not the case. Again by \cite[Main Theorem 4.23]{lusztig:1984:characters-of-reductive-groups} and the induction hypothesis we have for $i \in \{1,6\}$ that $\rho_i(u_i) = \chi_{A_i}(u_i)$. However when $i \in \{1,6\}$ we have $\rho_i$ is the trivial character, hence clearly this proves (A). We have $\rho_4$ is the Steinberg character, which proves (B) as $\mathcal{O}_3$ is not the trivial class. This proves (a) unless $A(\Lambda,\Xi)$ is cuspidal. By (A) we have $\chi_{A(\Lambda^0,\Xi^0)}(u_0) \neq 0$ therefore we clearly have $\supp(A(\Lambda^0,\Xi^0)) \cap \bG_{\uni} \neq \emptyset$ so $A(\Lambda^0,\Xi^0) = \widetilde{A}(\Lambda',\Xi')$ with $(\Lambda',\Xi') \in \bOmega_a\times\bPhi_b$ satisfying $\epsilon(\Lambda',\Xi') \leqslant 1$. If $(\Lambda',\Xi') = (\Lambda^0,\Xi^0)$ then we are done. Assume this is not the case then we know already that $\widetilde{A}(\Lambda',\Xi') = A(\Lambda',\Xi')$ but this implies $(\Lambda',\Xi') = (\Lambda^0,\Xi^0)$, a contradiction. This completes the proof of (a).

Having proven (a) we now return to proving (b). Assume $A(\Lambda,\Xi)$ is a cuspidal character sheaf such that $\epsilon(\Lambda,\Xi) > 1$ and assume for a contradiction that $A(\Lambda,\Xi) \cap \bG_{\uni} \neq \emptyset$. There is a unique cuspidal character sheaf of $\bG$ with unipotent support, namely $\widetilde{A}(\Lambda^0,\Xi^0)$ where $\Lambda^0$ and $\Xi^0$ are appropriate cuspidal symbols satisfying $\epsilon(\Lambda^0,\Xi^0) \leqslant 1$. In particular we must have $A(\Lambda,\Xi) = \widetilde{A}(\Lambda^0,\Xi^0)$. By (a) we have $\widetilde{A}(\Lambda^0,\Xi^0) = A(\Lambda^0,\Xi^0)$ but this implies $(\Lambda^0,\Xi^0) = (\Lambda,\Xi)$, a contradiction.

Finally we prove (c) hence we assume that $\bG$ supports a unipotently supported cuspidal character sheaf; this implies that $n$ is even. Let us take $i=0$ and $v = u_0$ in \cref{eq:special-character-sum} then we have
\begin{equation}\label{eq:cuspidal-root-of-unity}
2^{\ell}\rho_0(u_0) = \zeta_0q^{\dim\mathfrak{B}_{u_0}^{\bG} + \frac{1}{2}m(\Lambda^0\oplus\Xi^0)} + \sum_{(M,N) \in \mathcal{J}\setminus\{(\Lambda^0,\Xi^0)\}} q^{\dim\mathfrak{B}_{u_0}^{\bG} + \frac{1}{2}m(M\oplus N)}.
\end{equation}
Again using \cref{lem:princpal-series-Z} we have $\rho_0(u_0) \in \mathbb{Z}$, which implies $\zeta_0 \in \mathbb{Z}$ and furthermore $\zeta_0 \in \{\pm1\}$ as it is of absolute value 1. As above we may assume that $t \geqslant 1$ as $n > 0$, hence we have $2^{\ell} \geqslant 4$. With this we have 4 divides the right hand side of \cref{eq:cuspidal-root-of-unity} but as $q^{\dim\mathfrak{B}_{u_0}^{\bG}}$ is an odd integer we have \cref{eq:cuspidal-root-of-unity} gives us
\begin{equation*}
\zeta_0 + (2^{\ell}-1) \equiv 0 \pmod{4},
\end{equation*}
but this implies $\zeta_0 = 1$ as required.
\end{proof}

\appendix
%
\section{\texorpdfstring{Induction in Weyl Groups of Type $\D_n$}{Induction in Weyl Groups of Type D}}
\begin{pa}
In this appendix we formulate a lemma concerning the induction of characters in Weyl groups of type $\D_n$ which is an analogue of \cite[Lemma 6.1.3]{geck-pfeiffer:2000:characters-of-finite-coxeter-groups}. Note that we will not perform the difficult task of understanding explicitly the multiplicities of degenerate characters as we will not need this here. Note that in the following we freely use the notation and conventions of \cref{sec:weyl-groups}.
\end{pa}

\begin{lem}\label{lem:typeD-GP-Lemm6.1.3}
Assume $a,b \in \mathbb{N}_0$ are such that $a+b = n$ and $E = \ssymb{A_1}{A_2} \boxtimes \ssymb{B_1}{B_2} \in \Irr(W_a'\times W_b')$ is a character, (we set $A = \ssymb{A_1}{A_2}$ and $B = \ssymb{B_1}{B_2}$). Let $\ssymb{X_1}{X_2} \in \overline{\mathbb{V}}_n^0$ be a symbol such that $\rk(X_k) = \rk(A_{i_k}) + \rk(B_{j_k})$, ($k \in \{1,2\}$), for some indexing sets $\{i_1,i_2\} = \{j_1,j_2\} = \{1,2\}$ then we define
\begin{equation*}
a_{AB}^{X_1X_2} = \begin{cases}
c_{A_{i_1}B_{j_1}}^{X_1}c_{A_{i_2}B_{j_2}}^{X_2} + \delta_Ac_{A_{i_2}B_{j_1}}^{X_1}c_{A_{i_1}B_{j_2}}^{X_2} + \delta_Bc_{A_{i_1}B_{j_2}}^{X_1}c_{A_{i_2}B_{j_1}}^{X_2} + \delta_{A\odot B}c_{A_{i_2}B_{j_2}}^{X_1}c_{A_{i_1}B_{j_1}}^{X_2} &\text{if }A_1\neq A_2\text{ and }B_1\neq B_2\\
c_{A_{i_1}B_{j_1}}^{X_1}c_{A_{i_2}B_{j_2}}^{X_2} + \delta_Bc_{A_{i_1}B_{j_2}}^{X_1}c_{A_{i_2}B_{j_1}}^{X_2} &\text{if }A_1=A_2\text{ and }B_1\neq B_2\\
c_{A_{i_1}B_{j_1}}^{X_1}c_{A_{i_2}B_{j_2}}^{X_2} + \delta_Ac_{A_{i_2}B_{j_1}}^{X_1}c_{A_{i_1}B_{j_2}}^{X_2} &\text{if }A_1\neq A_2\text{ and }B_1=B_2\\
c_{A_{i_1}B_{j_1}}^{X_1}c_{A_{i_2}B_{j_2}}^{X_2} &\text{if }A_1=A_2\text{ and }B_1=B_2.
\end{cases}
\end{equation*}
If $Z = \ssymb{Z_1}{Z_2} \in \overline{\mathbb{V}}_n^0$ is a symbol then $\delta_Z$ is 1 if $\rk(Z_1) = \rk(Z_2)$ and 0 otherwise, (note that $\rk(A\odot B) = \rk(A)+\rk(B)$). The induced character is then
\begin{equation*}
\Ind_{W_a'\times W_b'}^{W_n'}(E) = \sum_{X_1\neq X_2} a_{AB}^{X_1X_2}\symb{X_1}{X_2} + \sum_{X_1 = X_2} a_{AB}^{X_1,\pm}\symb{X_1}{X_2}_{\pm}
\end{equation*}
where the sum is over all symbols $\ssymb{X_1}{X_2} \in \overline{\mathbb{V}}_n^0$ such that there exist indexing sets $\{i_1,i_2\} = \{j_1,j_2\} = \{1,2\}$ satisfying $\rk(X_k) = \rk(A_{i_k}) + \rk(B_{j_k})$ for $k \in \{1,2\}$. Here the coefficients $a^{X_1,\pm}_{AB}$ satisfy $a^{X_1,+}_{AB} + a^{X_1,-}_{AB} = a^{X_1X_2}_{AB}$ and are such that $a^{X_1,+}_{AB} = a^{X_1,-}_{AB} = \frac{1}{2}a^{X_1X_1}_{AB}$ whenever $A$ or $B$ is non-degenerate.
\end{lem}

\begin{proof}
Assume $A$, $B \in \overline{\mathbb{V}}_n^0$ are non-degenerate then
\begin{align*}
\Ind_{W_a'\times W_b'}^{W_n}(E) &= \Ind_{W_a \times W_b}^{W_n} \left( \symb{A_1}{A_2} \boxtimes \symb{B_1}{B_2} + \symb{A_2}{A_1} \boxtimes \symb{B_1}{B_2} + \symb{A_1}{A_2} \boxtimes \symb{B_2}{B_1} + \symb{A_2}{A_1} \boxtimes \symb{B_2}{B_1} \right)\\
&= \sum c_{A_1B_1}^{U_1}c_{A_2B_2}^{U_2}\symb{U_1}{U_2} + \sum c_{A_2B_1}^{V_1}c_{A_1B_2}^{V_2}\symb{V_1}{V_2} + \sum c_{A_1B_2}^{Y_1}c_{A_2B_1}^{Y_2}\symb{Y_1}{Y_2} + \sum c_{A_2B_2}^{Z_1}c_{A_1B_1}^{Z_2}\symb{Z_1}{Z_2}
\end{align*}
where the sums are as in \cite[Lemma 6.1.3]{geck-pfeiffer:2000:characters-of-finite-coxeter-groups}. The cases where either $A$ or $B$ are degenerate are easily adapted from the case where neither are degenerate. Assume $X = \ssymb{X_1}{X_2} \in \overline{\mathbb{V}}_n^0$ is non-degenerate then fixing the ordering of the rows determines an element $\widetilde{X} \in \mathbb{V}_n^0$ and by Frobenius reciprocity we have
\begin{equation*}
\langle \Ind_{W_a'\times W_b'}^{W_n'}(E),X\rangle_{W_n'} = \langle \Ind_{W_a'\times W_b'}^{W_n'}(E),\Res_{W_n'}^{W_n}(\widetilde{X})\rangle_{W_n'} = \langle \Ind_{W_a'\times W_b'}^{W_n}(E),\widetilde{X}\rangle_{W_n}.
\end{equation*}
Combining the above two statements one easily obtains the multiplicity of any non-degenerate character in $\Ind_{W_a' \times W_b'}^{W_n'}(E)$. To deal with the case when $X$ is degenerate we apply the Mackey formula to obtain
\begin{equation*}
\Ind_{W_a'\times W_b'}^{W_n'}(E) + \Ind_{W_a'\times W_b'}^{W_n'}({}^xE) = (\Res_{W_n'}^{W_n}\circ\Ind_{W_a'\times W_b'}^{W_n})(E)
\end{equation*}
where $x \in \{s_n,t_0\}$. Writing the induction $\Ind_{W_a'\times W_b'}^{W_n}$ as $\Ind_{W_a\times W_b}^{W_n}\circ\Ind_{W_a'\times W_b'}^{W_a\times W_b}$ it is easily seen that the right hand side of this equality becomes
\begin{equation*}
\sum_{\mathbb{V}_n^0} a_{AB}^{X_1X_2}\Res_{W_n'}^{W_n}\symb{X_1}{X_2} = 2\sum_{X_1 \neq X_2} a_{AB}^{X_1X_2}\symb{X_1}{X_2} + \sum_{X_1 = X_2} a_{AB}^{X_1X_2}\left(\symb{X_1}{X_2}_+ + \symb{X_1}{X_2}_-\right).
\end{equation*}
Note that on the right hand side the first sum is over all non-degenerate symbols in $\overline{\mathbb{V}}_n^0$ and the second sum is over all degenerate symbols. As we already know the multiplicities of non-degenerate characters we see that the result follows by noticing that $E = {}^{s_n}E$ if and only if $B$ is non-degenerate and $E = {}^{t_0}E$ if and only if $A$ is non-degenerate.
\end{proof}

\begin{pa}\label{pa:the-subgroup-H}
We now end this appendix by considering the induction of characters from the subgroup $H = (W_a'\times W_b')\langle s_nt_0\rangle \leqslant W_a \times W_b \leqslant W_n$. Assume that $\ssymb{A_1}{A_2}\boxtimes\ssymb{B_1}{B_2} \in \Irr(W_a\times W_b)$ is an irreducible character such that $A_1\neq A_2$ and $B_1 \neq B_2$ then it is clear that we have
\begin{align*}
\Res^{W_a\times W_b}_H\symb{A_1}{A_2}\boxtimes \symb{B_1}{B_2} &= \Res^{W_a\times W_b}_H\symb{A_2}{A_1}\boxtimes \symb{B_2}{B_1},\\
\Res^{W_a\times W_b}_H\symb{A_2}{A_1}\boxtimes \symb{B_1}{B_2} &= \Res^{W_a\times W_b}_H\symb{A_1}{A_2}\boxtimes \symb{B_2}{B_1}
\end{align*}
and these characters are irreducible. We call the characters of $H$ obtained in this way \emph{non-degenerate} and we label them by pairs of symbols $\ssymb{A_1}{A_2}\boxtimes\ssymb{B_1}{B_2} \in \mathbb{V}_a^0 \times \mathbb{V}_b^0$ such that the rows are simultaneously permutable. With this in mind we have the following result.
\end{pa}

\begin{lem}\label{lem:subgroup-H-W_n'}
Let $E = \ssymb{A_1}{A_2} \boxtimes \ssymb{B_1}{B_2} \in \Irr(H)$ be non-degenerate then we have
\begin{equation*}
\Ind_H^{W_n'}(E) = \sum_{X_1\neq X_2} e_{AB}^{X_1X_2}\symb{X_1}{X_2} + \sum_{X_1= X_2} \frac{1}{2}e_{AB}^{X_1X_2}\symb{X_1}{X_2}_{\pm}
\end{equation*}
where the sum runs over all symbols $\ssymb{X_1}{X_2} \in \overline{\mathbb{V}}_n^0$ such that there exists an indexing set $\{j_1,j_2\} = \{1,2\}$ satisfying $\rk(X_i) = \rk(A_{j_i}) + \rk(B_{j_i})$ for $i \in \{1,2\}$. Here we have for any symbol $\ssymb{X_1}{X_2} \in \overline{\mathbb{V}}_n^0$ that
\begin{equation*}
e_{AB}^{X_1X_2} = c_{A_{j_1}A_{j_2}}^{X_1}c_{B_{j_1}B_{j_2}}^{X_2} + \delta_{A\odot B}c_{A_{j_2}A_{j_1}}^{X_1}c_{B_{j_2}B_{j_1}}^{X_2}.
\end{equation*}
\end{lem}

\begin{proof}
This is easily adapted from the proof of the previous lemma.
\end{proof}

\setstretch{0.96}
\renewcommand*{\bibfont}{\small}
\printbibliography
\end{document}